\documentclass[a4paper]{article}
\usepackage{geometry}
\usepackage{amsthm,amsmath,amsfonts,amssymb}

\usepackage[colorlinks,citecolor=blue,urlcolor=blue]{hyperref}
\usepackage{graphicx, color}
\usepackage{enumerate, enumitem}
\usepackage{dsfont}
\usepackage{mathtools}
\usepackage{lineno}
\usepackage{hyperref}
\allowdisplaybreaks

\theoremstyle{plain}
\newtheorem{thm}{Theorem}[section]
\newtheorem{lem}[thm]{Lemma}
\newtheorem{cor}[thm]{Corollary}

\newtheorem{prop}[thm]{Proposition}

\theoremstyle{remark}
\newtheorem{defi}[thm]{Definition}

\newtheorem{exam}[thm]{Example}
\newtheorem{remk}[thm]{Remark}

\newcommand{\ind}{\mathds{1}}

\title{On a multivariate copula-based dependence measure and its estimation}
\author{Florian Griessenberger\footnote{Department of Mathematics, University of Salzburg, 5020 Salzburg, Austria. florian.griessenberger@plus.ac.at, wolfgang.trutschnig@plus.ac.at}, Robert R. Junker\footnote{Department of Biosciences, University of Salzburg, 5020 Salzburg, Austria} \footnote{Evolutionary Ecology of Plants, Department of Biology, University of Marburg, 35043 Marburg, Germany.}, Wolfgang Trutschnig$^*$}
\date{}

\begin{document}

\maketitle

\begin{abstract}
	Working with so-called linkages allows to define a copula-based, $[0,1]$-valued multivariate dependence measure 
	$\zeta^1(\boldsymbol{X},Y)$ quantifying the
	scale-invariant extent of dependence of a random variable $Y$ on a $d$-dimen\-sional random vector $\boldsymbol{X}=(X_1,\ldots,X_d)$ 
	which exhibits various good and natural properties. In particular, $\zeta^1(\boldsymbol{X},Y)=0$ if and only if
	$\boldsymbol{X}$ and $Y$ are independent, $\zeta^1(\boldsymbol{X},Y)$ is maximal 
	exclusively if $Y$ is a function of $\boldsymbol{X}$, and 
	ignoring one or several coordinates of $\boldsymbol{X}$
	can not increase the resulting dependence value.  
	After introducing and analyzing the metric $D_1$ underlying the construction of the dependence measure and 
	deriving examples showing how much information can be lost by only considering all pairwise dependence values 
	$\zeta^1(X_1,Y),\ldots,\zeta^1(X_d,Y)$ we derive a so-called checkerboard estimator for 
	$\zeta^1(\boldsymbol{X},Y)$ and show
	that it is strongly consistent in full generality, i.e., without any smoothness restrictions on the underlying copula.
	Some simulations illustrating the small sample performance of the estimator complement the established 
	theoretical results.  
\end{abstract}


\section{Introduction}

Consider a random vector $(X_1,X_2,\ldots,X_d,Y)$ on a probability space $(\Omega,\mathcal{A},\mathbb{P})$, write
$\boldsymbol{X}=(X_1,\ldots,X_d)$, and suppose that $(\boldsymbol{X}_1,Y_1),\ldots,(\boldsymbol{X}_n,Y_n)$ is
a sample of $(\boldsymbol{X},Y)$. Two of the major questions (think of problems like feature selection or regression) 
in statistics are (i) how much information $\boldsymbol{X}$ provides about $Y$ (or, equivalently, how 
dependent $Y$ is on $\boldsymbol{X}$) and (ii) how this dependence can be estimated in terms of 
$(\boldsymbol{X}_1,Y_1),\ldots,(\boldsymbol{X}_n,Y_n)$. 
Seemingly natural requirements for a measure $\delta$ quantifying the extent of dependence of $Y$ on $\boldsymbol{X}$ are: 
\begin{enumerate}
	\item[(N)] $\delta(\boldsymbol{X},Y) \in [0,1]$ (normalization).
	\item[(I)] $\delta(\boldsymbol{X},Y) = 0$ if and only if $Y$ and $\boldsymbol{X}$ are independent (independence). 
	\item[(C)] $\delta(\boldsymbol{X},Y) = 1$ if and only if $Y$ is a function of $\boldsymbol{X}$ (complete dependence).
	\item[(S)] $\delta$ is scale-invariant.
	\item[(IG)] $\delta$ fulfills the information gain inequality
	\begin{align*}
		\delta(X_1,Y) \leq \delta((X_1,X_2), Y) \leq  \cdots   \leq 
		\delta((X_1,\ldots, X_d),Y).
	\end{align*} 
\end{enumerate}  
Not surprisingly, various `dependence measures' (fulfilling the afore-mentioned five properties to different extents) 
and their estimators have been introduced and studied in the past 15 years, see for instance, \cite{azadkia2019,ding2016,chatterjee2020,dette2013, fuchs2022, junker, reshef2011,reshef2018, szekely2007, szekely2009, trutschnig2011} and the references therein. 
The majority of these works focuses on the bivariate setting, including the measures 
$\xi$ studied in \cite{chatterjee2020} (implemented in the R-package \textit{xicor}), $\zeta_1$ studied in 
\cite{junker, trutschnig2011} (R-package \textit{qad}) or \textit{MIC} studied in 
\cite{reshef2011,reshef2018} (R-package \textit{minerva}). 
Lea\-ving the $2$-dimensional setting and aiming at general multivariate measures  
the number of works decreases significantly, for a brief survey we refer to 
\cite{azadkia2019, Joe1989, szekely2009} and the references therein. 
In \cite{szekely2007,szekely2009}, for instance, the authors introduce the by now well known multivariate, characteristic function based dependence measure \textit{distance correlation} which generalizes the idea of correlation and is able 
to characterize independence of two random vectors $\boldsymbol{X}$ and $\boldsymbol{Y}$ in arbitrary dimension. 
Distance correlation, however, is not capable of detecting complete dependence in full generality, i.e., it does not attain the value $1$ exclusively for the case that $\boldsymbol{Y}$ is a function of $\boldsymbol{X}$ 
(and hence violates the afore-mentioned property (C)). 

{The recently introduced \emph{measure of conditional dependence} $T$ (see \cite{azadkia2019}), which is a multivariate extension of the highly performant bivariate measure $\xi$ (see \cite{chatterjee2020}) resolves this problem. Its remarkable estimator $T_n$, implemented in the R-package \textit{FOCI}, performs generally very well and has attracted a lot of attention in the statistical community, which is reflected by a large number of follow-up works.
	In \cite{deb2020, huang2020}, for instance, extensions of the (conditional) dependence measure to topological spaces and multivariate output $\boldsymbol{Y}$ were developed and studied. Furthermore, the authors in \cite{shi2021} derived 
	the (previously conjectured) asymptotic normality of $T_n$ under independence (see \cite{azadkia2019}) and some mild conditions, a result which, in turn, allows to apply $T_n$ in the context of testing for conditional independence. 
	Following \cite{shi2021} $T_n$ does not have optimal power in testing for 
	independence, which agrees with the bivariate setting as studied, e.g., in \cite{lin2021, shi_biometrika}. 
	For an excellent overview concerning $T_n$ and related notions/generalizations we refer to \cite{han2021} and the references therein.}

In the current paper we focus on an extension of the bivariate, copula-based dependence measure $\zeta_1$ and its 
so-called checkerboard estimator to multivariate $\boldsymbol{X}$. 
In dimension two, $\zeta_1$ and $\xi$ are closely related - in fact, $\xi$ is the $L_2$-version of the older 
$L_1$-based measure $\zeta_1$. Moreover, the estimators of the population 
values - although being based on entirely different ideas - perform comparably well (see \cite{junker}) 
in most situations, with the $\zeta_1$ estimator exhibiting a slightly higher power in detecting deviation from independence {as well as attaining only values within $[0,1]$, whereas the estimator of $\xi$ may
	also be negative (a property potentially hard to interpret for applicants outside the statistical community)}. 
As main result we will show that the introduced extension of $\zeta_1$ yields a dependence 
measure fulfilling all five properties (N), (I), (C), (S), (IG), and derive a strongly consistent estimator for it. 

Note that the seemingly natural extension of $\zeta_1$ interpreted as distance to independence 
does not yield a reasonable measure: In fact, for random variables $X,Y$ with continuous bivariate distribution 
function $H$, marginals $F,G$ and 
(unique) copula $A$ according to \cite{trutschnig2011} the measure $\zeta_1$ is defined by 
$$
\zeta_1(X,Y)=\zeta_1(A)=3 D_1(A,\Pi).
$$
Thereby $\Pi$ denotes the bivariate product copula and $D_1$ is the metric 
on the family $\mathcal{C}^2$ of all bivariate copulas defined by 
\begin{align*}
	D_1(A,B) := \int_{[0,1]} \int_{[0,1]} \left|K_A(u, [0,v]) - K_B(u,[0,v])\right| \, d\lambda(u) \, d\lambda(v),
\end{align*}
where $K_A(\cdot,\cdot), K_B(\cdot,\cdot)$ denote the Markov kernels (regular conditional distributions) 
of $A,B \in \mathcal{C}^2$, respectively. 
Proceeding analogously and considering  
\begin{align*}
	D_1(A,B):=\int_{[0,1]} \int_{[0,1]^d} |K_A(\boldsymbol{u}, [0,v]) - K_B(\boldsymbol{u},[0,v])| \, d\lambda^d(\boldsymbol{u}) 
	\, d\lambda(v),
\end{align*}
does not yield a metric on the family $\mathcal{C}^\rho$ of all $\rho$-dimensional copulas 
since, firstly, the (conditional) probability measures 
$K_A(\boldsymbol{u},\cdot)$ and $K_B(\boldsymbol{u},\cdot)$ are only unique $A^{1,\ldots,d}$ and $B^{1,\ldots,d}$-almost everywhere, whereby $A^{1,\ldots,d}$ and $B^{1,\ldots,d}$ denote the $d$-dimensional marginal copulas of $A$ and $B$, respectively. And secondly, none of the two marginals needs to be uniformly distributed on $[0,1]^d$. 
An alternative approach based on $\zeta_1$ of all pairs was studied in \cite{sanchez2015}, the obtained notion, 
however, did not exhibit the desired properties either since only considering all bivariate marginals may go
hand in hand with losing a lot of relevant information. 

In the sequel we will show how working with the so-called linkage operator (see \cite{Li1996}) which allows to transform 
general random vectors $\boldsymbol{X}$ to random vectors uniformly distributed on $[0,1]^d$ provides a way 
to overcome the afore-mentioned problems and hence opens the door to extending $\zeta_1$ to a measure
fulfilling all five desired properties. Notice that a similar, linkage-based approach was studied \cite{bonmee2016},
the focus there, however, was on the technically less demanding case of 
\textit{absolutely continuous} random vectors and, moreover, no estimator was provided. 
Also note that focusing only on univariate variables $Y$ entails no restriction since 
the dependence of a random vector $(Y_1,\ldots,Y_m)=\boldsymbol{Y}$ on $\boldsymbol{X}$ may be quantified by simply
considering the dependence scores of $(\boldsymbol{X},Y_1),\ldots, (\boldsymbol{X},Y_m)$. 

The rest of this paper is organized as follows: Section 2 gathers preliminaries and notations that will be used 
in the sequel. In Section 3 we recall the definition of linkages, express the connection between copulas and their 
linkage in terms of Markov kernels and prove a characterization of copulas fulfilling the so-called conditional independence property. In Section 4 we then introduce the metric $D_1$ on the space $\mathcal{C}_{\Pi_d}^\rho$ 
of linkages, discuss various (topological) properties (extending those in \cite{trutschnig2011}) and 
show that the family of so-called checkerboard copulas 
(which are key for the construction of our estimator) is dense in $(\mathcal{C}_{\Pi_d}^\rho, D_1)$. 
Based on $D_1$ we then construct the non-parametric measure of dependence $\zeta^1$ in Section 5 and 
show that the obtained measure fulfills all five properties mentioned before. 
Finally, we construct a checkerboard estimator for which we prove strong consistency in full generality 
(Section 6). Various examples and graphics illustrate both the obtained results and the ideas underlying the proofs. 
A simulation study illustrating the performance of our estimator and some technical proofs can be found
in the Appendix \ref{app:proofs} and \ref{app:simstudy}.

\section{Notation and preliminaries}
Throughout the paper $\rho \in \mathbb{N}$ will denote the dimension and $d$ is defined by $d:=\rho-1$. Bold symbols will be used to denote vectors, e.g., $\boldsymbol{x}=(x_1,x_2,\ldots, x_d) \in \mathbb{R}^d$ and we will also write $(\boldsymbol{x}, y)$ for $(x_1, \ldots, x_d, y) \in \mathbb{R}^\rho$. The $\rho$-dimensional Lebesgue-measure will be denoted by $\lambda^\rho$, in case of $\rho=1$ we will also simply write $\lambda$ and $\mathbb{I}$ will denote the unit interval $[0,1]$. 
Moreover, $\mathcal{C}^\rho$ denotes the family of all $\rho$-dimensional copulas, in the two-dimensional setting we will also write $\mathcal{C}$, $d_\infty$ will denote the uniform metric on $\mathcal{C}^\rho$, i.e., 
$$d_\infty(A,B):= \max_{\boldsymbol{x} \in \mathbb{I}^\rho} \left| A(\boldsymbol{x}) - B(\boldsymbol{x}) \right|.$$
It is well known that $(\mathcal{C}^\rho, d_\infty)$ is a compact metric space (see, for instance, \cite{durante, nelsen}). For every $C \in \mathcal{C}^\rho$ the corresponding $\rho$-stochastic measure will be denoted by $\mu_C$, i.e., $\mu_C([0,\boldsymbol{x}]) = C(\boldsymbol{x})$ for all $\boldsymbol{x} \in \mathbb{I}^\rho$, whereby $[0,\boldsymbol{x}]:=\bigtimes_{i=1}^\rho [0,x_i] =  [0,x_1] \times \cdots \times [0,x_\rho]$. $\Pi_\rho$ and $M_\rho$ denote the $\rho$-dimensional product and the $\rho$-dimensional minimum copula, i.e., $\Pi(x_1,\ldots, x_\rho) = \prod_{j=1}^\rho x_j$ as well as $M_\rho(x_1, \ldots, x_\rho) = \min\{x_1, \ldots, x_\rho\}$, for $\rho=2$ we will also simply write $\Pi$ and $M$.
If we consider the marginal co\-pula with respect to the first $k$ variables, we will simply write $A^k$, i.e., we have $A^k(x_1,\ldots, x_k) = A(x_1,\ldots, x_k, 1,\ldots, 1)$, furthermore $C^{ij}$ denotes the $i$-$j$-marginal 
for $i \neq j$ and $i,j \in \{1,\ldots, \rho\}$, i.e., 
$$C^{ij}(x_i, x_j) = C(1, \ldots, 1, x_i, 1, \ldots, 1, x_j, 1, \ldots, 1).$$ 
For more background on copulas and $\rho$-stochastic probability measures we refer to \cite{durante, klenke, nelsen} and the references therein. 

In what follows, Markov kernels will play a prominent role. For every metric space $(\Omega,d)$ the Borel $\sigma$-field on $\Omega$ will be denoted by $\mathcal{B}(\Omega)$. Let $(\Omega_1,\mathcal{A}_1)$ and $(\Omega_2, \mathcal{A}_2)$ be measurable spaces. A map $K: \Omega_1 \times \mathcal{A}_2 \to [0,1]$ is called Markov kernel from $(\Omega_1,\mathcal{A}_1)$ to $(\Omega_2,\mathcal{A}_2)$ if for every fixed $A_2 \in \mathcal{A}_2$ the map $\omega_1 \mapsto K(\omega_1, A_2)$ is $\mathcal{A}_1$-$\mathcal{B}(\mathbb{R})$-measurable and for every fixed $\omega_1 \in \Omega_1$ the map $A_2 \mapsto K(\omega_1, A_2)$ is a probability measure on $\mathcal{A}_2$.  
Given a $j$-dimensional random vector $\boldsymbol{Y}$ and a $k$-dimensional random vector $\boldsymbol{X}$ on a probability space $(\Omega, \mathcal{A}, \mathbb{P})$, i.e., measurable mappings  $\boldsymbol{Y}:(\Omega, \mathcal{A}, \mathbb{P}) \to (\mathbb{R}^j,\mathcal{B}(\mathbb{R}^j))$ and $\boldsymbol{X}:(\Omega, \mathcal{A}, \mathbb{P}) \to (\mathbb{R}^k,\mathcal{B}(\mathbb{R}^k))$, we say that a Markov kernel $K$ is a regular conditional distribution of $\boldsymbol{Y}$ given $\boldsymbol{X}$ if 
\begin{align*}
	K(\boldsymbol{X}(\omega), F) = \mathbb{E}(\ind_F \circ \boldsymbol{Y} | \boldsymbol{X})(\omega)
\end{align*}
holds $\mathbb{P}$-almost surely for every $F \in \mathcal{B}(\mathbb{R}^{j})$. It is well-known that for each random vector $(\boldsymbol{X}, \boldsymbol{Y})$ a regular conditional distribution $K(\cdot, \cdot)$ of 
$\boldsymbol{Y}$ given $\boldsymbol{X}$ always exists and is unique for $\mathbb{P}^{\boldsymbol{X}}$-a.e. $\boldsymbol{x} \in \mathbb{R}^k$, whereby $\mathbb{P}^{\boldsymbol{X}}$ denotes the push-forward of $\mathbb{P}$ under $\boldsymbol{X}$. It is well known that $K(\cdot, \cdot)$ only depends on $\mathbb{P}^{(\boldsymbol{X},\boldsymbol{Y})}$, hence, if $(\boldsymbol{X},\boldsymbol{Y})$ has distribution function $H$ (in which case we will also write $(\boldsymbol{X},\boldsymbol{Y}) \sim H$) we will let $K_H(\cdot, \cdot)$ denote (a version of) the regular conditional distribution of $\boldsymbol{Y}$ given $\boldsymbol{X}$ and simply refer to it as Markov kernel of $H$. Furthermore, we will write $\mathcal{F}^\rho$ for the family of all $\rho$-dimensional distribution functions. \\
If $C \in \mathcal{C}^\rho$ is a copula, then we will consider the Markov kernel of $C$ (with respect to the first $k$-coordinates) as mapping $K_{C}:\mathbb{I}^k \times \mathcal{B}(\mathbb{I}^{\rho - k}) \to \mathbb{I}$. 
Defining the $\boldsymbol{x}$-section of a set $G \in \mathcal{B}(\mathbb{I}^\rho)$ as $G_{\boldsymbol{x}}:=\{\boldsymbol{y} \in \mathbb{I}^{\rho - k}: (\boldsymbol{x},\boldsymbol{y}) \in G\}$ the so-called disintegration theorem yields 
\begin{align}
	\mu_C(G) = \int_{\mathbb{I}^k} K_{C}(\boldsymbol{x},G_{\boldsymbol{x}}) \, d\mu_{C^{k}}(\boldsymbol{x}),
\end{align}
so for the particular case $k = d:=\rho-1$ we have
\begin{align}
	\mu_C(\mathbb{I}^d \times F) = \int_{\mathbb{I}^d} K_{C}(\boldsymbol{x},F) \, d\mu_{C^{d}}(\boldsymbol{x}) = \lambda(F).
\end{align}
For more background on conditional expectations, regular conditional distributions and general disintegration we refer to \cite{kallenberg, klenke}. \\
In the sequel $\mathcal{U}(0,1)$ will denote the uniform distribution on $\mathbb{I}$ and if 
$X \sim F$ we write $F^-$ for the pseudo-inverse of the distribution function $F$, i.e., 
$F^-(y) := \inf\{x \in \mathbb{R}: F(x) \geq y\}$.

\section{Copulas and linkages}\label{sec:linkages}
Suppose that $X_1,\ldots, X_d$ are random variables on $(\Omega, \mathcal{A}, \mathbb{P})$ with cumulative distribution functions $F_1, \ldots, F_d$ and let $H$ denote the cumulative distribution function of the random vector $\boldsymbol{X}=(X_1, X_2, \ldots, X_d)$.
According to Sklar's theorem (see, for instance, \cite{durante, nelsen}) there exists a unique subcopula $S:cl(Range(F_1)) \times \cdots \times cl(Range(F_d)) \to [0,1]$ such that for all $\boldsymbol{x} \in \mathbb{R}^d$
\begin{align}\label{sklar}
	H(x_1,\ldots, x_d) = S(F_1(x_1), \ldots, F_d(x_d)).
\end{align}
If $F_1,\ldots, F_d$ are continuous distribution functions, then $S$ is a copula and unique, otherwise $S$ can be extended in uncountably many ways to a copula $C \in \mathcal{C}^d$ (see \cite{durante}). To ensure uniqueness of $C$ in the general setting, we agree on working with the multi-linear interpolation of $S$ (see, for instance, \cite{durante,genest2014}). 

In what follows we will work with the so-called modified distribution function (see \cite[2.3.4.]{durante} or \cite{rueschendorf2009}). Let $Z$ be a random variable on $(\Omega,\mathcal{A}, \mathbb{P})$ with distribution function $F$ and $r \in [0,1]$ be fixed. Then $F^r:\mathbb{R} \to [0,1]$, defined by 
\begin{align*}
	F^r(z) := F(z-) + r\left[ F(z)-F(z-)\right] = \mathbb{P}(Z < z) + r \cdot \mathbb{P}(Z = z) 
\end{align*}
is called modified distribution function of $Z$. The following well known lemma which considers random $r$ gathers the 
main properties of the generalized distributional transform (a.k.a. generalized PIT) 
$F^R(Z)$ of $Z$ (for a proof see \cite{rueschendorf2009}).
\begin{lem}[Generalized PIT]\label{lem:rpit}
	Suppose that $Z$ has distribution function $F$, that $R$ is uniformly distributed and that $Z$ and $R$ 
	are independent. Furthermore, set $U:=F^R(Z)$. Then the random variable $U$ is uniformly distributed on $[0,1]$ and 
	the identity $F^-(U) = Z$ holds with probability $1$.
\end{lem} 

To simplify notation we will write 
$$F_{d|1\ldots d-1}(x_d|x_1,\ldots, x_{d-1}) := K_{H}(x_1,\ldots, x_{d-1}, [0,x_d])$$ 
as well as
$$F^r_{d|1\ldots d-1}(x_d|x_1,\ldots, x_{d-1}) := K_{H}(x_1,\ldots, x_{d-1}, [0,x_d)) + r K_{H}(x_1,\ldots, x_{d-1}, \{x_d\}),$$
for the conditional and the modified conditional distribution functions, respectively, 
where $\boldsymbol{X} = (X_1,\ldots, X_d) \sim H$ and $r \in \mathbb{I}$. 
If $R \sim \mathcal{U}(0,1)$ and $\boldsymbol{X}$ are independent, then for all vectors $(x_1,\ldots, x_{d-1}) \in \mathbb{R}^{d-1}$ the random variable $F^R_{d|1\ldots, d-1}(X_d|x_1,\ldots, x_{d-1})$ (i.e., the generalized probability integral transform
$F^R_{d|1\ldots d-1}(Y| x_1,\ldots, x_d)$ of the random variable $Y \sim K_H(x_1,\ldots, x_{d-1}, \cdot)$),
is uniformly distributed on $\mathbb{I}$.  
Suppose that $\boldsymbol{r}=(r_1,r_2,\ldots,r_d) \in \mathbb{I}^d$, then following \cite{rosenblatt1952, rueschendorf1993} the transformations $\Phi^{\boldsymbol{r}}:\mathbb{R}^d \to \mathbb{I}^d$ and $\Psi:\mathbb{I}^d \to \mathbb{R}^d$ are defined by 
\begin{align*}
	\Phi^{\boldsymbol{r}}(x_1,x_2,\ldots, x_d):=(F^{r_1}_1(x_1),F^{r_2}_{2|1}(x_2|x_1), \ldots, F^{r_d}_{d|1\ldots d-1}(x_d|x_1,\ldots, x_{d-1}))
\end{align*}
for all $\boldsymbol{x} \in \mathbb{R}^d$ and 
\begin{align*}
	\Psi(u_1,u_2,\ldots, u_d):= (z_1, z_2, \ldots, z_d),
\end{align*}
whereby $z_1 = F_1^-(u_1)$ and, inductively
$$z_i = F_{i|1\ldots i-1}^-(u_i|z_1,z_2,\ldots, z_{i-1})$$
for $i = 2,3,\ldots, d$. 
Both transformations are well known in statistics: $\Phi^{\boldsymbol{r}}$ is called Rosenblatt transformation and was introduced by Rosenblatt in 1952 \cite{rosenblatt1952} for absolutely continuous random variables and generalized in \cite{rueschendorf1981, rueschendorf1993} to the general setting. The \textquoteleft inverse' transformation $\Psi$ is well known in the context of generating samples. The following lemma gathers the main properties of the mappings $\Phi^{\boldsymbol{R}}$ and $\Psi$ (see \cite{Li1996, rosenblatt1952} for the absolutely continuous setting as well as \cite[Proposition 2 and Theorem 3]{rueschendorf1993} and \cite[Section 3]{rueschendorf2009} for the general case). 

\begin{lem}\label{lem::psiphi}
	Suppose that $\boldsymbol{X} = (X_1,X_2,\ldots, X_d) \sim H$, $\boldsymbol{U}=(U_1,U_2,\ldots U_d) \sim \Pi_d \in \mathcal{C}^d$ and that $\boldsymbol{R}=(R_1,R_2,\ldots, R_d) \sim \Pi_d$ is independent of $\boldsymbol{X}$ and $\boldsymbol{U}$. Then the following statements hold:   
	\begin{enumerate}[parsep = 1ex]
		\item $\Phi^{\boldsymbol{R}}(X_1,X_2,\ldots, X_d) \sim \Pi_d$,
		\item $\Psi(U_1,U_2,\ldots, U_d) \sim H$.
	\end{enumerate}
	In particular, $\Phi^{\boldsymbol{R}} \circ \Psi (U_1,U_2,\ldots, U_d) \sim \Pi_d$, and
	\begin{enumerate}
		\item[3.] $\Psi \circ \Phi^{\boldsymbol{R}}(X_1,X_2,\ldots, X_d) = (X_1,X_2,\ldots, X_d)$ with 
		probability $1$.
	\end{enumerate}
\end{lem}

For the rest of the paper we agree on the following conventions. Suppose that $\boldsymbol{X} = (X_1,\ldots, X_d)$ and $Y$ are defined on the same probability space $(\Omega, \mathcal{A}, \mathbb{P})$, $(\boldsymbol{X},Y) \sim H$ with univariate marginals $F_1, \ldots, F_d$ and $G$, respectively. We say a random vector $\boldsymbol{X}$ is uniform if its distribution function restricted to $\mathbb{I}^d$ coincides with $\Pi_d$, i.e., $\mathbb{P}^{\boldsymbol{X}} = \lambda^d\rvert_{\mathbb{I}^d}$. Moreover, let $R_1,\ldots, R_d,R_\rho$ be i.i.d., uniformly distributed on $\mathbb{I}$ and write $\boldsymbol{R} = (R_1,\ldots, R_d)$. Following \cite{Li1996} the joint distribution function of the random vector $(\boldsymbol{U},V)$ defined by 
\begin{align}\label{eq:link}
	(\boldsymbol{U},V) := (\Phi^{\boldsymbol{R}}(\boldsymbol{X}), \Phi^{R_\rho}(Y))
\end{align}
is called \textit{linkage} of $(\boldsymbol{X},Y)$, or linkage of $H$. The family of all $\rho$-dimensional linkages, 
i.e., co\-pulas with $d$-dimensional uniform margin, will be denoted by 
$$
\mathcal{C}_{\Pi_d}^\rho:= \{ C \in \mathcal{C}^\rho: C(x_1,\ldots, x_d, 1) = \Pi_d(x_1,\ldots, x_d) \}.
$$
Interpreting the linkage construction as operator from $\mathcal{F}^\rho$ to $\mathcal{C}^\rho_{\Pi_d}$ we will also write $L:\mathcal{F}^\rho \to \mathcal{C}_{\Pi_d}^\rho$. \\In other words: if $(\boldsymbol{X}, Y) \sim H$ then $(\boldsymbol{U}, V) = (\Phi^{\boldsymbol{R}}(\boldsymbol{X}), \Phi^{R_\rho}(Y)) \sim L(H)$. 

In the sequel we will mainly focus on the linkage operation on the space of copulas, i.e., the mapping $L:\mathcal{C^\rho} \to \mathcal{C}_{\Pi_d}^\rho$ assigning each copula its linkage, implying that for every $\rho$-dimensional random vector 
$(\boldsymbol{X},Y) \sim C \in \mathcal{C}^\rho$ we have $(\boldsymbol{U},V) = (\Phi^{\boldsymbol{R}}(\boldsymbol{X}), Y) \sim L(C)$. Notice that the random vectors $(X_1,Y)$ and $(U_1,V)$ have the same distribution by construction. The following useful lemmata describe the connection between the copula and its linkage in terms of the corresponding Markov kernels and partial derivatives.    

\begin{lem}\label{lem::MKLinkage}
	Suppose that $(X_1,\ldots, X_d,Y) \sim A \in \mathcal{C}^\rho$ and let $(\boldsymbol{U}, V)$ be defined 
	according to Eq. \eqref{eq:link}. Then
	\begin{align*}
		K_{L(A)}(x_1, x_2 \ldots, x_d, [0,y]) := K_A\big(\Psi(x_1,x_2, \ldots, x_d), [0,y]\big)
	\end{align*}
	defines a Markov kernel of $L(A)$, i.e., the following identity holds for 
	all $(\boldsymbol{x},y)=(x_1,\ldots, x_d,y) \in \mathbb{I}^\rho$:
	\begin{align*}
		L(A)(x_1,x_2,\ldots, x_d,y) = \int_{[0,\boldsymbol{x}]} K_A\big(\Psi(u_1,u_2,\ldots, u_d), [0,y]\big) \, d\lambda^d(\boldsymbol{u})
	\end{align*}
	
\end{lem}

\begin{proof}
	Obviously, the mapping $F \mapsto K_A(\Psi(x_1,x_2,\ldots, x_d), F)$ is a probability measure for every fixed $\boldsymbol{x} \in \mathbb{I}^d$. Moreover, measurability of the mapping $(x_1,\ldots, x_d) \mapsto K_A\big(\Psi(x_1,x_2, \ldots, x_d), [0,y]\big)$ for every fixed $y \in [0,1]$ is a direct consequence of measurability of $\Psi$ and the fact that $K_A(\cdot, \cdot)$ is a Markov kernel. Since the family $$\mathcal{D}:=\{ E \subseteq \mathbb{I}: (x_1,\ldots, x_d) \mapsto K_A\big(\Psi(x_1,x_2, \ldots, x_d), E\big) \text{ is measureable} \}$$ forms a Dynkin system containing the family of all intervals of the form $[0,y]$ we conclude that $K_{L(A)}(\cdot, \cdot)$ is a Markov kernel and 
	it remains to show that $(\boldsymbol{x},y) \mapsto K_A(\Psi(\boldsymbol{x}), [0,y])$ is a regular conditional distribution of $L(A)$. 
	Fix $y \in [0,1]$, then using change of coordinates and applying Lemma \ref{lem::psiphi} repeatedly yields
	\begin{align*}
		\int_{[0,\boldsymbol{u}]} K_A\big(\Psi(x_1,x_2,&\ldots, x_d), [0,y]\big) \, d\lambda^d(\boldsymbol{x}) \\
		&= \int_{\mathbb{I}^d} \left(\ind_{[0,\boldsymbol{u}]}(\boldsymbol{\boldsymbol{x}})\right)  K_A\big(\Psi(\boldsymbol{\boldsymbol{x}}), [0,y]\big) \, d\mathbb{P}^{\Phi^{\boldsymbol{R}} \circ \Psi(\boldsymbol{U})}(\boldsymbol{x})\\
		&= \int_\Omega \ind_{[0,\boldsymbol{u}]}\left(\Phi^{\boldsymbol{R}} \circ \Psi(\boldsymbol{U})\right) K_A\big(\Psi \circ \Phi^{\boldsymbol{R}} \circ \Psi(\boldsymbol{U}), [0,y]\big) \, d\mathbb{P} \\
		&= \int_\Omega \ind_{[0,\boldsymbol{u}]}\left(\Phi^{\boldsymbol{R}} (\boldsymbol{X})\right) K_A\big(\Psi \circ \Phi^{\boldsymbol{R}} (\boldsymbol{X}), [0,y]\big) \, d\mathbb{P} \\
		&= \int_\Omega \ind_{[0,\boldsymbol{u}]}\left(\Phi^{\boldsymbol{R}} (\boldsymbol{X})\right) K_A\big(\boldsymbol{X}, [0,y]\big) \, d\mathbb{P} \\
		&= \mathbb{P}(\Phi^{\boldsymbol{R}}(\boldsymbol{X}) \leq \boldsymbol{u}, Y \leq y) 
		= L(A)(u_1,u_2,\ldots, u_d,y),
	\end{align*}
	which completes the proof.
\end{proof}

\begin{lem}\label{lem:derivative}
	Suppose that $(U_1, \ldots, U_d, V) \sim C \in \mathcal{C}_{\Pi_d}^\rho$ and let $K_C(\cdot, \cdot)$ denote 
	the Markov kernel of $C$. Then, for every $v \in \mathbb{I}$ we have
	\begin{align}\label{eq:derivative}
		K_C(u_1, \ldots, u_d, [0,v]) = \frac{\partial^d C}{\partial u_1 \cdots \partial u_d} (u_1, \ldots, u_d, v)
	\end{align}
	for $\lambda^d$-a.e. $\boldsymbol{u} \in \mathbb{I}^d$. {The same result holds 
		for any ordering of the mixed partial derivatives in Eq. \eqref{eq:derivative}.}
\end{lem}
\begin{proof}
	{Since both, a proof for the existence of the mixed partial derivatives and the fact that the 
		ordering does not matter seems to be hard to find in the literature a proof is given in the Appendix (see \ref{proof:lem:derivative}).}
\end{proof}
Obviously every linkage $C \in \mathcal{C}_{\Pi_d}^\rho$ fulfills $L(C) = C$ and, consequently, we have 
$K_{L(C)}(\boldsymbol{x}, [0,y]) = K_C(\boldsymbol{x}, [0,y])$ for every $y \in \mathbb{I}$ and 
$\lambda^d$-a.e. $\boldsymbol{x} \in \mathbb{I}^d$.

Figure \ref{fig:exa2_sample} depicts a sample of size $n = 1.000$ drawn from a tri-variate copula $C$ 
underlying the following random vector $(\tilde{X}_1, \tilde{X}_2, \tilde{Y})$: The bivariate margin $(\tilde{X}_1, \tilde{X}_2)$ follows a Marshall Olkin copula $MO_{0.5,1}$ with parameter $\alpha = 0.5$ and $\beta = 1$, whereas $\tilde{Y}$ is defined by $\tilde{Y} = \tilde{X}_1 + \tilde{X}_2$. The graphics show how the dependence 
structure between the first and second coordinate is removed by applying the linkage operation. 
\begin{figure}[!ht]
	\centering
	\includegraphics[width=4.6cm, page = 1, trim=60 60 60 60, clip]{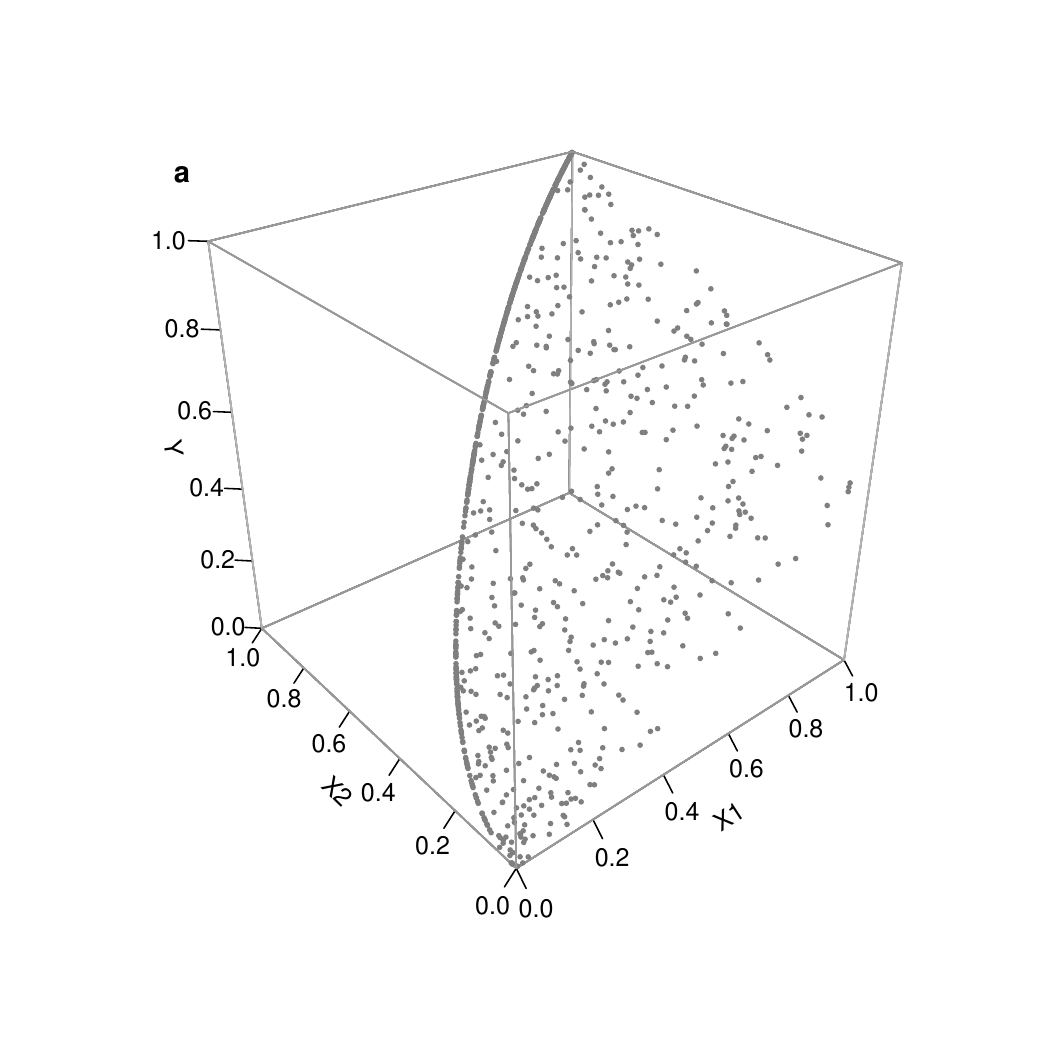} \hspace*{2cm}
	\includegraphics[width=4.6cm, page = 2, trim=60 60 60 60, clip]{exam2_3d.pdf}
	\includegraphics[width=12cm, page = 1]{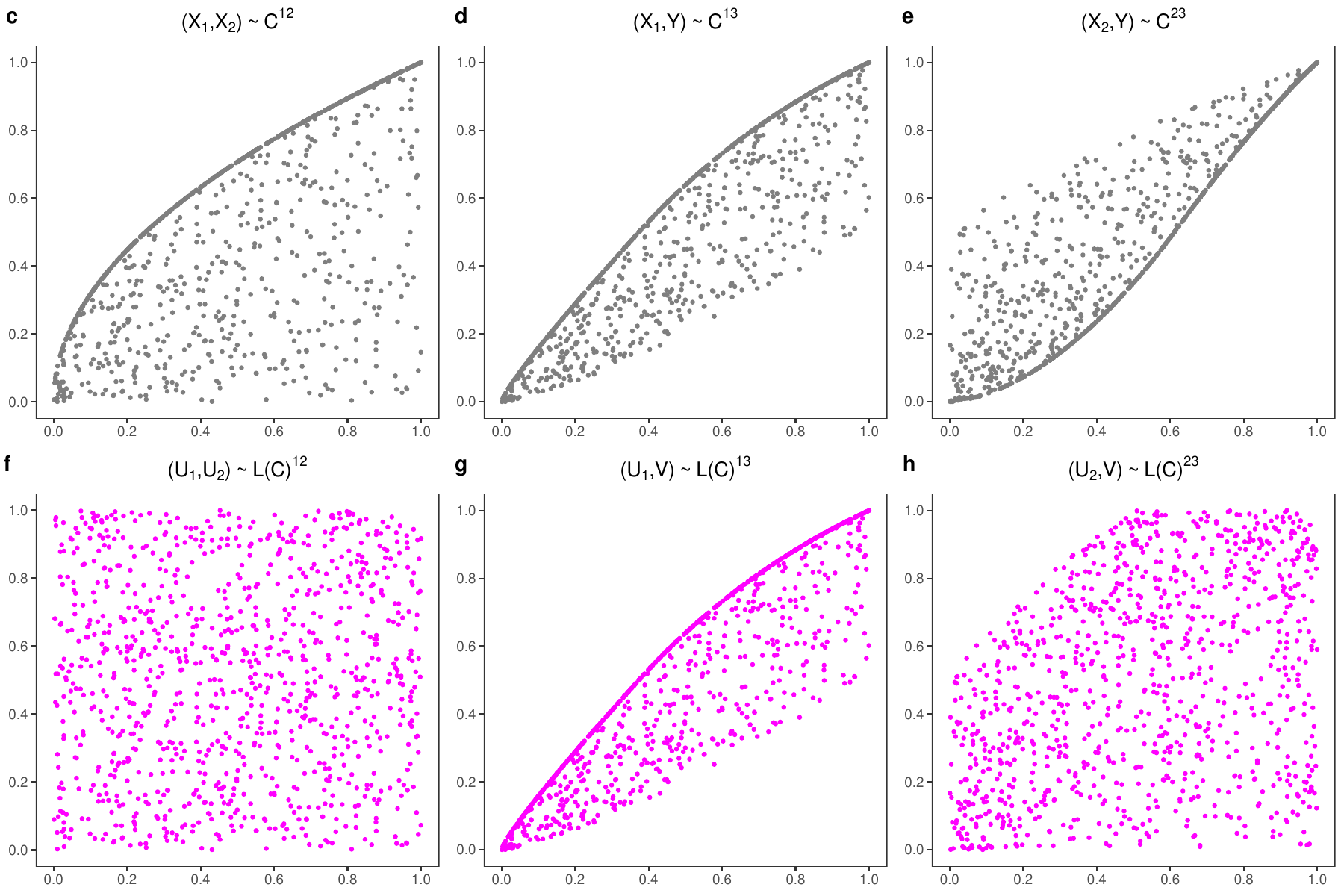}
	\caption{Sample of size $n=1.000$ drawn from the copula $C$ underlying to the random vector $(\tilde{X}_1,\tilde{X}_2, \tilde{Y})$, where $(\tilde{X}_1,\tilde{X}_2)$ follows a Marshall Olkin copula $MO_{0.5,1}$ with parameter $\alpha = 0.5$ and $\beta = 1$ and $\tilde{Y}$ is defined by $\tilde{X}_1 + \tilde{X}_2$ and its linkage $L(C_\alpha)$. (a) depicts the sample of $ (X_1,X_2,Y) \sim C$, (b) a sample of $(U_1,U_2,V) \sim L(C)$. (c-e) show the bivariate marginal 
		samples of $C$ and (f-h) the bivariate samples of the linkage $L(C)$, respectively.}
	\label{fig:exa2_sample}
\end{figure}

In general, deriving analytic formulas for $L(A)$ is intractable. For copulas $A \in \mathcal{C}^\rho$ 
satisfying the so-called conditional independence property (w.r.t. the first coordinate, also see \cite{bairamov})  
\begin{align}\label{simplified}
	K_A\left(x_1, \bigtimes_{i = 2}^\rho E_i \right) = \prod_{i = 2}^\rho K_{A^{1i}} (x_1,E_i) 
\end{align}
for $\lambda$-a.e. $x_1 \in \mathbb{I}$ and $E_2,\ldots, E_\rho \in \mathcal{B}(\mathbb{I})$, however, 
the linkage is easily expressable (see Proposition \ref{prop:condIndep}). 
Letting $\mathcal{C}^\rho_{\perp}$ denote the family of all copulas satisfying Eq. \eqref{simplified} 
it is straightforward to verify that there is a one-to-one correspondence between $\mathcal{C}^\rho_{\perp}$ 
and $\mathcal{C}^2 \times \mathcal{C}^2 \times \cdots \times \mathcal{C}^2=(\mathcal{C}^2)^d$. {The construction of higher dimensional copulas from copulas in $\mathcal{C}^2$ is also discussed in \cite{durante}[Section 5] via 
	the so-called $\mathbf{C}$-lifting - considering $\mathbf{C} = \{\Pi_2\}_{t \in \mathbb{I}}$ 
	establishes the link to Eq. \eqref{simplified}. Furthermore, the conditional independence property is also crucial 
	in the context of factor copula models, see, e.g., \cite{ansari2021,krupski2013, krupski2015} and the references therein. 
	In fact, Eq. \eqref{simplified} yields a $1$-factor copula model in which $X_2,\ldots, X_\rho$ are assumed to be conditionally independent given the random variable $X_1$.}\\
The following lemma will be useful for deriving the afore-mentioned expression for the linkage: 
\begin{lem}\label{lem:simplified}
	Suppose that $A \in \mathcal{C}^\rho$ fulfills the conditional independence property (w.r.t. to the first coordinate).
	Then for every $j \in \{2,\ldots, \rho\}$ we have
	\begin{enumerate}\setlength\itemsep{1em}
		\item $K_{A^j}(x_1, [0,x_2] \times \cdots \times [0,x_j]) = \prod_{i = 2}^j K_{A^{1i}}(x_1,[0,x_i])$ for $\lambda$-a.e. $x_1 \in \mathbb{I}$ and 
		\item $K_{A^j}(x_1,\ldots, x_{j-1}, [0,x_j]) = K_{A^{1j}}(x_1,[0,x_j])$ for $\mu_{A^{j-1}}$-a.e. $(x_1,\ldots, x_{j-1}) \in \mathbb{I}^{j-1}$. 
	\end{enumerate}
\end{lem}
\begin{proof}
	Setting $x_{j+1} = 1,x_{j+2} = 1, x_{j+3} = 1, \ldots, x_{\rho} = 1$ and using the fact that $K_{A^{1i}}(x_1,[0,1]) = 1$ for every $i \in {2,\ldots, \rho}$ the first assertion follows. Using disintegration for every $(x_1,\ldots, x_{j}) \in \mathbb{I}^{j}$ on the one hand we have
	\begin{align*}
		A^j(x_1,\ldots, x_j) &= \int_{[0,x_1] \times \cdots \times [0,x_{j-1}]} K_{A^j}(s_1,\ldots, s_{j-1}, [0,x_j]) d\mu_{A^{j-1}}(s_1, \ldots, s_{j-1}),
	\end{align*}
	on the other hand, applying \eqref{simplified} and disintegration yields
	\begin{align*}
		A^j(x_1,\ldots, x_j) &= \int_{[0,x_1]} K_{A^{j}}(s_1,[0,x_2] \times \cdots \times [0,x_j]) \, d\lambda(s_1) \\
		&= \int_{[0,x_1]} K_{A^{1j}}(s_1,[0,x_j]) \prod_{i = 2}^{j-1}K_{A^{1i}}(s_1,[0,x_i])\, d\lambda(s_1) \\
		&= \int_{[0,x_1]} K_{A^{1j}}(s_1,[0,x_j]) K_{A^{j-1}}(s_1,[0,x_2] \times \cdots \times [0,x_{j-1}])\, d\lambda(s_1) \\
		&=\int_{[0,x_1] \times \cdots \times [0,x_{j-1}]} K_{A^{1j}}(s_1,[0,x_j]) d\mu_{A^{j-1}}(s_1, \ldots, s_{j-1}).
	\end{align*}
	Applying Radon-Nikodym therefore yields the second assertion.
\end{proof}

\begin{prop}\label{prop:condIndep}
	Suppose that $(X_1,\ldots, X_d, Y) \sim A \in \mathcal{C}^\rho$ and assume there is some coordinate $j \in \{1,\ldots, d\}$ such that $X_1, \ldots, X_{j-1}, X_{j+1}, \ldots, X_d,Y$ are conditionally independent given $X_j$. 
	Then the following identity holds for all $x_1, \ldots, x_d, y \in \mathbb{I}$:  
	\begin{align*}
		L(A)(x_1, \ldots, x_d, y) = A^{j\rho}(x_j,y) \Pi_{d}(x_1, \ldots, x_{j-1}, x_{j+1}, \ldots, x_d).
	\end{align*} 
\end{prop}
\begin{proof}
	Without loss of generality we may consider $j=1$. Using disintegration and change of coordinates as well as applying Lemma \ref{lem::MKLinkage} and Lemma \ref{lem:simplified} yields
	{\small
		\begin{align*}
			L(A)&(x_1, \ldots, x_d, y) = \int_{\Omega} \ind_{[0,\boldsymbol{x}]}\circ \Phi^{\boldsymbol{R}}(\boldsymbol{X})  K_{A}(\boldsymbol{X}, [0,y]) \, d\mathbb{P}\\
			&=\int_\Omega \int_{\mathbb{I}^d}  \ind_{[0,\boldsymbol{x}]} \circ \Phi^{\boldsymbol{R}(\omega)}(\boldsymbol{s}) K_A(\boldsymbol{s},[0,y]) d\mathbb{P}^{\boldsymbol{X}}(\boldsymbol{s}) \,  d\mathbb{P}(\omega) \\
			&=\int_\Omega \int_{\mathbb{I}} \int_{\mathbb{I}^{d-1}} \ind_{[0,\boldsymbol{x}]} \circ \Phi^{\boldsymbol{R}(\omega)}(\boldsymbol{s}) K_{A^{1\rho}}(s_1,[0,y]) K_{A^d}(s_1, d(s_2,\ldots, s_d)) d\lambda(s_1) d\mathbb{P}(\omega) \\
			&=\int_{[0,x_1]} \int_\Omega K_{A^{1\rho}}(s_1,[0,y]) \prod_{i = 2}^d \int_{\mathbb{I}} \ind_{[0,x_i]}(F_{i|1}^{R_i(\omega)}(s_i|s_1)) K_{A^{1i}}(s_1, ds_i) d\mathbb{P}(\omega) d\lambda(s_1) \\
			&=\int_{[0,x_1]} K_{A^{1\rho}}(s_1,[0,y]) \prod_{i = 2}^d x_i \, d\lambda(s_1)
			= A^{1\rho}(x_1, y)\Pi_{d-1}(x_2,\ldots, x_d),
		\end{align*}
	}
	which completes the proof.
\end{proof}

The conditional independence property in Eq. \eqref{simplified} may seem far to restrictive to be 
of any practical relevance. The following examples, however, show that both, 
the family of {empirical multilinear copulas} and the class of completely dependent copulas (according to \cite[Definition 1]{sanchez2015}) satisfy Eq. \eqref{simplified}, implying that the family of copulas 
satisfying Eq. \eqref{simplified} is dense in $(\mathcal{C}^\rho, d_\infty)$.

\begin{exam}\label{exa:empCop}
	Suppose that $(\boldsymbol{X}_1,Y_1) \ldots, (\boldsymbol{X}_n,Y_n)$ is a sample from
	$C \in \mathcal{C}^\rho$ and let $\hat{C}_n$ denote the resulting empirical copula obtained via multilinear interpolation of the subcopula (see, for instance, \cite{genest2017}). Since every $\rho$-dimensional empirical {multilinear} copula is universally simplified (see \cite[Theorem 7.1]{vinepaper}), Eq. \eqref{simplified} is obviously satisfied and Proposition \ref{prop:condIndep} directly yields 
	\begin{align*}
		L(\hat{C}_n)(x_1,\ldots, x_d, y) &=\hat{C}_n^{1\rho}(x_1,y) \cdot \Pi_{d-1}(x_2,\ldots, x_d)
	\end{align*}
	for all $(\boldsymbol{x}, y) \in \mathbb{I}^\rho$.
	
	Figure \ref{fig:exa3_empCop} depicts the density of a $3$-dimensional empirical {multilinear} copula $\hat{C}_n$ and its linkage $L(\hat{C}_n)$ for a sample of size $n=20$ drawn from $(\tilde{X}_1,\tilde{X}_2,\tilde{Y})$, whereby $(\tilde{X}_1,\tilde{X}_2) \sim MO_{0.3,1}$ and $\tilde{Y} = \tilde{X}_1 + \tilde{X}_2$ (top 2 panels).
	\begin{figure}[!ht]
		\centering
		\includegraphics[width=4.6cm, page = 1, trim=60 60 60 60, clip]{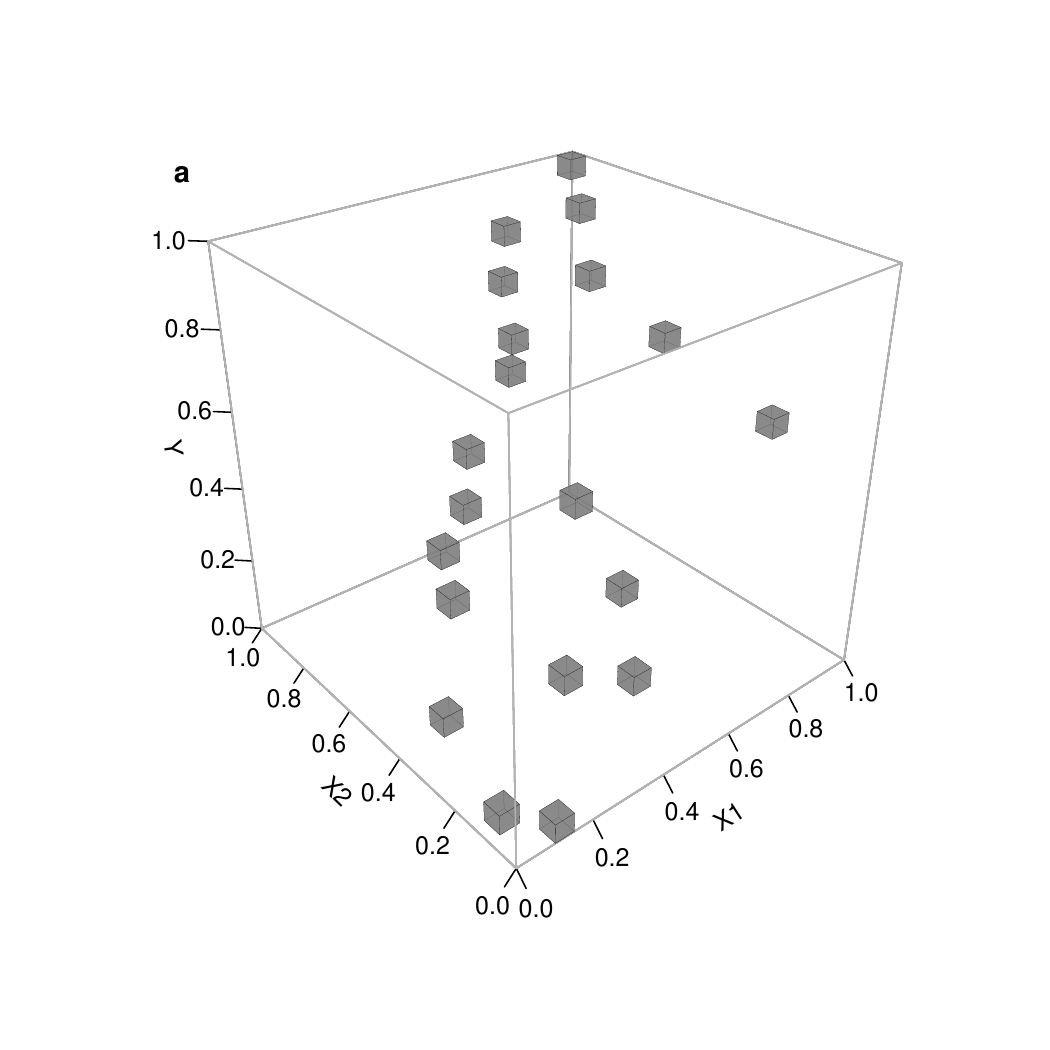} \hspace*{2cm}
		\includegraphics[width=4.6cm, page = 1, trim=60 60 60 60, clip]{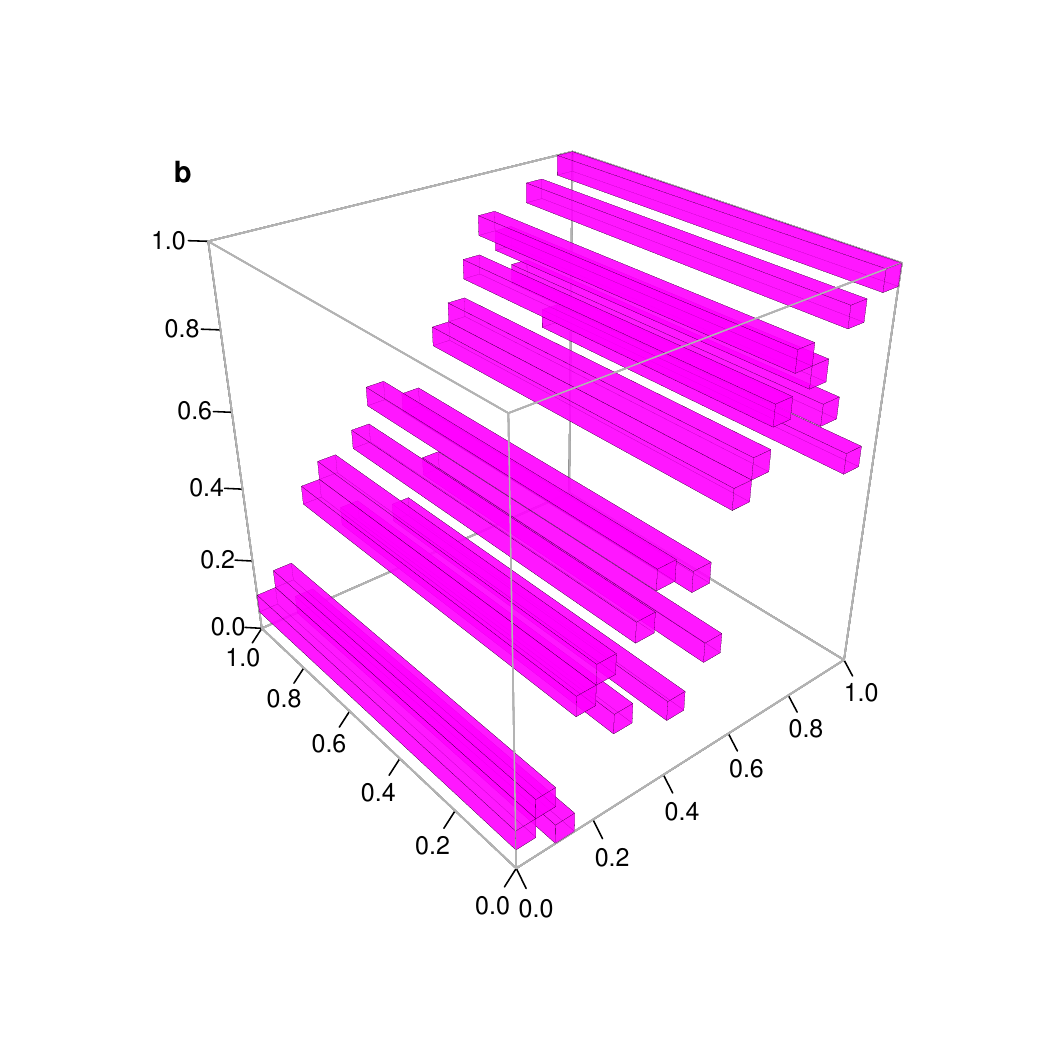}
		\includegraphics[width=12cm, page = 1]{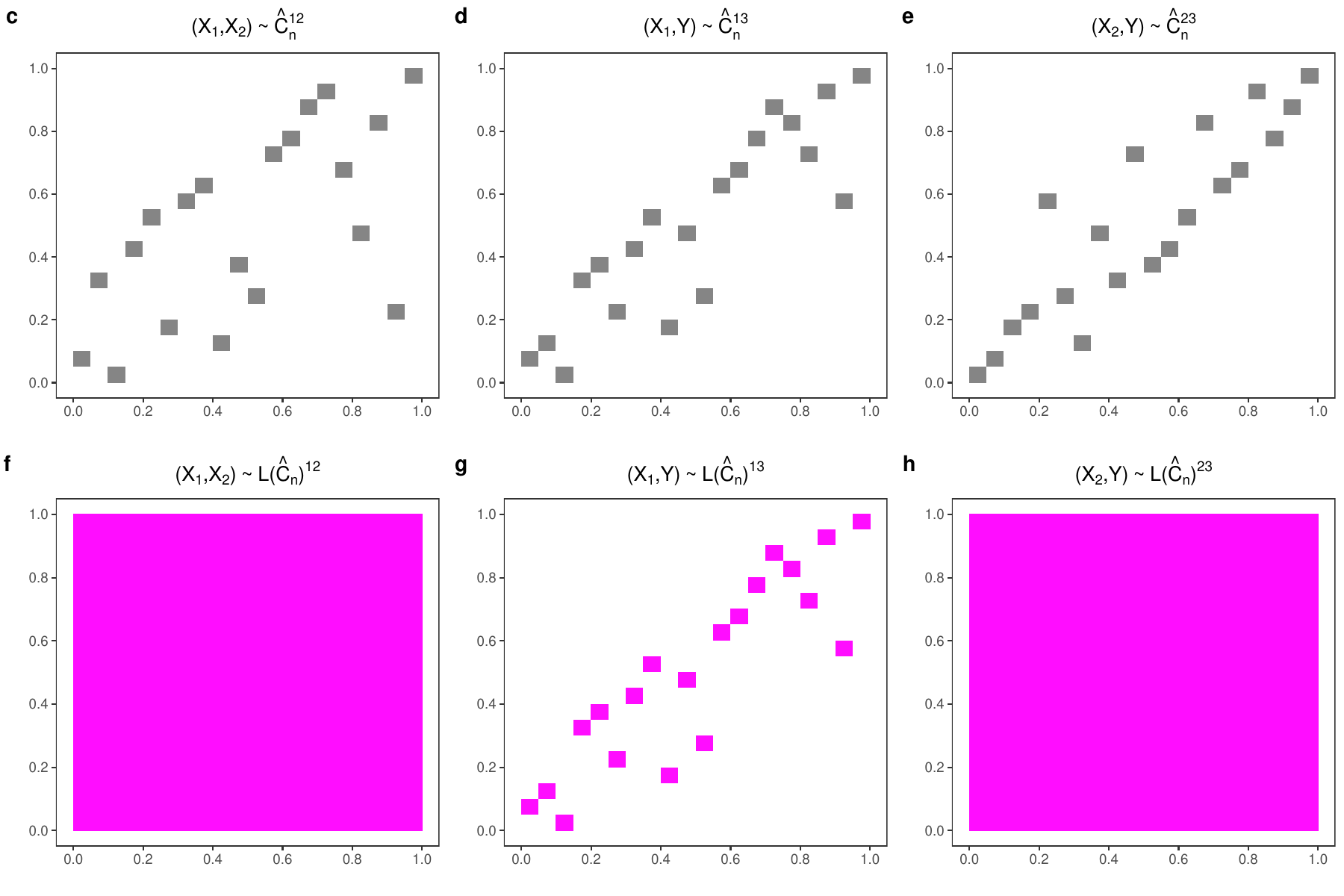}
		\caption{(a) Density of the empirical {multilinear} copula $\hat{C}_n \in \mathcal{C}^3$ of a sample of size $n = 20$ according to Example \ref{exa:empCop}. (b) Density of the linkage $L(\hat{C}_n)$ of the empirical {multilinear} copula $\hat{C}_n$. (c-e) Density of the bivariate margins of $\hat{C}_n$. (f-h) Density of the bivariate margins of $L(\hat{C}_n)$.}
		\label{fig:exa3_empCop}
	\end{figure}
\end{exam}

\begin{exam}
	Suppose that $(\boldsymbol{X},Y) \sim A \in \mathcal{C}^\rho$, whereby $\boldsymbol{X} \sim A^{d}$ is completely dependent (w.r.t. the first coordinate) according to Definition 1 in \cite{sanchez2015}, i.e., there exist $\lambda$-preserving transformations $S_2,S_3,\ldots, S_d:\mathbb{I} \to \mathbb{I}$ such that for 
	all $E \in \mathcal{B}(\mathbb{I}^{d-1})$ 
	$$K_{A^d}(x_1, E) = \ind_{E}(S_2(x_1), S_3(x_1), \ldots, S_d(x_1))$$
	is a regular conditional distribution of $A^d$. In particular, for every $j \in \{2,\ldots, d\}$ and $E_j \in \mathcal{B}(\mathbb{I})$ we have that $K_{A^{1j}}(x_1, E_j) = \ind_{E_j} \circ S_j(x_1)$ holds 
	for $\lambda$-a.e. $x_1 \in \mathbb{I}$ and every $j \in \{2,\ldots, d\}$. Hence using disintegration yields  
	\begin{align*}
		K_{A^{1\rho}}(x_1, E_\rho) &= \int_{\mathbb{I}^{d-1}} K_A(x_1,s_2,\ldots, s_d, E_\rho) K_{A^d}(x_1, d(s_2,\ldots, s_d)) \\
		&= K_A(x_1, S_2(x_1), \ldots, S_d(x_1), E_\rho) 
	\end{align*}
	for every fixed $E_\rho \in \mathcal{B}(\mathbb{I})$ and $\lambda$-almost every $x_1 \in \mathbb{I}$.
	Furthermore, again by using disintegration we get  
	\begin{align*}
		\int_{[0,x_1]} K_A&\left(s_1, \bigtimes_{j = 2}^\rho E_j\right) d\lambda(s_1) = \mu_A\left([0,x_1] \times \left(\bigtimes_{j=2}^\rho E_j\right)\right)\\
		& = \int_{[0,x_1] \times \left(\bigtimes_{j = 2}^d E_j\right)} K_A(s_1,s_2, \ldots, s_d, E_\rho) d\mu_{A^{d}}(s_1,\ldots, s_d)\\
		&= \int_{[0,x_1]} \int_{\bigtimes_{j = 2}^d E_j} K_A(\boldsymbol{s}, E_\rho) K_{A^d}(s_1, d(s_2,\ldots, s_d)) d\lambda(s_1) \\
		&=\int_{[0,x_1]} K_A(s_1, S_2(s_1), \ldots, S_d(s_1), E_\rho)  \prod_{j=2}^d \ind_{E_j}\circ S_j(s_1) d\lambda(s_1)\\
		&=\int_{[0,x_1]} \prod_{j=2}^\rho K_{A^{1j}}(s_1, E_j) d\lambda(s_1)
	\end{align*} 
	for every $x_1 \in \mathbb{I}$ implying Eq. \eqref{simplified} via Radon-Nikodym theorem. 
	Altogether this shows   
	$$L(A)(\boldsymbol{x},y) = A^{1\rho}(x_1,y)\Pi_{d-1}(x_2,\ldots, x_d).$$
\end{exam}

\noindent {Studying the map $L$ in more detail} it is possible to find copulas $A,B \in \mathcal{C}^\rho$ such that $d_\infty(A,B)$ { is strikingly large where although $d_\infty(L(A),L(B))=0$ holds.} In the $3$-dimensional setting, for instance, define $A \in \mathcal{C}^3$ implicitly by $\mu_A := \mu_M \otimes \lambda$ and $B \in \mathcal{C}^3$ by $\mu_B:=\mu_W \otimes \lambda$, whereby $M$ denotes the upper and $W$ the lower Fr\'echet-Hoeffding bound. It is straightforward to verify that $d_\infty(A,B)=1/2$ holds, {corresponding to 75\% of the diameter of $(\mathcal{C}^3, d_\infty)$}; on the other hand we have $L(A)=\Pi_3=L(B)$, implying $d_\infty(L(A),L(B))=0$. Conversely, if $d_\infty(A,B)$ is small, one might expect that so is $d_\infty(L(A), L(B))$. The following theorem shows that this conjecture is wrong.   

\begin{thm}\label{thm:discont}
	The linkage operation $L:\mathcal{C}^\rho \to \mathcal{C}_{\Pi_d}^\rho$ is not continuous w.r.t. $d_\infty$. 
\end{thm}
\begin{proof}
	We construct a sequence of $3$-dimensional copulas $A_n$ which converges to $A$ w.r.t. $d_\infty$, however, $L(A_n)$ does not converge to $L(A)$.   
	Suppose that $A \in \mathcal{C}_{\Pi_2}^3 \subseteq \mathcal{C}^3$ fulfills $A^{23} \neq \Pi_2$. Since $A \in \mathcal{C}_{\Pi_2}^3$ we have $L(A)=A$. 
	According to \cite[Proposition 3.2]{durante2010} we can find a sequence of empirical {multilinear} copulas $(A_n)_{n \in \mathbb{N}}$ such that $\lim_{n \to \infty} d_\infty(A_n,A)=0$. Considering Example \ref{exa:empCop} 
	we know that $L(A_n)^{23} = \Pi_2$, so altogether we get 
	\begin{align*}
		\liminf_{n \to \infty}d_\infty(L(A_n), L(A)) \geq  \liminf_{n \to \infty}d_\infty(L(A_n)^{23}, L(A)^{23}) =  d_\infty(\Pi_2, A^{23}) > 0.
	\end{align*}  
	Extending this construction in arbitrary dimensions $\rho \geq 4$ is straightforward.
\end{proof}

\section{The metric space $(\mathcal{C}_{\Pi_d}^\rho,D_1)$}

We now tackle the extension of the dependence measure $\zeta^1$ introduced in \cite{trutschnig2011} to the multivariate setting and start with introducing the pseudometrics $D_1,D_\infty$ and $D_p$ on $\mathcal{C}^\rho$ whose restrictions
to $\mathcal{C}_{\Pi_d}^\rho$ are metrics. 

\begin{defi} Suppose that $A,B \in \mathcal{C}^\rho$ and let $L(A), L(B)$ denote the linkage of $A$ and $B$ according to 
	Eq. \eqref{eq:link}, respectively. Then the pseudometrics $D_1, D_\infty, D_p: \mathcal{C}^\rho \to [0,1]$ are defined by
	\begin{align}\label{def::D_1}
		D_1(A,B)  &:= \int_{\mathbb{I}} \int_{\mathbb{I}^d} \lvert K_{L(A)}(\boldsymbol{x},[0,y]) - K_{L(B)}(\boldsymbol{x},[0,y])\rvert d\lambda^d(\boldsymbol{x})d\lambda(y) \\
		D_\infty(A,B) &:= \sup_{y \in \mathbb{I}} \int_{\mathbb{I}^d} \lvert K_{L(A)}(\boldsymbol{x},[0,y]) - K_{L(B)}(\boldsymbol{x},[0,y])\rvert d\lambda^d(\boldsymbol{x}) \label{def::D_infty}
	\end{align}
	as well as
	\begin{align}\label{def::D_p}
		D_p^p(A,B) &:= \int_{\mathbb{I}} \int_{\mathbb{I}^d} \left| K_{L(A)}(\boldsymbol{x},[0,y]) - K_{L(B)}(\boldsymbol{x},[0,y])\right|^p d\lambda^d(\boldsymbol{x}) d\lambda(y) 
	\end{align}
	for $p \in (1,\infty)$. 
\end{defi}
As in \cite{trutschnig2011} in the sequel we will mainly work with $D_1$ which can be interpreted as 
$L^1$-distance of the conditional distribution functions of the corresponding linkages. 
To simplify notation we will also write 
\begin{align}
	\phi_{A,B}(y):= \int_{\mathbb{I}^d} \lvert K_{L(A)}(\boldsymbol{x},[0,y]) - K_{L(B)}(\boldsymbol{x},[0,y])\rvert d\lambda^d(\boldsymbol{x}) \label{def::phi}
\end{align}
for all $A,B \in \mathcal{C}^\rho$. 
\begin{remk}
	Since every copula $A \in \mathcal{C}^2$ fulfills $A = L(A)$, in the bivariate setting the pseudometrics $D_1,D_\infty$ and $D_p$ coincide with the metrics defined in \cite{trutschnig2011}. {Notice that in the 
		$\rho$-dimensional setting the pseudometrics defined according to Eq. \eqref{def::D_p} 
		are conceptionally different to those studied in \cite{sanchez2015} - the latter consider conditioning on only 
		one variable (hence assuring identical distributions of the conditioning variable), whereas here we 
		condition on $d$-variables in order to quantify the influence of a 
		$d$-dimensional random vector $\boldsymbol{X}$ on a random variable $Y$.}
\end{remk}

The following lemma shows that $D_1,D_p, D_\infty$ defined according to \eqref{def::D_1}, \eqref{def::D_infty} and \eqref{def::D_p} are indeed metrics on the family of linkages and pseudometrics on $\mathcal{C}^\rho$. The proof essentially follows \cite[Lemma 4]{trutschnig2011} with some minor adaptations and is, therefore, deferred to the Appendix \ref{app:proofs}.

\begin{lem} \label{lem:metric}
	$D_1$, $D_\infty$ and $D_p$ defined according to \eqref{def::D_1}, \eqref{def::D_infty} and \eqref{def::D_p} are metrics on  $\mathcal{C}_{\Pi_d}^{\rho}$ and pseudo-metrics on $\mathcal{C}^\rho$. 
\end{lem}

As in the bivariate setting the following inequalities, the proof of which is analogous to \cite[Lemma 5 and Theorem 6]{trutschnig2011} and can therefore be found in the Appendix, hold:

\begin{prop}\label{prop:ineq}
	For every pair $A, B \in \mathcal{C}^\rho$ the function $\phi_{A,B}$ defined according to 
	Eq.  \eqref{def::phi} is Lipschitz-continuous with Lipschitz constant $2$. Moreover 
	the following inequalities hold:
	\begin{enumerate}[parsep = 1ex]
		\item $d_\infty(L(A),L(B)) \leq D_\infty(A,B)$
		\item $D_1(A,B) \leq D_\infty(A,B) \leq 2\sqrt{D_1(A,B)}$
		\item $D_p^p(A,B) \leq D_1(A,B) \leq D_p(A,B)$
	\end{enumerate}
\end{prop}

Proposition \ref{prop:ineq}  yields the following direct consequence. 

\begin{cor} For $A,A_1,A_2,\ldots \in \mathcal{C}^\rho$ the following three conditions are equivalent:
	\begin{enumerate}
		\item $\lim_{n \to \infty} D_1(A_n,A) = 0$,
		\item $\lim_{n \to \infty} D_\infty(A_n,A) = 0$,
		\item $\lim_{n \to \infty} D_p(A_n,A) = 0$.
	\end{enumerate}
	In other words: All three metrics induce the same notion of convergence.	
\end{cor}

\begin{remk}
	{The idea of constructing metrics via Markov operators corresponding to copulas as
		introduced in \cite{trutschnig2011} carries over to the constructions studied here: In fact, 
		defining a linear operator $T_C: L^1(\mathbb{I}, \mathcal{B}(\mathbb{I}), \lambda) \to  L^1(\mathbb{I}^d, \mathcal{B}(\mathbb{I}^d), \lambda^d)$ by setting 
		\begin{align*}
			(T_Cf)(\boldsymbol{x}):=\int_{\mathbb{I}^d} f(y) K_C(\boldsymbol{x}, dy)
		\end{align*}
		for every $f \in L^1(\mathbb{I}, \mathcal{B}(\mathbb{I}), \lambda)$ and every $C \in \mathcal{C}_{\Pi_d}^\rho$
		it is straightforward to verify that $T_C$ fulfills the following three conditions:
		\begin{enumerate}
			\item $T_C1 = 1$.
			\item For every $f \in L^1(\mathbb{I}, \mathcal{B}(\mathbb{I}), \lambda)$, $f \geq 0$ implies $T_Cf \geq 0$. 
			\item $\int_{\mathbb{I}^d} (T_Cf)(\boldsymbol{x}) d\lambda^d(\boldsymbol{x}) = \int_{\mathbb{I}} f(y) d\lambda(y)$ for every $f \in L^1(\mathbb{I}, \mathcal{B}(\mathbb{I}), \lambda)$.
		\end{enumerate}
		Hence, $T_C$ is a Markov operator in the sense of \cite{mikusinski_2009}[Definition 1]). 
		Moreover, following the proofs of Lemma 2 and Theorem 6 in \cite{trutschnig2011} it can easily be verified that for every sequence $(C_n)_{n \in \mathbb{N}} \in \mathcal{C}_{\Pi_d}^\rho$ and $C \in \mathcal{C}_{\Pi_d}^\rho$ 
		convergence of the corresponding Markov operators in the strong operator topology on $L^1(\mathbb{I}^d, \mathcal{B}(\mathbb{I}^d), \lambda^d)$ is equivalent to convergence w.r.t. $D_1$.}
\end{remk}

It is well known that the family of shuffles of $M$ as well as the family of checkerboard copulas are a dense subset of $\mathcal{C}^\rho$ with respect to the uniform distance $d_\infty$ (see, e.g., \cite{mikusinski2010}). 
According to \cite{trutschnig2011} in the bivariate setting shuffles of $M$ are not dense in 
$(\mathcal{C}^2,D_1)$ since shuffles of $M$ are special cases of complete dependence, checkerboard copulas, however, are
dense. We now show that the just-mentioned denseness also holds in $(\mathcal{C}_{\Pi_d}^\rho,D_1)$ 
and start by recalling the definition of checkerboard approximations. Let $N \in \mathbb{N}$ be arbitrary but fixed. 
Doing so, for every $N \in \mathbb{N}$ and every vector 
$\boldsymbol{i}=(i_1,\ldots, i_d) \in \mathcal{I}:= \{1,\ldots, N\}^d$ the $d$-dimensional 
hypercube $R_N^{\boldsymbol{i}}$ is defined by 
\begin{align}\label{hypercubes}
	R_N^{\boldsymbol{i}} = \left[\frac{i_1-1}{N}, \frac{i_1}{N}\right] \times \cdots \times \left[\frac{i_d-1}{N}, \frac{i_d}{N}\right].
\end{align}

\begin{defi}
	A copula $A \in \mathcal{C}^\rho$ is called checkerboard copula with resolution $N \in \mathbb{N}$ if and only if $\mu_A$ distributes its mass uniformly on each $\rho$-dimensional hypercube $R_N^{\boldsymbol{i}}$ with $\boldsymbol{i} \in \mathcal{I}$. We will refer to $\mathcal{CB}_N$ as the family of all checkerboard copulas with resolution $N$, the set $\mathcal{CB}:=\bigcup_{N \in \mathbb{N}} \mathcal{CB}_N$ denotes the family of all checkerboard copulas.   	
\end{defi}

\begin{defi}
	For a copula $A \in \mathcal{C}^\rho$ and $N \in \mathbb{N}$ the (absolutely continuous) copula $\mathfrak{CB}_N(A) \in \mathcal{CB}_N$, defined by 
	{\small\begin{align}\label{def::CB}
			\mathfrak{CB}_N(A)(\boldsymbol{x},y) := 
			N^\rho \int_{[0,(\boldsymbol{x},y)]} \sum_{\boldsymbol{i} \in \mathcal{I}, j\in \mathcal{J}} \mu_A\left(R_N^{\boldsymbol{i}} \times \left[\tfrac{j-1}{N}, \tfrac{j}{N}\right]\right) \ind_{R_N^{\boldsymbol{i}} \times \left[\frac{j-1}{N}, \frac{j}{N}\right]}(\boldsymbol{s},t)  d\lambda^\rho(\boldsymbol{s},t) ,
	\end{align}} 
	is called $N$-checkerboard approximation of $A$ (or simply $N$-checkerboard of $A$), whereby $\mathcal{I} := \{1,\ldots, N\}^d$ and $\mathcal{J} :=\{1,\ldots, N\}$. 
\end{defi}

\begin{exam}\label{exam:check}
	Suppose that $X_1 \sim \mathcal{U}(0,1)$, set $X_2 = 2X_1 (mod \ 1)$ as well as $Y = X_1$. Furthermore, let 
	$(\boldsymbol{X}_1,Y_1), \ldots, (\boldsymbol{X}_n,Y_n$ be a sample of size $n = 30$ drawn from 
	$(\boldsymbol{X},Y)=(X_1,X_2,Y)$. Figure \ref{fig:exa_CB} depicts the empirical {multilinear} copula $\hat{A}_n$ 
	together with the pseudo-observations of the sample and the $7$-checkerboard approximation 
	$\mathfrak{CB}_7(\hat{A}_n)$. 
	\begin{figure}[!ht]
		\centering
		\includegraphics[width=4.5cm, page = 1, trim=60 50 70 70, clip]{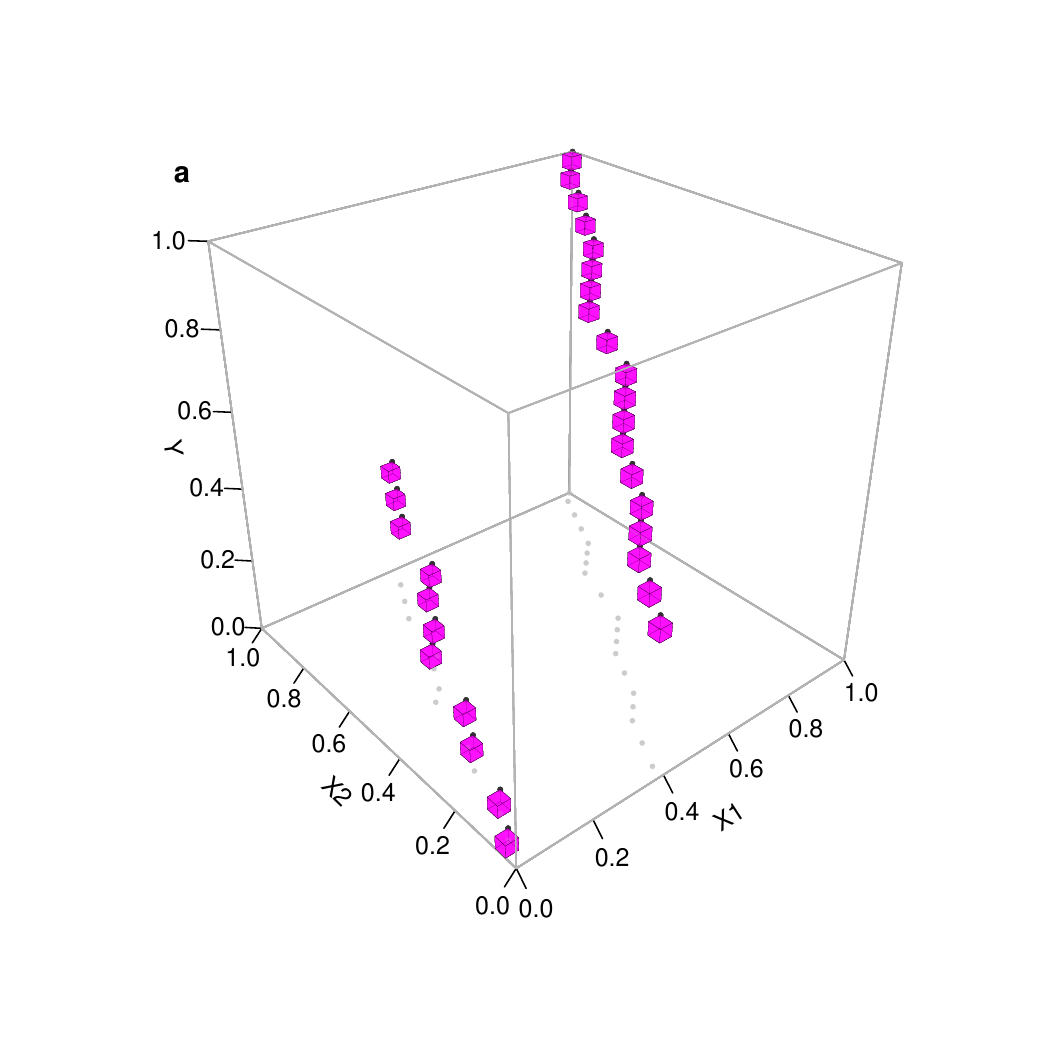} \hspace*{2cm}
		\includegraphics[width=5.3cm, page = 2, trim=65 40 10 70, clip]{exa_CB1.pdf}
		\caption{Empirical copula and pseudo-observations of the sample $(\boldsymbol{X},Y)_1, \ldots, (\boldsymbol{X},Y)_n$ of size $n = 30$ drawn from $(\boldsymbol{X}, Y) = (X_1,X_2,Y)$ according to Example \ref{exam:check} (left panel) and density of the $7$-checkerboard approximation $\mathfrak{CB}_7(\hat{A}_n)$ of the empirical {multilinear} copula $\hat{A}_n$ (right panel).}
		\label{fig:exa_CB}
	\end{figure}
\end{exam}

In the following we {partially} extend Definition 3.1 from \cite{kasper2020} and say that a sequence of copulas $(A_n)_{n \in \mathbb{N}} \in \mathcal{C}^\rho$ converges $k$-weakly conditional to $C \in \mathcal{C}^\rho$ if and only if for $\mu_{A^k}$-a.e. $\boldsymbol{x} \in \mathbb{I}^k$ we have that the sequence $(K_{A_n}(\boldsymbol{x}, \cdot))_{n \in \mathbb{N}}$ of probability measures on $\mathcal{B}(\mathbb{I}^{\rho-k})$ converges weakly to the probability measure $K_A(\boldsymbol{x}, \cdot)$. In the latter case we will write $A_n \xrightarrow{\textquoteleft \text{k-wcc'}} A$. 
{Considering, however, that bivariate and higher dimensional marginal measures of two $\rho$-dimensional 
	copulas may be singular with respect to each other, this definition of $k$-weak conditional convergence 
	only makes sense in subclasses of copulas having, e.g., the same marginals 
	(as it is the case for $\mathcal{C}^\rho_{\Pi_{d}}$), a fact illustrated by the following example:    
	\begin{exam}
		Let $(A_n)_{n \in \mathbb{N}}$ be a sequence of $\rho$-dimensional copulas such that $A_n = \Pi_\rho$ for every $n \in \mathbb{N}$ and $\rho \geq 3$. Then the Markov-kernel $K_{A_n}(\boldsymbol{x}, \cdot)$ is only unique for $\lambda^d$-a.e. $\boldsymbol{x}  \in \mathbb{I}^d$ and we could, e.g., define $K_{A_n}(\boldsymbol{x}, \cdot)$ for every $x \in [0,1]$ by 
		$K_{A_n}((x,x,\ldots, x), [0,y]):= \ind_{[0,y]}(x)$ for every $x \in \mathbb{I}$. 
		Considering the upper Fr\'echet Hoeffding bound $M_\rho \in \mathcal{C}^\rho$ with Markov kernel 
		$K_{M_\rho}((x,x,\ldots, x), [0,y]) = \ind_{[0,y]}(x)$ for $\lambda$-a.e. $x \in \mathbb{I}$, we would have 
		that $A_n$ converges $d$-weakly conditional to $M_\rho$, which is absurd.     
	\end{exam}
	Within the class of linkages, however, weak conditional convergence does make sense and, additionally, fulfills 
	the following natural properties:
	\begin{prop}\label{prop:k_wcc}
		Let $(C_n)_{n \in \mathbb{N}} \in \mathcal{C}^\rho_{\Pi_{d}}$ be a sequence of linkages such that $C_n$ 
		converges \textquoteleft$d$-wcc', to $C$, i.e. for $\lambda^d$-a.e. $\boldsymbol{x} \in \mathbb{I}^d$ we have  
		$$K_{C_n}(\boldsymbol{x}, \cdot) \xrightarrow{weakly} K_C(\boldsymbol{x}, \cdot)$$
		as $n \rightarrow \infty$. Then the following assertions hold:
		\begin{enumerate}
			\item $\lim_{n \to \infty} D_1(C_n,C) = 0$.
			\item $\lim_{n \to \infty} d_\infty(C_n,C) = 0$.
			\item For every $k \in \{1,\ldots, d\}$ we have $C_n \xrightarrow{\textquoteleft \text{k-wcc'}} C$ as $n \to \infty$.
		\end{enumerate}
	\end{prop}
	\begin{proof}
		Let $y \in [0,1]$ be arbitrary but fixed. Then considering  
		\begin{align*}
			0 = \mu_C(\mathbb{I}^d \times \{y\}) = \int_{\mathbb{I}^d} K_C(\boldsymbol{x}, \{y\}) d\lambda^d(\boldsymbol{x})
		\end{align*}
		we can find a set $\Lambda_1 \in \mathcal{B}(\mathbb{I}^d)$ {with $\lambda^d(\Lambda_1) = 1$} such that $K_C(\boldsymbol{x}, \{y\}) = 0$ holds for every $\boldsymbol{x} \in \Lambda_1$. Letting $\Lambda$ denote the set of all $\boldsymbol{x} \in \Lambda_1$ such that $K_{C_n}(\boldsymbol{x},\cdot)$ converges weakly to $K_C(\boldsymbol{x},\cdot)$ for every $\boldsymbol{x} \in \Lambda$ we get
		\begin{align*}
			\lim_{n \to \infty} K_{C_n}(\boldsymbol{x},[0,y]) = K_C(\boldsymbol{x},[0,y]),
		\end{align*}
		and applying Lebesgue's theorem on dominated convergence completes the proof of the first assertion. 
		Assertion (2) is a direct consequence of Proposition \ref{prop:ineq} (1) and (2). \\
		To prove assertion (3) we can proceed as follows:
		Set $m:= d-k$ and fix $q \in \mathbb{I} \cap \mathbb{Q}$. Then there exists a set $\Lambda_q \in \mathcal{B}(\mathbb{I}^d)$ with $\lambda^d(\Lambda_q)=1$ such that $K_C(\boldsymbol{x}, \{q\}) = 0$ as well as $\lim_{n\to \infty} K_{C_n}(\boldsymbol{x}, [0,q])=K_C(\boldsymbol{x}, [0,q])$ for every $\boldsymbol{x} \in \Lambda_{q}$. 
		Setting $\Lambda:=\bigcap_{q \in \mathbb{Q} \cap \mathbb{I}} \Lambda_q$ obviously  
		$\lambda^d(\Lambda)=1$. 
		Considering the $\boldsymbol{\tilde{x}}$-section $\Lambda_{\boldsymbol{\tilde{x}}}$ 
		of $\Lambda$, defined by
		$$\Lambda_{\boldsymbol{\tilde{x}}}:=\{\boldsymbol{s} \in \mathbb{I}^m: (\boldsymbol{\tilde{x}}, \boldsymbol{s}) \in \Lambda\}$$
		for every $\boldsymbol{\tilde{x}}:=(x_1,\ldots, x_k)$, and defining 
		$\Gamma:=\{\boldsymbol{\tilde{x}} \in \mathbb{I}^k: \lambda^m(\Lambda_{\boldsymbol{\tilde{x}}}) = 1\}$ 
		yields $\Gamma \in \mathcal{B}(\mathbb{I}^k)$ and $\lambda^k(\Gamma) = 1$. 
		For every $\boldsymbol{\tilde{x}} \in \Gamma, (y_1,\ldots, y_m) \in \Lambda_{\boldsymbol{\tilde{x}}}$ and $y_{m+1} \in \mathbb{Q} \cap \mathbb{I}$ we therefore get
		\begin{align*}
			\lim_{n \to \infty}  K_{C_n}\left(\boldsymbol{\tilde{x}}, \bigtimes_{i = 1}^{m+1} [0,y_i]\right) &= \lim_{n \to \infty} \int_{\bigtimes_{i = 1}^m [0,y_i]} K_{C_n}((\boldsymbol{\tilde{x}}, \boldsymbol{s}), [0,y_{m+1}]) d\lambda^m(\boldsymbol{s})\\
			&=  \lim_{n \to \infty} \int_{\bigtimes_{i = 1}^m [0,y_i]\cap \Lambda_{\boldsymbol{\tilde{x}}}} K_{C_n}((\boldsymbol{\tilde{x}}, \boldsymbol{s}), [0,y_{m+1}]) d\lambda^m(\boldsymbol{s})\\
			&= \int_{\bigtimes_{i = 1}^m [0,y_i] \cap \Lambda_{\boldsymbol{\tilde{x}}}} \lim_{n \to \infty}  K_{C_n}((\boldsymbol{\tilde{x}}, \boldsymbol{s}), [0,y_{m+1}]) d\lambda^m(\boldsymbol{s})\\
			&= \int_{\bigtimes_{i = 1}^m [0,y_i]} K_{C}((\boldsymbol{\tilde{x}}, \boldsymbol{s}), [0,y_{m+1}]) d\lambda^m(\boldsymbol{s})\\
			&= K_{C}\left(\boldsymbol{\tilde{x}}, \bigtimes_{i = 1}^{m+1} [0,y_i]\right).
		\end{align*} 
		Thereby we used the fact that, without loss of generality, we may assume that the identity 
		$$
		K_{A}\left(\boldsymbol{\tilde{x}}, \bigtimes_{i = 1}^{m+1} [0,y_i]\right) = 
		\int_{\bigtimes_{i = 1}^m [0,y_i]} K_{A}((\boldsymbol{\tilde{x}}, \boldsymbol{s}), [0,y_{m+1}]) d\lambda^m(\boldsymbol{s})
		$$
		holds for $A=C_n$ and $A=C$, all $\boldsymbol{\tilde{x}} \in \mathbb{I}^k$ and all $(y_1,\ldots,y_m) \in \mathbb{I}^m$. 
		Since weak convergence of $m+1$-dimensional distribution functions $H_1, H_2, \ldots $ to $H$ is well known 
		to be equivalent to pointwise convergence on a dense subset, we have shown that $\lambda^k$-almost all conditional distribution functions 
		$$(y_1, \ldots, y_{m+1}) \mapsto K_{C_n}\left(\boldsymbol{\tilde{x}}, \bigtimes_{i = 1}^{m+1} [0,y_i]\right)
		$$
		converge weakly to $(y_1, \ldots, y_{m+1}) \mapsto K_{C}\left(\boldsymbol{\tilde{x}}, \bigtimes_{i = 1}^{m+1} [0,y_i]\right)$, which completes the proof.
\end{proof}}

The following theorem shows that the checkerboard aggregation $\mathfrak{CB}_N(C)$ of a linkage $C$ converges to 
$C$ even w.r.t. \textquoteleft $d$-wcc', implying convergence w.r.t. $D_1$.   

\begin{thm}\label{thm:CBdense}
	For every $C \in \mathcal{C}_{\Pi_d}^\rho$ the checkerboard approximation $\mathfrak{CB}_N(C)$ of $C$ converges 
	$d$-weakly conditional to $C$, i.e., for $\lambda^d$-a.e. $\boldsymbol{x} \in \mathbb{I}^d$ we have  
	$$K_{\mathfrak{CB}_N(C)}(\boldsymbol{x}, \cdot) \xrightarrow{weakly} K_C(\boldsymbol{x}, \cdot)$$
	as $N \rightarrow \infty$. 
\end{thm}
\begin{proof}
	The proof is quite technical and hence deferred to the Appendix \ref{app:proofs}.
\end{proof}

\begin{cor}\label{cor:D1_conv}
	For every $C \in \mathcal{C}_{\Pi_d}^\rho$ the checkerboard approximation $\mathfrak{CB}_N(C)$ of $C$ converges to $C$ with respect to the metric $D_1$, i.e., we have 
	$$\lim_{N \to \infty} D_1(\mathfrak{CB}_N(C),C) = 0.$$
\end{cor}
{
	\begin{proof}
		The assertion directly follows from property (1) in Proposition \ref{prop:k_wcc}.
\end{proof}}

\noindent Since checkerboard copulas are absolutely continuous, Corollary \ref{cor:D1_conv} implies the following result. 

\begin{cor}
	The set of all absolutely continuous copulas is dense in $(\mathcal{C}_{\Pi_d}^\rho, D_1)$. 
\end{cor}

Further topological properties of the metric space $(\mathcal{C}_{\Pi_d}^\rho,D_1)$ analogous to the ones mentioned in \cite{trutschnig2011} also hold in arbitrary dimensions. Proving the following result can be done in 
exactly same manner as in \cite[Theorem 3]{sanchez2015}.

\begin{thm}
	The metric space $(\mathcal{C}_{\Pi_d}^\rho, D_1)$ is complete and separable. 
\end{thm}

\section{A non-parametric and multivariate dependence measure $\zeta^1$}\label{sec:depmeasure}

Main objective of this section is to study the multivariate measure of dependence $\zeta^1(\boldsymbol{X}, Y)$ 
defined according to \eqref{def::depmeasure} which quantifies the extent of dependence of a 
univariate random variable $Y$ on a $d$-dimensional random vector $\boldsymbol{X} = (X_1,\ldots, X_d)$. 
Ensuring uniqueness of the underlying copula we will from now only consider $\rho$-dimensional random vectors 
$(\boldsymbol{X},Y)$ with continuous distribution function, the general setting and consequences on the dependence 
measure $\zeta^1$ are briefly discussed in Section \ref{sec:outlook}. 
After introducing the measure $\zeta^1$ and providing an alternative handy expression we show that
$\zeta^1$ fulfills all five properties mentioned in the introduction and then derive a 
strongly consistent estimator $\hat{\zeta}_n^1$ of $\zeta^1$ which works in full generality (i.e., without any 
smoothness assumptions about the underlying copula $A$).  

For $(\boldsymbol{X}, Y)$ with copula $A \in \mathcal{C}^\rho$ the dependence measure $\zeta^1(\boldsymbol{X}, Y)$ is defined by
\begin{align}\label{def::depmeasure}
	\zeta^1(\boldsymbol{X},Y) := \zeta^1(A) = 3D_1(A,\Pi_\rho).
\end{align}
The normalization factor $3$ will become clear in Theorem \ref{thm:properties} and will, additionally, 
also be discussed in Section \ref{sec:outlook}. Obviously, $\zeta^1$ can be interpreted as $L_1$-distance 
of the conditional distribution function of $L(A)$ and independence copula $\Pi_\rho$. 

\begin{lem}\label{lem:zeta1_expression}
	Suppose that $(\boldsymbol{X}, Y)$ has continuous distribution function $H$ and copula $A$. Then 
	\begin{align*}
		\zeta^1(\boldsymbol{X}, Y) =  3\int_{\mathbb{I}} \int_{\mathbb{I}^d} \lvert K_{A}(\boldsymbol{x},[0,y]) - y \rvert d\mu_{A^d}(\boldsymbol{x}) d\lambda(y),
	\end{align*}
	whereby $A^d$ denotes the (marginal) copula of the random vector $\boldsymbol{X}=(X_1,\ldots, X_d)$. 
\end{lem}
\begin{proof} Suppose that $\boldsymbol{U} = (U_1,\ldots, U_d) \sim \Pi^d$. Since the Markov kernel of $\Pi_\rho$ is $K_{\Pi_\rho}(\boldsymbol{x}, [0,y]) = y$ applying Lemma \ref{lem::MKLinkage}, change of coordinates 
	and using the second assertion in Lemma \ref{lem::psiphi} yields
	\begin{align*}
		D_1(A, \Pi_\rho) &= \int_{\mathbb{I}} \int_{\mathbb{I}^d} \lvert K_{L(A)}(\boldsymbol{x},[0,y]) - y \rvert \, d\lambda^d(\boldsymbol{x})d\lambda(y) \\
		&=\int_{\mathbb{I}} \int_{\mathbb{I}^d} \lvert K_{A}(\Psi(\boldsymbol{x}),[0,y]) - y \rvert \, d\mathbb{P}^{\boldsymbol{U}}(\boldsymbol{x}) d\lambda(y) \\
		&=\int_{\mathbb{I}} \int_{\mathbb{I}^d} \lvert K_{A}(\boldsymbol{x},[0,y]) - y \rvert \, d\mathbb{P}^{\Psi(\boldsymbol{U})}(\boldsymbol{x}) d\lambda(y) \\
		&=\int_{\mathbb{I}} \int_{\mathbb{I}^d} \lvert K_{A}(\boldsymbol{x},[0,y]) - y \rvert \, \mu_{A^{d}}(\boldsymbol{x})d\lambda(y),
	\end{align*}
	which completes the proof.
\end{proof}

\begin{remk}
	The so-called \textquoteleft simple measure of conditional dependence $T(Y,\boldsymbol{X})$' as recently introduced in \cite{azadkia2019} and defined by 
	\begin{align}\label{zeta2}
		T(Y,\boldsymbol{X}) := \frac{\int \mathbb{E}(Var(\mathbb{P}(Y \geq t|\boldsymbol{X}))) d\mu(t)}{\int Var(\ind_{\{Y \geq t\}}) d\mu(t)},
	\end{align} 
	whereby $\mu$ denotes the law $\mathbb{P}^Y$ of $Y$, is easily shown to coincide with the squared $D_2$ distance between $A$ and $\Pi_\rho$ in our setting, whereby $A$ denotes the copula underlying 
	the random vector $(\boldsymbol{X},Y)$, i.e., we have $T(Y, \boldsymbol{X}) = 6 D_2^2(A,\Pi_\rho)$.
	In \cite{azadkia2019} a strongly consistent estimator $T_n$ for $T(Y, \boldsymbol{X})$ with very good small sample performance was derived. Since from our point of view an interpretation of negative values of 
	$T_n$ (which are possible) are not at all straightforward for applicants outside the statistical community
	our estimator for $\zeta^1$ (which is easily transferable to $D_2$) only attains values in $[0,1]$.    
\end{remk}

Before deriving the estimator we will focus on the main properties of $\zeta^1$, and start with recalling the definition of complete dependence of a random variable $Y$ (response) on a random vector $\boldsymbol{X}$ (input).

\begin{defi}[Complete dependence, \cite{lancaster1963}]
	Let $(\boldsymbol{X},Y)$ be a random vector defined on $(\Omega, \mathcal{A}, \mathbb{P})$. We say that $Y$ is completely dependent on $\boldsymbol{X}$ if there exists a measurable function $f$ such that we have 
	$\mathbb{P}(Y = f(\boldsymbol{X})) = 1.$ 
\end{defi}

The subsequent lemma provides some characterizations of complete dependence in setting considered so far:

\begin{lem}\label{lem:completeDependence}
	Suppose that $(\boldsymbol{X},Y) \sim H$ with $H$ continuous and let $F_1,\ldots, F_d$ and $G$ denote
	the corresponding univariate marginal distribution functions. Furthermore set $\boldsymbol{U} = (U_1,\ldots, U_d)$ with
	$U_1 := F_1(X_1), U_2:=F_2(X_2), \ldots, U_d:=F_d(X_d)$ and $V:=G(Y)$, and let $A$ denote the copula of 
	$(\boldsymbol{U}, V)$. Then the following statements are equivalent:
	\begin{enumerate}[parsep = 1ex]
		\item $(\boldsymbol{X},Y)$ is completely dependent.
		\item There exists a $\mu_{A^d}$-$\lambda$ preserving transformation $h:\mathbb{I}^d \to \mathbb{I}$ such that $V = h(\boldsymbol{U})$ a.s..
		\item There exists a $\mu_{A^d}$-$\lambda$ preserving transformation $h:\mathbb{I}^d \to \mathbb{I}$ such that the measure $\mu_A$ can be expressed as
		$$\mu_A(\boldsymbol{E} \times F) =  \mu_{A^d}(\boldsymbol{E} \cap h^{-1}(F))$$
		for every $\boldsymbol{E} \in \mathcal{B}(\mathbb{I}^d)$ and $F \in \mathcal{B}(\mathbb{I})$. In particular, we have 
		$$A(\boldsymbol{x}, y) = \mu_{A^{d}}([0,\boldsymbol{x}] \cap h^{-1}([0,y]))$$
		for all $\boldsymbol{x} \in \mathbb{I}^d$ and $y \in \mathbb{I}$.  
		\item There exists a $\mu_{A^d}$-$\lambda$ preserving transformation $h:\mathbb{I}^d \to \mathbb{I}$ such that 
		$$K(\boldsymbol{x}, F) := \ind_{F} \circ h(\boldsymbol{x}) = \delta_{h(\boldsymbol{x})}(F)$$
		is a regular conditional distribution of $A$.
		\item There exists a $\mu_{A^d}$-$\lambda$ preserving transformation $h:\mathbb{I}^d \to \mathbb{I}$ such that $\mu_{A}(\Gamma(h))=1$, whereby $\Gamma(h):=\{(\boldsymbol{x}, h(\boldsymbol{x})): \boldsymbol{x} \in \mathbb{I}^d \}$ denotes the graph of $h$. 
	\end{enumerate}
\end{lem}
\begin{proof}
	(1) $\Leftrightarrow$ (2): Suppose that $(\boldsymbol{X},Y)$ is completely dependent. Defining
	$h:\mathbb{I}^d \to \mathbb{I}$ by 
	$$h(u_1,\ldots, u_d) := G \circ f (F_1^-(u_1), \ldots, F_d^-(u_d)),$$
	obviously $h$ is measurable. Moreover, for $\mathbb{P}$-almost every $\omega \in \Omega$ we have 
	\begin{align*}
		V = G(Y) = G \circ f(\boldsymbol{X}) = G \circ f \circ \boldsymbol{F}^-(\boldsymbol{U}) =h(\boldsymbol{U}),
	\end{align*}
	whereby $\boldsymbol{F}^-=(F_1^-,\ldots, F_d^-)$. Considering that for every $E \in \mathcal{B}([0,1])$ we have
	\begin{align*}
		\mu_{A^d}^h(E) &= \mu_{A^d}(h^{-1}(E)) = \mathbb{P}^{\boldsymbol{U}}(h^{-1}(E)) = \mathbb{P}^{h(U)}(E) = \mathbb{P}^V(E) = \lambda(E),
	\end{align*}
	shows that $h$ is $\mu_{A^{d}}-\lambda$ preserving. To show the reverse implication define the measurable function 
	$f:\mathbb{R}^d \to \mathbb{R}$ by $f(\boldsymbol{x}):=G^- \circ h \circ \boldsymbol{F} ( \boldsymbol{x})$, which yields
	\begin{align*}
		Y = G^-(V) = G^- \circ h(\boldsymbol{U}) = G^- \circ h \circ \boldsymbol{F}(\boldsymbol{X}) = f(\boldsymbol{X}) 
		\quad a.s..
	\end{align*}
	(2) $\Rightarrow$ (3): The implication follows from the fact that for every $\boldsymbol{E} \in \mathcal{B}(\mathbb{I}^d)$ and $F \in \mathcal{B}(\mathbb{I})$ we have
	\begin{align*}
		\mu_A(\boldsymbol{E} \times F) &= \mathbb{P}^{(\boldsymbol{U},V)}(\boldsymbol{E} \times F) = \mathbb{P}^{(\boldsymbol{U}, h(\boldsymbol{U}))}(\boldsymbol{E} \times F)\\
		&= \mathbb{P}^{\boldsymbol{U}}(\boldsymbol{E} \cap h^{-1}(F)) = \mu_{A^{d}}(\boldsymbol{E} \cap h^{-1}(F))).
	\end{align*}	
	(3) $\Rightarrow$ (4): It suffices to show that $K(\boldsymbol{x},F):=\ind_{F} \circ h(\boldsymbol{x}) = \delta_{h(\boldsymbol{x})}(F)$ is a Markov kernel of $A$. Since $h$ is $\mu_{A^{d}}$-$\lambda$ preserving, $K(\boldsymbol{x},F)$ is a Markov kernel. If $\boldsymbol{E} \in \mathcal{B}([0,1]^d)$ and $F \in \mathcal{B}([0,1])$ then
	\begin{align*}
		\mu_{A}(\boldsymbol{E} \times F) = \mu_{A^d}(\boldsymbol{E} \cap h^{-1}(F)) =\int_{\boldsymbol{E}} \ind_{F} \circ h(\boldsymbol{u}) \mu_{A^d}(\boldsymbol{u}) = \int_{\boldsymbol{E}} K(\boldsymbol{x}, F) \, d\mu_{A^{d}}(\boldsymbol{x}),
	\end{align*}  
	so $K(\boldsymbol{x}, F)$ is a regular conditional distribution of $A$. \\
	(4) $\Rightarrow$ (2):
	Using disintegration we have
	\begin{align*}
		\mathbb{P}(V=h(\boldsymbol{U})) &= \mu_{A}(\Gamma(h)) =\int_{\mathbb{I}^d} K_A(\boldsymbol{u},\{h(\boldsymbol{u})\}) d\mu_{A^{d}}(\boldsymbol{u}) \\
		&= \int_{\mathbb{I}^d} \ind_{\{h(\boldsymbol{u})\}}\left(h(\boldsymbol{u})\right)\, d\mu_{A^{d}}(\boldsymbol{u})=1.
	\end{align*}
	Altogether we have therefore shown that the conditions (2), (3) and (4) are equivalent and it remains to show that (4) and (5) are equivalent, which, using disintegration, can be done as follows: On the one hand,
	\begin{align*}
		\mu_{A}(\Gamma(h)) = \int_{\mathbb{I}^d}K_A(\boldsymbol{x}, (\Gamma(h))_{\boldsymbol{x}}) \, d\mu_{A^d}(\boldsymbol{x}) = \int_{\mathbb{I}^d} \ind_{\{h(\boldsymbol{x})\}}\circ h(\boldsymbol{x}) d\mu_{A^{d}}(\boldsymbol{x}) = 1,
	\end{align*}
	and on the other hand, $\mu_{A}(\Gamma(h)) = 1$ implies that $K_A(\boldsymbol{x}, \{h(\boldsymbol{x})\}) = 1$ for $\mu_{A^{d}}$-almost every $\boldsymbol{x} \in \mathbb{I}^d$, and the proof is complete.
\end{proof}

The following example shows that Lemma \ref{lem:completeDependence} no longer holds for discontinuous marginal 
distribution functions and working with multilinear interpolation to assure uniqueness of the copula. 

\begin{exam}\label{exa:discrete}
	Suppose that $X$ and $Y$ are discrete random variables with 
	\begin{align*}
		\mathbb{P}(X = i,Y = j) = \begin{cases}
			1/3 \quad & \text{if } i = 1, j = 1\\
			1/3 \quad & \text{if } i = 2, j = 2 \\
			1/3 \quad & \text{if } i = 3, j = 1 \\
			0 \quad & \text{otherwise}.
		\end{cases}
	\end{align*}
	\begin{figure}[!ht]
		\centering
		\includegraphics[width=10cm]{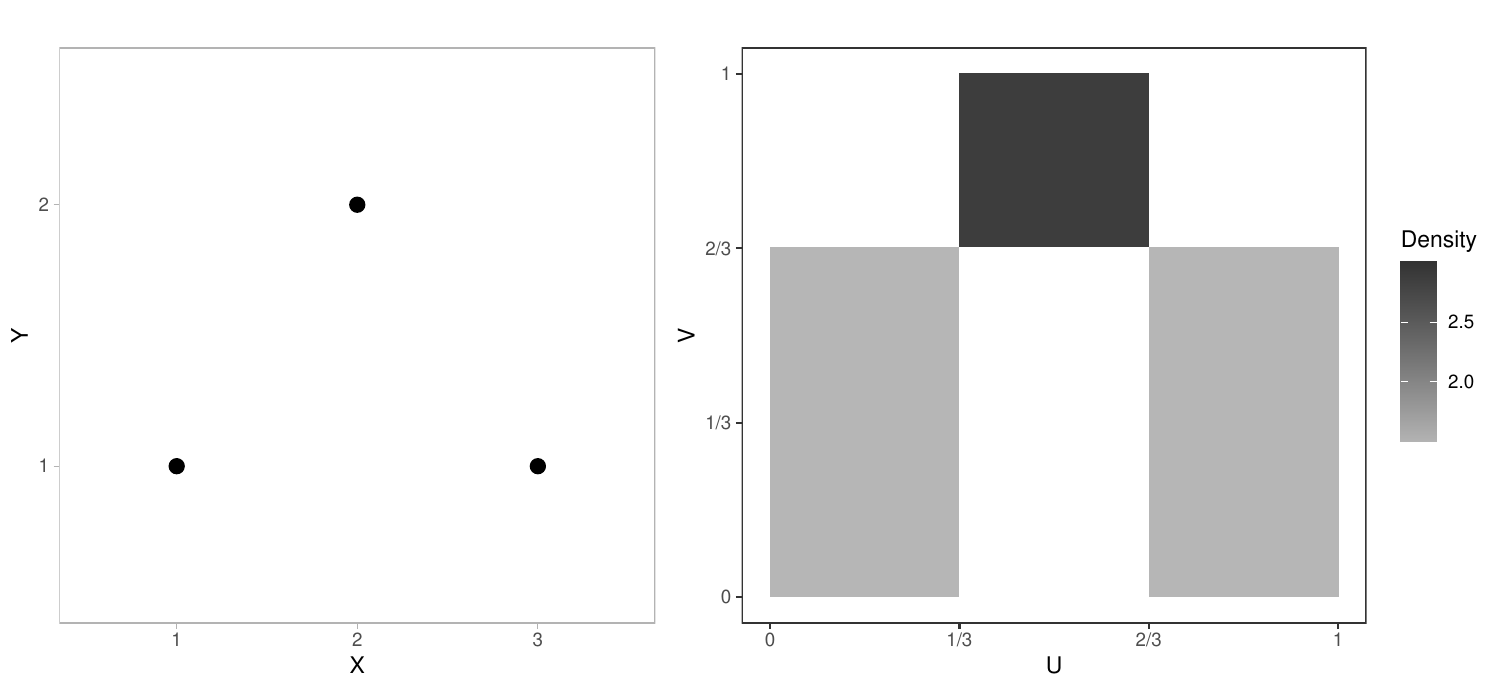}
		\caption{Support of $(X,Y)$ (left panel) and density of the corresponding unique copula $A$ (right panel) according to Example \ref{exa:discrete}.}
		\label{fig:discrete}
	\end{figure}
	
	Obviously, $Y$ is completely dependent on $X$ (but not vice versa). However, it is straightforward to verify that the unique copula $A$ underlying $(X,Y)$ in the afore-mentioned sense assigns mass uniformly to each of the rectangles $[0,1/3] \times [0,2/3]$, $[1/3,2/3] \times [2/3,1]$ and $[2/3,1] \times [0,2/3]$ and is hence far away from complete
	dependence according to Eq. (5) in Lemma \ref{lem:completeDependence}. Figure \ref{fig:discrete} depicts the support of $(X,Y)$ as well as the density of the absolutely continuous copula $A$.    
\end{exam}

We can now prove the main result of this section: 

\begin{thm}\label{thm:properties}
	Let $(\boldsymbol{X},Y)$ be a $\rho$-dimensional random vector with continuous distribution function $H$ and copula $A$. Then $\zeta^1$ fulfils the following properties:
	\begin{enumerate}[parsep = 1ex]
		\item $\zeta^1(\boldsymbol{X},Y) \in [0,1]$ (normalization).
		\item $\zeta^1(\boldsymbol{X},Y) = 0$ if and only if $Y$ and $\boldsymbol{X}$ are independent 
		(or equivalently, if and only if $L(A) = \Pi_\rho$)  (independence). 
		\item $\zeta^1(\boldsymbol{X},Y) = 1$ if and only if $Y$ is a function of $\boldsymbol{X}$ (complete dependence).
		\item If $\varphi_1,\ldots, \varphi_d,\varphi_\rho:\mathbb{R} \to \mathbb{R}$ are strictly increasing or decreasing transformations, then writing $\boldsymbol{\varphi}:=(\varphi_1, \ldots, \varphi_d)$ we have 
		$\zeta^1(\boldsymbol{\varphi}(\boldsymbol{X}),\varphi_\rho(Y)) = \zeta^1(\boldsymbol{X},Y)$ (scale-invariance).
		\item For $\boldsymbol{X} = (X_1, X_2, \ldots, X_d)$ the following chain of inequalities holds:
		\begin{align*}
			\zeta^1(X_1,Y) \leq \zeta^1((X_1,X_2), Y) \leq \ldots \leq  \zeta^1((X_1,\ldots, X_d),Y)
		\end{align*}
		(information gain).
	\end{enumerate}
\end{thm}
\begin{proof}
	Considering that $D_1$ is a metric on $\mathcal{C}_{\Pi_d}^\rho$ the second assertion {directly follows if we can show that $L(A) = \Pi_\rho$ if, and only if, $Y$ and $\boldsymbol{X}$ are independent, which 
		can be done as follows: Since $Y$ and $\boldsymbol{X}$ have continuous marginal distribution functions 
		Sklar's theorem implies that $Y$ and $\boldsymbol{X}$ are independent if, and only if, $V$ and $\boldsymbol{U}$, defined according to Lemma \ref{lem:completeDependence}, are independent. Using the independence property of conditional expectations (see \cite{klenke}) we get that $K_A(\boldsymbol{x}, [0,y]) = y$ for every $y \in \mathbb{I}$ and $\mu_{A^{d}}$-a.e. $\boldsymbol{x} \in \mathbb{I}^d$. Hence, applying Lemma \ref{lem::MKLinkage} we have 
		$$L(A)(\boldsymbol{x},y) = \int_{[0,\boldsymbol{x}]} y d\lambda^d(\boldsymbol{u}) = \Pi_\rho(\boldsymbol{x}, y).$$ 
		Suppose now that $L(A) = \Pi_\rho$. Then applying Lemma \ref{lem:derivative} we get for every $y \in \mathbb{I}$ that $ K_{L(A)}(\boldsymbol{x}, [0,y]) = y$ for $\lambda^d$-a.e. $\boldsymbol{x} \in \mathbb{I}^d$. Hence, using Lemma \ref{lem::psiphi} and \ref{lem::MKLinkage} it follows that  
		\begin{align*}
			A(\boldsymbol{x}, y) &= \int_{[0,\boldsymbol{x}]} K_A(\boldsymbol{s}, [0,y]) \, d\mu_{A^{d}}(\boldsymbol{s}) = \int_{\mathbb{I}^d} \ind_{[0,\boldsymbol{x}]}(\boldsymbol{s}) K_A(\boldsymbol{s}, [0,y]) \, d\mu_{A^{d}}(\boldsymbol{s}) \\
			&= \int_{\mathbb{I}^d} \ind_{[0,\boldsymbol{x}]}\circ \Psi(\boldsymbol{s}) K_A(\Psi(\boldsymbol{s}), [0,y]) \, d\lambda^d(\boldsymbol{s}) = y \, \int_{\mathbb{I}^d} \ind_{[0,\boldsymbol{x}]}\circ \Psi(\boldsymbol{s}) \, d\lambda^d(\boldsymbol{s}) \\
			&= y \int_{[0,\boldsymbol{x}]} d\mu_{A^{d}} = A^d(\boldsymbol{x}) y,
		\end{align*} 
		implying that $V$ and $\boldsymbol{U}$ are independent.}\\
	Let $y \in (0,1)$ be arbitrary but fixed and set 
	\begin{align*}
		D_y:=\left\{f:\mathbb{I}^d \to [0,1]: f \text{ is {measurable} and }\int_{\mathbb{I}^d} f(\boldsymbol{x}) d\mu_{A^{d}}(\boldsymbol{x})=y \right\}.
	\end{align*}
	Then according to \cite[Lemma 3.3]{bonmee2016} $f$ maximizes the function 
	$$\psi:\hat{f} \mapsto \int_{\mathbb{I}^d} \left|\hat{f}(\boldsymbol{x}) - y\right| d\mu_{A^{d}}$$
	on $D_y$ if and only if $f$ is an indicator function. In addition to that, we obtain $\max_{f \in D_y} \int_{\mathbb{I}^d} \left| f(\boldsymbol{x}) - y\right| d\mu_{A^{d}}(\boldsymbol{x}) = 2y(1-y)$. Following the proof of \cite[Lemma 12]{trutschnig2011} we conclude that 
	that for every $A \in \mathcal{C}^\rho$ the function $\phi_{A,\Pi_\rho}$ fulfills $\phi_{A,\Pi_\rho}(y) \leq 2y(1-y)$ for every $y \in [0,1]$ with equality if and only if $(\boldsymbol{X},Y)$ is completely dependent. Hence 
	$\zeta^1(\boldsymbol{X},Y) \leq 1$ with equality $\zeta^1(\boldsymbol{X},Y) = 1$ if and only if 
	$(\boldsymbol{X},Y)$ is completely dependent. To show the remaining assertions we proceed as follows:
	
	According to \cite[Theorem 2.4.1]{durante} copulas are invariant with respect to continuous and strictly increasing transformations, hence, assertion (4) is trivial if $\varphi_1, \ldots, \varphi_\rho$ are all strictly increasing. Otherwise we can proceed as follows: Let $I=\{i_1,\ldots, i_k\}$ denote the set of indices corresponding to strictly decreasing mappings $(\varphi_{i_j})_{j \in \{1,\ldots, k\}}$ and define $\sigma_i:\mathbb{I} \to \mathbb{I}$ by 
	\begin{align*}
		\sigma_i(u) = \begin{cases}
			1-u & \text{if } i \in I\\
			u & \text{otherwise.}
		\end{cases}
	\end{align*}
	Following \cite[Theorem 2.4.3]{durante} we denote the copula $\tilde{A}$ underlying the random vector $(\boldsymbol{\varphi}(\boldsymbol{X}), \varphi_\rho(Y))$ by 
	\begin{align*}
		\tilde{A}(\boldsymbol{u},v) &:= \mu_A^{(\sigma_1,\ldots, \sigma_d,\sigma_\rho)}([0,u_1] \times \ldots \times [0,u_d] \times [0,v]) \\
		&=\mathbb{P}(\sigma_1(U_1) \leq u_1,\ldots, \sigma_d(U_d) \leq u_d, \sigma_{\rho}(V) \leq v).
	\end{align*}
	Suppose that $\rho \in I$, i.e., that $\varphi_\rho$ is strictly decreasing (otherwise the proof proceeds 
	in the same manner with several simplifications). Using change of coordinates as well as the fact that $$K_{\tilde{A}}(\sigma_1(x_1), \ldots, \sigma_d(x_d),[0,y]) = K_A(x_1,\ldots, x_d, [1-y,1])$$
	holds for every $y \in [0,1]$ and $\mu_{A^{d}}$-a.e. $\boldsymbol{x} \in \mathbb{I}^d$ we have 
	\begin{align*}
		\int_{\mathbb{I}} \int_{\mathbb{I}^d} \bigg| K_{\tilde{A}}&(\boldsymbol{x},[0,y]) - y \bigg| \, d\mu_{\tilde{A}^d}(\boldsymbol{x}) d\lambda(y) \\
		& \quad = \int_{\mathbb{I}} \int_{\mathbb{I}^d} \left| K_{\tilde{A}}(\boldsymbol{x},[0,y]) - y \right| \, d\mu_{A^d}^{(\sigma_1,\ldots, \sigma_d)}(\boldsymbol{x}) d\lambda(y) \\
		&\quad =\int_{\mathbb{I}} \int_{\mathbb{I}^d} \left| K_{\tilde{A}}(\sigma_1(x_1), \ldots, \sigma_d(x_d),[0,y]) - y \right| \, d\mu_{A^d}(x_1,\ldots, x_d) d\lambda(y) \\
		&\quad =\int_{\mathbb{I}} \int_{\mathbb{I}^d} \left| K_{A}(\boldsymbol{x},[1-y,1]) - y \right| \, d\mu_{A^d}(\boldsymbol{x}) d\lambda(y).
	\end{align*}
	Therefore considering 
	$$\int_{\mathbb{I}} \int_{\mathbb{I}^d} \left| K_{A}(\boldsymbol{x},[1-y,1]) - y \right|  d\mu_{A^d}(\boldsymbol{x}) d\lambda(y) = \int_{\mathbb{I}} \int_{\mathbb{I}^d} \left| K_{A}(\boldsymbol{x},[0,y]) - y \right| d\mu_{A^d}(\boldsymbol{x}) d\lambda(y)$$ 
	assertion (4) follows.
	
	\noindent Finally, using disintegration yields that 
	\begin{align*}
		K_{A^{1\rho}}(x_1,[0,y]) &= \int_{\mathbb{I}} K_{A^{12,\rho}}(x_1,x_2,[0,y]) K_{A^{12}}(x_1,dx_2)
	\end{align*}
	holds for $\lambda$-a.e. $x_1 \in \mathbb{I}$ and that 
	\begin{align*}
		K_{A^{1,\ldots ,j-1;\rho}}&(x_1,\ldots, x_{j-1},[0,y]) \\
		&= \int_{\mathbb{I}} K_{A^{1,\ldots, j;\rho}}(x_1,\ldots, x_j,[0,y]) K_{A^{1,\ldots, j-1;j}}(x_1,\ldots, x_{j-1},dx_j)
	\end{align*}
	holds for $\mu_{A^{1,\ldots, j-1}}$-a.e. $(x_1,\ldots, x_{j-1}) \in \mathbb{I}^{j-1}$ and every $j \in \{3, \ldots, d\}$. 
	Using the triangle inequality and disintegration we therefore obtain
	\begin{align*}
		\zeta^1(X_1,Y) &= 3\int_{\mathbb{I}}\int_{\mathbb{I}} \left|K_{A^{1\rho}}(x_1,[0,y])-y\right| d\lambda(x_1) d\lambda(y) \\
		&=3\int_{\mathbb{I}}\int_{\mathbb{I}} \left|\int_{\mathbb{I}} K_{A^{12,\rho}}(x_1,x_2,[0,y]) K_{A^{12}}(x_1,dx_2)-y\right| d\lambda(x_1) d\lambda(y) \\
		&\leq 3\int_{\mathbb{I}}\int_{\mathbb{I}^2} \left| K_{A^{12,\rho}}(x_1,x_2,[0,y]) -y\right| d\mu_{A^{12}}(x_1,x_2) d\lambda(y) \\
		&= \zeta^1((X_1,X_2),Y). 
	\end{align*}
	Proceeding in the same manner yields assertion (5). 
\end{proof}

The following example shows that there are situations where random variables (input/predictor variables) 
considered individually have no influence on the response $Y$, considered jointly, however, they 
provide a lot of information on $Y$.   

\begin{exam}\label{exa:cube}
	Suppose that $(X_1$, $X_2,Y) \sim C_{Cube} \in \mathcal{C}_{\Pi_2}^3$, whereby $C_{Cube}$ denotes the uniform distribution on the union of the four cubes 
	\begin{align*}
		&\left(0,\tfrac{1}{2}\right) \times \left(0,\tfrac{1}{2}\right) \times \left(0,\tfrac{1}{2}\right) \qquad 
		\left(0,\tfrac{1}{2}\right) \times \left(\tfrac{1}{2},1\right) \times \left(\tfrac{1}{2},1\right) \\
		&\left(\tfrac{1}{2},1\right) \times \left(0,\tfrac{1}{2}\right) \times \left(\tfrac{1}{2},1\right) \qquad
		\left(\tfrac{1}{2},1\right) \times \left(0,\tfrac{1}{2}\right) \times \left(\tfrac{1}{2},1\right).
	\end{align*}
	Figure 1 in \cite{vinepaper} depicts the density of the copula $C_{Cube}$. Obviously, $C_{Cube}$ satisfies
	\begin{align*}
		C_{Cube}^{12} = C_{Cube}^{13} = C_{Cube}^{23} = \Pi,
	\end{align*}
	implying $\zeta^1(X_1,Y) = \zeta^1(X_2,Y) = \zeta^1(X_1,X_2) = 0$. In other words: Only knowing $X_1$ or 
	only knowing $X_2$ provides no additional information on $Y$. On the other hand, it is straightforward to verify 
	that $\zeta^1(C_{Cube}) = 0.75$, i.e., knowing $(X_1,X_2)$ provides a lot of information on $Y$. 
\end{exam}
Slightly modifying Example \ref{exa:cube} yields the following striking result.
\begin{thm}\label{thm:maxdep}
	For every $\delta \in [0,1)$ we find a copula $A \in \mathcal{C}^\rho$ with the following properties: 
	If $(X_1,X_2,\ldots, X_d,Y) \sim A$ then $\zeta^1(X_1,Y)  = \zeta^1(X_2,Y) = \ldots = \zeta^1(X_d,Y) = 0$ but 
	$$
	\zeta^1((X_1,X_2,\ldots, X_d),Y)> \delta.
	$$ 
\end{thm}
\begin{proof} 
	Fix $N \in \mathbb{N}$ satisfying $N > \frac{1}{1-\delta}$ and let $R_N^{i,j,k}$ denote the $3$-dimensional hypercube defined by 
	$$R_N^{i,j,k}:= \left[\frac{i-1}{N}, \frac{i}{N}\right] \times \left[\frac{j-1}{N}, \frac{j}{N}\right] \times \left[\frac{k-1}{N}, \frac{k}{N}\right]$$
	for $(i,j,k) \in \{1,\ldots, N\}^3$. Letting $C_{Cube}^N$ denote the uniform distribution on the union of the $N^2$ cubes $R_N^{i,j,k}$ whereby $(i,j) \in \{1,\ldots, N\}^2$ and $k:= j+i-1 (mod \ N)$, it is straightforward to verify that $\left(C_{Cube}^N\right)^{12} = \left(C_{Cube}^N\right)^{13} = \left(C_{Cube}^N\right)^{23} = \Pi_2$. 
	Figure \ref{fig:Ncube} (top panel) depicts the density of $C_{Cube}^4$.  
	\begin{figure}[!ht]
		\centering
		\includegraphics[width=11.8cm]{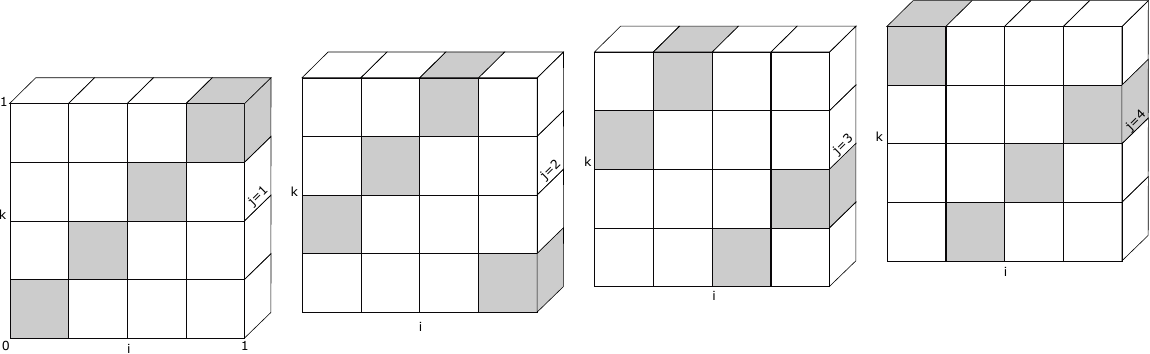}
		\includegraphics[width=11.8cm]{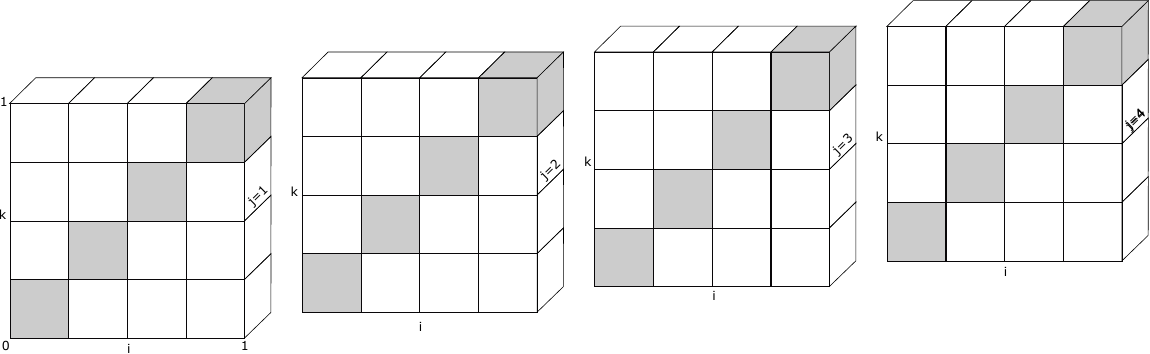}
		\caption{Density of the copula $C_{Cube}^N$ (top panel) and the copula 
			$\tilde{C}_{Cube}^N$ (bottom panel) for $N=4$ defined in the proof of Theorem \ref{thm:maxdep}. 
			The dependence measure $\zeta^1$ fulfills $\zeta^1\left(\left(C_{Cube}^N\right)^{13}\right) = \zeta^1\left(\left(C_{Cube}^N\right)^{23}\right) = \zeta^1\left(\left(C_{Cube}^N\right)^{12}\right) = 0$ as well as
			$\zeta^1\left(C_{Cube}^N\right) = \tfrac{7}{8} = \zeta^1\left(\tilde{C}_{Cube}^N\right)$.}
		\label{fig:Ncube}
	\end{figure}
	Let $\tilde{C}_{Cube}^N$ denote the uniform distribution on the union of the cubes $R_N^{i,j,i}$ with $(i,j) \in \{1,\ldots, N\}^2$. Figure \ref{fig:Ncube} (bottom panel) depicts the density of $\tilde{C}_{Cube}^4$. 
	For every $j \in \{2,\ldots, N\}$ there exists an interval exchange transformations $h_j:\mathbb{I} \to \mathbb{I}$ such that 
	$$K_{C_{Cube}^N}(h_j(x_1), x_2,[0,y]) = K_{\tilde{C}_{Cube}^N}(x_1,x_2,[0,y])$$
	holds for $\lambda^2$-a.e. $(x_1,x_2) \in \mathbb{I} \times \left(\frac{j-1}{N}, \frac{j}{N}\right)$ and every $y \in \mathbb{I}$. Since every $h_j$ is $\lambda$-preserving using change of coordinates it is straightforward to verify that  
	$D_1\left(C_{Cube}^N,\Pi_3\right) = D_1\left(\tilde{C}_{Cube}^N, \Pi_3\right)$ holds for every $N \in \mathbb{N}$. 
	It remains to show  $$D_1\left(\left(\tilde{C}_{Cube}^N\right)^{13}, \Pi_2\right) = D_1(\mathfrak{CB}_N(M),\Pi_2) \geq \delta$$ 
	which can be done as follows:  
	According to the results of Section 6 in \cite{trutschnig2011} we can find a transformation matrix $T_N$ such 
	that $V_N \, \Pi_2 := V_N(\mu_{\Pi_2}) = \mu_{\mathfrak{CB}_N(M)}$ holds (see Eq. 29 in \cite{trutschnig2011}). 
	Using the triangle inequality and the fact that {for all bivariate copulas $C_1,C_2 \in \mathcal{C}^2$ the inequality}
	$$D_1(V_N\ C_1,V_N \ C_2) \leq \tfrac{1}{N} D_1(C_1,C_2)$$
	holds, we get
	\begin{align*}
		D_1\left(\mathfrak{CB}_N(M),\Pi_2\right) & \geq D_1(\Pi_2, M) - D_1\left(\mathfrak{CB}_N(M), M\right) \\
		&=\frac{1}{3} - D_1(V_N \ \Pi_2, V_N \ M) \geq \frac{1}{3} - \frac{1}{N}D_1(\Pi_2,M) \\
		&= \frac{1}{3} - \frac{1}{3N}.
	\end{align*}
	Letting $(X_1,X_2,Y)$ be a random vector with distribution function $C_{Cube}^N$ we therefore have
	$$\zeta^1((X_1,X_2),Y) =\zeta^1(C_{Cube}^N)= \zeta^1(\tilde{C}_{Cube}^N) \geq 1- \tfrac{1}{N} > \delta,$$
	whereas $\zeta^1(X_1,Y) = \zeta^1(X_2,Y) = 0$, which proves the assertion for dimension 3. 
	Adding uniformly distributed random variables $X_3, \ldots, X_d$ such that $(X_3,\ldots, X_d)$ and $(X_1,X_2,Y)$ are independent completes the proof for arbitrary $\rho \geq 3$.  
\end{proof}

{While Theorem \ref{thm:maxdep} demonstrates that considering additional `input' variables can increase the information gain on $Y$ significantly, it is equally interesting to know, in which cases 
	adding input variables provides no further information on the output $Y$ w.r.t. $\zeta^1$. The following proposition shows that the conditional independence property as defined in Eq. \eqref{simplified}, is sufficient.  
	\begin{prop}\label{prop:no_information}
		If $(X_1,X_2,\ldots, X_d,Y) \sim A$ and let $X_2,\ldots, X_d,Y$ be conditionally independent given the variable $X_1$ (i.e., Eq. \eqref{simplified} is satisfied), then we have 
		\begin{align}\label{eq:no_information}
			\zeta^1((X_1,\ldots, X_d), Y) = \zeta^1(X_1,Y).
		\end{align}
	\end{prop}
	\begin{proof}
		Since $X_2,\ldots, X_d,Y$ are conditional independent given $X_1$ we can apply Lemma \ref{lem:simplified} (2) and obtain that for every $y \in \mathbb{I}$
		$$K_A(x_1,\ldots, x_d,[0,y]) = K_{A^{1\rho}}(x_1, [0,y])$$
		holds for $\mu_{A^d}$-a.e. $\boldsymbol{x} \in \mathbb{I}^d$. Therefore, we have 
		\begin{align*}
			\zeta^1(X_1,Y) &= 3\int_{\mathbb{I}}\int_{\mathbb{I}} |K_{A^{1\rho}}(x_1,[0,y]) - y| d\lambda(x_1)d\lambda(y)\\
			&=3\int_{\mathbb{I}}\int_{\mathbb{I}^d} |K_{A^{1\rho}}(x_1,[0,y]) - y| d\mu_{A^{d}}(x_1,\ldots, x_d)d\lambda(y)\\
			&=3\int_{\mathbb{I}}\int_{\mathbb{I}^d} |K_{A}(x_1,\ldots, x_d,[0,y]) - y| d\mu_{A^{d}}(x_1,\ldots, x_d)d\lambda(y)\\
			&= \zeta^1((X_1,\ldots, X_d),Y),
		\end{align*}
		which completes the proof.
	\end{proof}
	At the first glance, it might seem natural that the conditional independence property is also a necessary condition for Eq. \eqref{eq:no_information} - the following example falsifies this conjecture.  
	\begin{exam}
		Let $A \in \mathcal{C}_{\Pi_2}^3$ be defined by
		$$A(x_1,x_2,y):= x_1x_2y + \frac{1}{2} x_1(1-x_1)x_2^2y(1-y).$$
		It is straightforward to verify that $A$ is a $3$-dimensional copula with the following two-dimensional marginals: $A^{12}(x_1,x_2) = x_1x_2$, $A^{23}(x_2,y) = x_2y$ and $A^{13}(x_1,y)=x_1y+\frac{1}{2}x_1(1-x_1)y(1-y)$. Hence, the corresponding Markov kernels are given by
		\begin{description}
			\item $K_{A^{12}}(x,[0,z]) = K_{A^{23}}(x,[0,z]) = z$ for $\lambda$-a.e. $x\in \mathbb{I}$,
			\item $K_{A^{13}}(x_1, [0,y]) = y + \frac{1}{2}(2x_1-1)y(y-1)$ for $\lambda$-a.e. $x_1\in \mathbb{I}$,
			\item $K_{A}(x_1, [0,x_2] \times [0,y])  = x_2y + \frac{1}{2}(2x_1-1)x_2^2y(y-1)$ for $\lambda$-a.e. $x_1\in \mathbb{I}$ and
			\item $K_A(x_1,x_2,[0,y])  = y + (2x_1-1)x_2(y-1)y$ for $\lambda^2$-a.e. $(x_1,x_2)\in \mathbb{I}^2$.
		\end{description} 
		Obviously, we have that $K_{A}(x_1, [0,x_2] \times [0,y]) \neq K_{A^{12}}(x_1,[0,y])K_{A^{13}}(x_1, [0,y])$ for $\lambda$-a.e. $x_1 \in \mathbb{I}$, so Eq. \eqref{simplified} does not hold and $X_2$ and $Y$ are not conditionally independent given $X_1$. Nevertheless, the integrand of $D_1$ fulfills  
		\begin{align}
			|K_{A^{13}}(x_1,[0,y]) - y| & = \left|\int_{\mathbb{I}} \left(K_A(x_1,x_2,[0,y]) - y\right) K_{A^{12}}(x_1,dx_2)\right| \\
			&\leq \int_{\mathbb{I}} \left| K_A(x_1,x_2,[0,y]) - y   \right| K_{A^{12}}(x_1,dx_2), \label{eq:no_info_gain}
		\end{align}
		with equality if and only if $K_A(x_1,x_2,[0,y]) \geq y$ or $K_A(x_1,x_2,[0,y]) \leq y$ for $K_{A^{12}}(x_1,\cdot)$-a.e. $x_2 \in \mathbb{I}$. The latter implies that $D_1(A^{13},\Pi_2) = D_1(A,\Pi_3)$ if, and only if, we have equality in Eq. \eqref{eq:no_info_gain} for $\lambda^2$-a.e. $(x_1,y) \in \mathbb{I}^2$. Since in our example obviously
		\begin{align*}
			K_A(x_1,x_2,[0,y])  = y + (2x_1-1)x_2(y-1)y \geq y \Leftrightarrow (2x_1-1)(y-1)y \geq 0,
		\end{align*}
		we have equality in Eq. \eqref{eq:no_info_gain}, implying 
		$$\zeta^1(X_1,Y) = \zeta^1((X_1,X_2),Y) = \frac{1}{8}.$$  
\end{exam} }

\section{A strongly consistent estimator for $\zeta^1(\boldsymbol{X},Y)$}
As illustrated in Proposition 3.3 in \cite{junker}, without prior aggregation the empirical {multilinear} copula 
can not be used to estimate $\zeta^1$ in full generality. Mimicking the approach \cite{junker} we will 
use so-called empirical checkerboard copulas and construct a strongly consistent estimator for $\zeta^1$ in full generality, i.e., without any regularity assumptions on the underlying copula. The following lemma will be key for establishing strong consistency (see Lemma 1 in \cite{janssen2012}). 

\begin{lem}\label{lem:empCop}
	Suppose that $A\in \mathcal{C}^\rho$ and $\hat{A}_n$ denotes the empirical {multilinear} copula estimator. Then with probability $1$ we have
	\begin{align}
		d_\infty(\hat{A}_n, A) = \mathcal{O}\left(\sqrt{\frac{\log\log n}{n}}\right).
	\end{align}
\end{lem}

\noindent Our so-called empirical checkerboard estimator $\hat{\zeta}_n^1$ of $\zeta^1$ is defined by 
\begin{align*}
	\hat{\zeta}^1_n &:= 3 \int_{\mathbb{I}} \int_{\mathbb{I}^d} \lvert K_{L(\mathfrak{CB}_{N(n)}(\hat{A}_n))}(\boldsymbol{x},[0,y]) - y \rvert \, d\lambda^d(\boldsymbol{x}) d\lambda(y)\\
	&=3 \int_{\mathbb{I}} \int_{\mathbb{I}^d} \lvert K_{\mathfrak{CB}_{N(n)}(\hat{A}_n)}(\boldsymbol{x},[0,y]) - y \rvert \, d\mu_{(\mathfrak{CB}_{N(n)}(\hat{A}_n))^d}(\boldsymbol{x}) d\lambda(y),
\end{align*}
where the resolution $N(n)$ depends on the sample size $n$. In other words: We aggregate the empirical {multilinear} copula to a checkerboard with a coarser grid and then plug the checkerboard copula 
in the analytic expression of $\zeta^1(A)$.  
Considering that according to Theorem \ref{thm:discont} the linkage operation is discontinuous w.r.t. $d_\infty$, 
the proof in the $2$-dimensional setting as established in \cite{junker} can not be directly applied to the 
$\rho$-dimensional setting considered here, a different approach is needed. We start with the following lemma 
whose proof is analogous to the one of \cite[Lemma 3.10]{junker}: 
\begin{lem}\label{lem::ineq}
	Suppose that $A,B \in \mathcal{C}^\rho$. Then the corresponding checkerboard approximations $\mathfrak{CB}_N(A)$, $\mathfrak{CB}_N(B) \in \mathcal{CB}_N$ fulfill
	$$d_\infty(\mathfrak{CB}_N(A),\mathfrak{CB}_N(B)) \leq d_\infty(A,B).$$
\end{lem}
The next lemma states that $D_1(A,\Pi_\rho)$ can be approximated well by discretization  in $y$.  
\begin{lem}\label{lem::approx}
	For every $A \in \mathcal{C}^\rho$ and $N \geq 1$ we have 
	\begin{align}\label{eq:phi_approx}
		D_1(A,\Pi_\rho) - \frac{2}{N} \leq \frac{1}{N} \sum_{j=1}^N \int_{\mathbb{I}^d} \left|K_A\left(\boldsymbol{x},\left[0,\tfrac{j}{N}\right]\right) - \frac{j}{N} \right| d\mu_{A^{d}}(\boldsymbol{x}) \leq D_1(A,\Pi_\rho) + \frac{2}{N}.
	\end{align} 
\end{lem}
\begin{proof} 
	According to Proposition \ref{prop:ineq} the function $\phi_{A,\Pi}$ is Lipschitz-continuous with Lipschitz constant $2$. We therefore have 
	\begin{align*}
		\left| \sum_{j=1}^N \int_{\frac{j-1}{N}}^{\frac{j}{N}} \phi_{A,\Pi_\rho}(y) - \Phi_{A,\Pi_\rho}\left(\tfrac{j}{N}\right)d\lambda(y) \right| &\leq \sum_{j=1}^N \int_{\frac{j-1}{N}}^{\frac{j}{N}} \left|\phi_{A,\Pi_\rho}(y) - \phi_{A,\Pi_\rho}\left(\tfrac{j}{N}\right) \right|d\lambda(y) \\
		&\leq \frac{2}{N},
	\end{align*}
	which, considering
	$$\frac{1}{N} \sum_{j=1}^N \int_{\mathbb{I}^d} \left|K_A\left(\boldsymbol{x},\left[0,\tfrac{j}{N}\right]\right) - \frac{j}{N} \right| d\mu_{A^{d}}(\boldsymbol{x}) = \frac{1}{N} \sum_{j=1}^N \phi_{A,\Pi_\rho}\left(\frac{j}{N}\right)$$
	yields the desired result. 	
\end{proof}

Based on the previous two lemmata we can now prove the main result of this section (despite being technical the 
proof is not moved to the Appendix since it contains various ideas). 

\begin{thm}\label{thm:maintheorem}
	Let $(\boldsymbol{X}_1,Y_1), (\boldsymbol{X}_2,Y_2), \ldots$ be a random sample from $(\boldsymbol{X}, Y)$ and assume that $(\boldsymbol{X}, Y)$ has continuous distribution function $H$ and underlying copula $A \in \mathcal{C}^\rho$. Then setting $N(n):=\lfloor n^s \rfloor$ for some $s$ fulfilling $0<s<\frac{1}{2d}$ 
	\begin{align*}
		\lim_{n \to \infty} \hat{\zeta}^1_n=\lim_{n \to \infty}\zeta^1(\mathfrak{CB}_{N(n)}(\hat{A}_n)) =\zeta^1(A)
	\end{align*}
	holds with probability $1$. 
\end{thm}

\begin{proof}
	To simplify notation we will write $A_N$ for the checkerboard approximation $\mathfrak{CB}_{N}(A)$ and $\hat{A}_{N}$ for the empirical checkerboard copula $\mathfrak{CB}_{N(n)}(\hat{A}_n)$. Fix $\varepsilon > 0$. Then there exists a continuous function $f_\varepsilon:\mathbb{I}^\rho \to \mathbb{R}$ such that
	\begin{align}\label{eq:L1kernel}
		\Vert K_A - f_\varepsilon \Vert_{L^1(\mu_{A^{d}} \times \lambda)} = \int_{\mathbb{I}^d \times \mathbb{I}} \lvert K_A(\boldsymbol{x}, [0,y]) - f_\varepsilon(\boldsymbol{x},y)\rvert \, d(\mu_{A^{d}} \times \lambda)(\boldsymbol{x},y)  < \varepsilon,
	\end{align} 
	see, for instance, \cite[Theorem 3.14]{rudin} or \cite[Corollary 4.2.2]{bogachevI}. 
	Since $f_\varepsilon$ is uniformly continuous there exists some $\delta \in (0,\varepsilon)$ such that
	$$||\boldsymbol{x}- \boldsymbol{x}'||_\rho  < \delta \Longrightarrow |f_\varepsilon(\boldsymbol{x})-f_\varepsilon(\boldsymbol{x}')| < \varepsilon$$
	holds for all $\boldsymbol{x}, \boldsymbol{x}' \in \mathbb{I}^\rho$. If $N_0 \in \mathbb{N}$ fulfills $\tfrac{1}{N_0} < \delta$ and $N \geq N_0$ then using the triangle inequality yields
	{\small\begin{align*}
			&\left| D_1(\hat{A}_N,\Pi_\rho) - D_1(A, \Pi_\rho)\right|  \\
			&=\left| \int_{\mathbb{I}} \int_{\mathbb{I}^d} \lvert K_{\hat{A}_{N}}(\boldsymbol{x},[0,y]) - y \rvert \, d\mu_{\hat{A}_{N}^{d}}(\boldsymbol{x}) d\lambda(y) - \int_{\mathbb{I}} \int_{\mathbb{I}^d} \lvert K_{A}(\boldsymbol{x},[0,y]) - y \rvert d\mu_{A^{d}}(\boldsymbol{x}) d\lambda(y) \right| \\
			&\leq \underbrace{\left| \int_{\mathbb{I}} \int_{\mathbb{I}^d} \lvert K_{\hat{A}_{N}}(\boldsymbol{x},[0,y]) - y \rvert \, d\mu_{\hat{A}_{N}^{d}}(\boldsymbol{x}) d\lambda(y) - \int_{\mathbb{I}} \int_{\mathbb{I}^d} \lvert K_{A_N}(\boldsymbol{x},[0,y]) - y \rvert d\mu_{A_N^{d}}(\boldsymbol{x}) d\lambda(y) \right|}_{=:I_1} \\
			&+ \underbrace{\left| \int_{\mathbb{I}} \int_{\mathbb{I}^d} \lvert K_{A_{N}}(\boldsymbol{x},[0,y]) - y \rvert \, d\mu_{A_{N}^{d}}(\boldsymbol{x}) d\lambda(y) - \int_{\mathbb{I}} \int_{\mathbb{I}^d} \lvert K_{A}(\boldsymbol{x},[0,y]) - y \rvert d\mu_{A^{d}}(\boldsymbol{x}) d\lambda(y) \right|}_{=:I_2}.
	\end{align*}}
	We first consider the second integral $I_2$. Using the fact that $A_N$ is an $N$-checkerboard copula we have 
	\begin{align}\label{eq:mk_kernel}
		\mu_{A_{N}^{d}}(R_N^{\boldsymbol{i}}) K_{A_N}(\boldsymbol{x},[0,y]) = \mu_{A_N}(R_N^{\boldsymbol{i}} \times [0,y])
	\end{align} 
	for every $\boldsymbol{x} \in int(R_N^{\boldsymbol{i}})$; if $\mu_{A_{N}^{d}}(R_N^{\boldsymbol{i}})=0$ for some $\boldsymbol{i} \in \mathcal{I}$ we obviously have $\mu_{A_N}(R_N^{\boldsymbol{i}} \times [0,y])=0=\mu_{A}(R_N^{\boldsymbol{i}} \times [0,y])$ for every $y \in \mathbb{I}$.  
	Applying Lemma \ref{lem::approx}, Eq.  \eqref{eq:mk_kernel} and the fact that 
	$$\frac{\mu_{A_N}(R_N^{\boldsymbol{i}} \times [0,y])}{\mu_{A_{N}^{d}}(R_N^{\boldsymbol{i}})} = \frac{\mu_{A}(R_N^{\boldsymbol{i}} \times [0,y])}{\mu_{A^{d}}(R_N^{\boldsymbol{i}})}$$
	holds for $y \in \left\{0,\tfrac{1}{N},\tfrac{2}{N}, \ldots, \tfrac{N}{N}\right\}$ and every $\boldsymbol{i} \in \mathcal{I}_0:=\{\boldsymbol{i} \in \mathcal{I} \text{ with } \mu_{A^{d}}(R_N^{\boldsymbol{i}})>0\}$ we get 
	{\small\begin{align*}
			I_2  
			& \leq \frac{4}{N} + \bigg| \frac{1}{N} \sum_{j=1}^N \int_{\mathbb{I}^d} \left|K_{A_N}\left(\boldsymbol{x}, \left[0,\tfrac{j}{N}\right]\right)- \frac{j}{N}\right| \, d\mu_{A_{N}^{d}}(\boldsymbol{x}) \\
			&\hspace*{4.5cm}- \frac{1}{N} \sum_{j=1}^N \int_{\mathbb{I}^d} \left|K_{A}\left(\boldsymbol{x}, \left[0,\tfrac{j}{N}\right]\right) - \frac{j}{N}\right| \, d\mu_{A^{d}}(\boldsymbol{x})\bigg| \\
			&= \frac{4}{N} + \bigg| \frac{1}{N} \sum_{j=1}^N \sum_{\boldsymbol{i} \in \mathcal{I}_0} \int_{R_N^{\boldsymbol{i}}} \left|\frac{\mu_{A_N}\left(R_N^{\boldsymbol{i}} \times \left[0,\tfrac{j}{N}\right]\right)}{\mu_{A_{N}^{d}}(R_N^{\boldsymbol{i}})} - \frac{j}{N}\right| \, d\mu_{A_{N}^{d}}(\boldsymbol{x}) \\
			&\hspace*{4.5cm}- \frac{1}{N} \sum_{j=1}^N \sum_{\boldsymbol{i} \in \mathcal{I}_0} \int_{R_N^{\boldsymbol{i}}} \left|K_{A}\left(\boldsymbol{x}, \left[0,\tfrac{j}{N}\right]\right) - \frac{j}{N}\right| \, d\mu_{A^{d}}(\boldsymbol{x})\bigg| \\
			&= \frac{4}{N} + \bigg| \frac{1}{N} \sum_{j=1}^N \sum_{\boldsymbol{i} \in \mathcal{I}_0} \int_{R_N^{\boldsymbol{i}}} \left|\frac{\mu_{A}\left(R_N^{\boldsymbol{i}} \times \left[0,\tfrac{j}{N}\right]\right)}{\mu_{A^{d}}(R_N^{\boldsymbol{i}})} - \frac{j}{N}\right| \, d\mu_{A^{d}}(\boldsymbol{x}) \\
			&\hspace*{4.5cm}- \frac{1}{N} \sum_{j=1}^N \sum_{\boldsymbol{i} \in \mathcal{I}_0} \int_{R_N^{\boldsymbol{i}}} \left|K_{A}\left(\boldsymbol{x}, \left[0,\tfrac{j}{N}\right]\right) - \frac{j}{N}\right| \, d\mu_{A^{d}}(\boldsymbol{x})\bigg| \\
			&\leq \frac{4}{N} +  \frac{1}{N} \sum_{j=1}^N \sum_{\boldsymbol{i} \in \mathcal{I}_0} \int_{R_N^{\boldsymbol{i}}} \left| \frac{\mu_A\left(R_N^{\boldsymbol{i}} \times \left[0, \tfrac{j}{N}\right]\right)}{\mu_{A^{d}}(R_N^{\boldsymbol{i}})} - K_A\left(\boldsymbol{x}, \left[0,\tfrac{j}{N}\right]\right)\right| \, d\mu_{A^{d}}(\boldsymbol{x})\\
			&=\frac{4}{N} +  \frac{1}{N} \sum_{j=1}^N \sum_{\boldsymbol{i} \in \mathcal{I}_0} \int_{R_N^{\boldsymbol{i}}} \bigg| \frac{1}{\mu_{A^{d}}(R_N^{\boldsymbol{i}})}\int_{R_N^{\boldsymbol{i}}} K_A\left(\boldsymbol{s}, \left[0,\tfrac{j}{N}\right]\right) \\
			&\hspace*{4.5cm}- K_A\left(\boldsymbol{x}, \left[0,\tfrac{j}{N}\right]\right) \, d\mu_{A^{d}}(\boldsymbol{s})\bigg| \, d\mu_{A^{d}}(\boldsymbol{x})\\
			&\leq \frac{4}{N} + \sum_{j=1}^N \int_{\left[\frac{j-1}{N}, \frac{j}{N}\right]} \sum_{\boldsymbol{i} \in \mathcal{I}_0} \int_{R_N^{\boldsymbol{i}}} \frac{1}{\mu_{A^{d}}(R_N^{\boldsymbol{i}})}\int_{R_N^{\boldsymbol{i}}}  \bigg|K_A\left(\boldsymbol{s}, \left[0,\tfrac{j}{N}\right]\right) \\
			&\hspace*{4.5cm} - K_A\left(\boldsymbol{x}, \left[0,\tfrac{j}{N}\right]\right)\bigg| \, d\mu_{A^{d}}(\boldsymbol{s}) \, d\mu_{A^{d}}(\boldsymbol{x}) \, d\lambda(y).
	\end{align*}}
	Again using the triangle inequality, the fact that the mapping $y \mapsto K_A(\cdot, [0,y])$ is non-decreasing as 
	well as disintegration we therefore get 
	{\small\begin{align*}
			I_2 &\leq \frac{4}{N} + \sum_{j=1}^N \int_{\left[\frac{j-1}{N}, \frac{j}{N}\right]} \sum_{\boldsymbol{i} \in \mathcal{I}_0} \int_{R_N^{\boldsymbol{i}}} \frac{1}{\mu_{A^{d}}(R_N^{\boldsymbol{i}})}\int_{R_N^{\boldsymbol{i}}}  \bigg|K_A\left(\boldsymbol{s}, \left[0,\tfrac{j}{N}\right]\right) \\
			&\hspace*{4cm}- K_A\left(\boldsymbol{s}, \left[0,y\right]\right)\bigg| \, d\mu_{A^{d}}(\boldsymbol{s}) \, d\mu_{A^{d}}(\boldsymbol{x}) \, d\lambda(y)\\
			&\quad + \sum_{j=1}^N \int_{\left[\frac{j-1}{N}, \frac{j}{N}\right]} \sum_{\boldsymbol{i} \in \mathcal{I}_0}\int_{R_N^{\boldsymbol{i}}} \frac{1}{\mu_{A^{d}}(R_N^{\boldsymbol{i}})}\int_{R_N^{\boldsymbol{i}}}  \bigg|K_A\left(\boldsymbol{s}, \left[0,y\right]\right) \\
			&\hspace*{4cm}- K_A\left(\boldsymbol{x}, \left[0,y\right]\right)\bigg| \, d\mu_{A^{d}}(\boldsymbol{s}) \, d\mu_{A^{d}}(\boldsymbol{x}) \, d\lambda(y)\\
			&\quad + \sum_{j=1}^N \int_{\left[\frac{j-1}{N}, \frac{j}{N}\right]} \sum_{\boldsymbol{i} \in \mathcal{I}_0}\int_{R_N^{\boldsymbol{i}}} \frac{1}{\mu_{A^{d}}(R_N^{\boldsymbol{i}})}\int_{R_N^{\boldsymbol{i}}}  \bigg|K_A\left(\boldsymbol{x}, \left[0,y\right]\right) \\
			&\hspace*{4cm}- K_A\left(\boldsymbol{x}, \left[0,\tfrac{j}{N}\right]\right)\bigg| \, d\mu_{A^{d}}(\boldsymbol{s}) \, d\mu_{A^{d}}(\boldsymbol{x}) \, d\lambda(y) \\
			&\leq \frac{6}{N}  + \int_{[0,1]} \sum_{\boldsymbol{i} \in \mathcal{I}_0} \int_{R_N^{\boldsymbol{i}}} \frac{1}{\mu_{A^{d}}(R_N^{\boldsymbol{i}}))}\int_{R_N^{\boldsymbol{i}}}  \bigg|K_A\left(\boldsymbol{s}, \left[0,y\right]\right) \\
			&\hspace*{4cm}- K_A\left(\boldsymbol{x}, \left[0,y\right]\right)\bigg| \, d\mu_{A^{d}}(\boldsymbol{s}) \, d\mu_{A^{d}}(\boldsymbol{x}) \, d\lambda(y).
	\end{align*}}
	\noindent Applying the triangle inequality and equation  \eqref{eq:L1kernel} yields 
	{\small\begin{align*}
			I_2
			&\leq \frac{6}{N}  \\
			&\quad+\int_{[0,1]} \sum_{\boldsymbol{i} \in \mathcal{I}_0}\int_{R_N^{\boldsymbol{i}}} \frac{1}{\mu_{A^{d}}(R_N^{\boldsymbol{i}})}\int_{R_N^{\boldsymbol{i}}}  \left|K_A\left(\boldsymbol{s}, \left[0,y\right]\right) - f_{\varepsilon}(\boldsymbol{s},y)\right| \, d\mu_{A^{d}}(\boldsymbol{s}) \, d\mu_{A^{d}}(\boldsymbol{x}) \, d\lambda(y) \\
			&\quad+ \int_{[0,1]} \sum_{\boldsymbol{i} \in \mathcal{I}_0} \int_{R_N^{\boldsymbol{i}}} \frac{1}{\mu_{A^{d}}(R_N^{\boldsymbol{i}})}\int_{R_N^{\boldsymbol{i}}}  \left|f_{\varepsilon}(\boldsymbol{s},y) - f_{\varepsilon}(\boldsymbol{x},y)\right| \, d\mu_{A^{d}}(\boldsymbol{s}) \, d\mu_{A^{d}}(\boldsymbol{x}) \, d\lambda(y) \\
			&\quad + \int_{[0,1]} \sum_{\boldsymbol{i} \in \mathcal{I}_0} \int_{R_N^{\boldsymbol{i}}} \frac{1}{\mu_{A^{d}}(R_N^{\boldsymbol{i}})}\int_{R_N^{\boldsymbol{i}}}  \left|f_{\varepsilon}(\boldsymbol{x},y) - K_A\left(\boldsymbol{x}, \left[0,y\right]\right)\right| \, d\mu_{A^{d}}(\boldsymbol{s}) \, d\mu_{A^{d}}(\boldsymbol{x}) \, d\lambda(y) \\
			&\leq \frac{6}{N}  + \Vert K_A - f_\varepsilon \Vert_{L^1(\mu_{A^{d}} \times \lambda)} + \varepsilon +  \Vert K_A - f_\varepsilon \Vert_{L^1(\mu_{A^{d}} \times \lambda)} \\
			&< 9\varepsilon.
	\end{align*}}
	
	\noindent For $I_1$ we proceed as follows. Using equation  \eqref{eq:mk_kernel}, applying the triangle inequality and 
	the fact that $\big| |x|-|y|\big| \leq |x-y|$ holds for all real numbers $x,y$ we have\\
			{\small
				\begin{align*}
					I_1 &= \bigg| \int_{\mathbb{I}} \sum_{\boldsymbol{i} \in \mathcal{I}_0} \bigg(\int_{R_N^{\boldsymbol{i}}} \left| \frac{\mu_{\hat{A}_{N}}(R_N^{\boldsymbol{i}} \times [0,y])}{\mu_{\hat{A}_{N}^{d}}(R_N^{\boldsymbol{i}})} - y \right| \, d\mu_{\hat{A}_{N}^{d}}(\boldsymbol{x}) \\
					&\qquad - \int_{R_N^{\boldsymbol{i}}} \left| \frac{\mu_{A_N}(R_N^{\boldsymbol{i}} \times [0,y])}{\mu_{A_{N}^{d}}(R_N^{\boldsymbol{i}})} - y \right| d\mu_{A_N^{d}}(\boldsymbol{x}) \bigg)d\lambda(y) \bigg|\\
					&=\left| \int_{\mathbb{I}}\sum_{\boldsymbol{i} \in \mathcal{I}_0} \left| \mu_{\hat{A}_{N}}(R_N^{\boldsymbol{i}} \times [0,y])- y\mu_{\hat{A}_{N}^{d}}(R_N^{\boldsymbol{i}}) \right|  - \left| \mu_{A_N}(R_N^{\boldsymbol{i}} \times [0,y])- y\mu_{A_{N}^{d}}(R_N^{\boldsymbol{i}}) \right|\, d\lambda(y) \right|\\
					&\leq \int_{\mathbb{I}} \sum_{\boldsymbol{i} \in \mathcal{I}_0} \left| \mu_{\hat{A}_{N}}(R_N^{\boldsymbol{i}} \times [0,y]) - \mu_{A_N}(R_N^{\boldsymbol{i}} \times [0,y])\right| + y \left|\mu_{\hat{A}_N^{d}}(R_N^{\boldsymbol{i}})-\mu_{A_{N}^{d}}(R_N^{\boldsymbol{i}}) \right| \, d\lambda(y)\\
					&\leq C \cdot N^d  \cdot \left(d_\infty(\hat{A}_{N}, A_{N}) + d_\infty(\hat{A}_N^{d}, A_{N}^{d}) \right) \leq C \cdot N^d  \cdot \left(d_\infty(\hat{A}_n, A) + d_\infty(\hat{A}_n^{d}, A^{d}) \right)\\
					&\leq 2C \cdot N^d  \cdot d_\infty(\hat{A}_n, A).
				\end{align*}
	}
	\noindent According to Lemma \ref{lem:empCop} there exists a set $\Lambda \in \mathcal{A}$ with $\mathbb{P}(\Lambda)=1$ such that for every $\omega \in \Lambda$ we can find a constant $c(\omega)>0$ and an index $n_0=n_0(\omega) \in \mathbb{N}$ such that 
	\begin{align*}
		d_\infty(\hat{A}_n(\omega), A) \leq c(\omega) \sqrt{\frac{\log(\log(n))}{n}}
	\end{align*}
	holds for all $n \geq n_0$. Altogether we conclude that for every $\omega \in \Lambda$ and $N(n)=\lfloor n^s \rfloor$ with $0<s<\frac{1}{2d}$ we have convergence and the proof is complete. 
\end{proof}

\begin{remk}
	Simulations (see Appendix \ref{app:simstudy}) insinuate that the range of the parameter $s$ according to
	Theorem \ref{thm:maintheorem} for which we have consistency can be extended to the interval 
	$\left(0,\frac{1}{d}\right)$. We conjecture that the optimal choice of $s$ (optimal in the sense that the estimator performs well independent of the underlying dependence structure) is setting $s = \frac{1}{\rho}$. \\
	As a consequence, the publicly available 
	R-package \textquoteleft qmd' (short for: quantification of multivariate dependence), which 
	contains the afore-mentioned estimator considers $s = \frac{1}{\rho}$. All simulations presented
	in the Appendix can be reproduced using the \textquoteleft qmd'-package.
\end{remk}

\section{Concluding remarks}\label{sec:outlook}
This paper generalizes and extends results going back to \cite{bonmee2016, trutschnig2011} by considering 
linkages and a metric on the space of linkages which induced the multivariate dependence measure $\zeta^1$. 
The derived checkerboard estimator is strongly consistent in full generality. 

As one of the next steps we will address properties of $\zeta^1$ (Theorem \ref{thm:properties}) and strong consistency of 
$\hat{\zeta}_n^1$ in the discrete and mixed setting. As already mentioned, ties have a strong influence on the 
resulting copula. We conjecture that when working with multilinear interpolations (to assure uniqueness of the copula) 
the marginal distribution of $Y$ has to be incorporated in the definition of $\zeta^1$, particularly in order to 
attain the maximum value of $1$ can not be reached in the setting of complete dependence. 
One possible approach could be to modify the definition of $\zeta^1$ to 
\begin{align}\label{def:zeta1_discrete}
	\zeta^1(\boldsymbol{X},Y) := \frac{D_1(A,\Pi_\rho)}{D_1(C,\Pi_2)},
\end{align}
where $A\in \mathcal{C}^\rho$ is the copula underyling $(\boldsymbol{X}, Y)$ and $C \in \mathcal{C}^2$ denotes the 
copula underlying $(Y,Y)$ (constructed via bilinear interpolation of the subcopula induced by $(Y,Y)$). 

Additionally, we will try to prove or falsify the conjecture that the permissible range of the parameter $s$ according 
to Theorem \ref{thm:maintheorem} can be enlarged to $\left(0,\tfrac{1}{d}\right)$.

\appendix

\section{Additional proofs}\label{app:proofs}

\begin{proof}[Proof of Lemma \ref{lem:derivative}] \label{proof:lem:derivative}
	We prove the statement for dimension $\rho=4$, the gene\-ral case can be handled analogously. 
	Since $C \in \mathcal{C}_{\Pi_3}^4$ is a linkage disintegration yields  
	\begin{align*}
		C(x_1,x_2,x_3,y) = \int_{[0,x_1]} \int_{[0,x_2]} \int_{[0,x_3]} K_C(s_1,s_2,s_3,[0,y]) d\lambda(s_3)d\lambda(s_2)d\lambda(s_1)
	\end{align*}
	for all $x_1,x_2,x_3,y \in \mathbb{I}$. 
	Fix $y \in \mathbb{I}$. Then for arbitrary $(x_2,x_3) \in \mathbb{I}^2$ there exists a set $\Lambda^y_{(x_2,x_3)} \in \mathcal{B}(\mathbb{I})$ with $\lambda(\Lambda^y_{(x_2,x_3)}) = 1$ such that for every 
	$x_1 \in \Lambda^y_{(x_2,x_3)} \cap (0,1)$ the partial derivative $\frac{\partial C}{\partial x_1}$ in $(x_1,x_2,x_3,y)$ exists and fulfills 
	\begin{align*}
		\frac{\partial}{\partial x_1}C(x_1,x_2,x_3,y) = \int_{[0,x_2]} \int_{[0,x_3]} K_C(x_1,s_2,s_3,[0,y]) d\lambda(s_3)d\lambda(s_2).
	\end{align*}
	Setting $\Lambda^y:=(0,1) \cap \bigcap_{(x_2,x_3) \in \mathbb{Q}^2 \cap \mathbb{I}^2} \Lambda_{(x_2,x_3)}^y$ obviously yields 
	$\Lambda^y \in \mathcal{B}(\mathbb{I})$ as well as $\lambda(\Lambda^y) = 1$. 
	Consider $x_1 \in \Lambda^y$ as well as $(x_2,x_3) \in \mathbb{I}^2$, suppose that $\underline{x}_2, \overline{x}_2, \underline{x}_3, \overline{x}_3 \in \mathbb{Q} \cap \mathbb{I}$ fulfill $\underline{x}_2 \leq x_2 \leq \overline{x}_2$, $\underline{x}_3 \leq x_3 \leq \overline{x}_3$ and define $I_h$ by
	\begin{align*}
		I_h(x_1,x_2,x_3,y):= \frac{C(x_1+h, x_2, x_3, y)-C(x_1,x_2,x_3,y)}{h}
	\end{align*}
	for $h \in [-x_1, 1-x_1] \setminus \{0\}$.
	Using the fact that $C$ is $4$-increasing obviously 
	\begin{align}\label{I_h_ineq}
		I_h(x_1,\underline{x}_2, \underline{x}_3,y) \leq I_h(x_1,x_2,x_3,y) \leq I_h(x_1,\overline{x}_2, \overline{x}_3,y)
	\end{align} 
	holds. Moreover, by construction, for the left and the right part of Eq. \eqref{I_h_ineq} the limit for $h \to 0$ exists and fulfills  
	\begin{align*}
		\lim_{h \to 0} I_h(x_1,\underline{x}_2, \underline{x}_3,y) &= \int_{[0,\underline{x}_2]} \int_{[0,\underline{x}_3]} K_C(x_1,s_2,s_3,[0,y]) d\lambda(s_3)d\lambda(s_2), \\
		\lim_{h \to 0} I_h(x_1,\overline{x}_2, \overline{x}_3,y) &= \int_{[0,\overline{x}_2]} \int_{[0,\overline{x}_3]} K_C(x_1,s_2,s_3,[0,y]) d\lambda(s_3)d\lambda(s_2), 
	\end{align*}
	which (via considering limes inferior and limes superior and the fact that $\mathbb{Q}$ is dense in $\mathbb{R}$) 
	implies the existence of $\lim_{h \to 0} I_h(x_1,x_2,x_3,y)=\frac{\partial}{\partial x_1} C(x_1,x_2,x_3,y)$ as well as
	\begin{align*}
		\frac{\partial}{\partial x_1}C(x_1,x_2,x_3,y) = \underbrace{\int_{[0,x_2]} \int_{[0,x_3]} K_C(x_1,s_2,s_3,[0,y]) d\lambda(s_3)d\lambda(s_2)}_{=:F_1(x_1,x_2,x_3,y)}
	\end{align*}
	for every $(x_2,x_3) \in \mathbb{I}^2$ and $x_1 \in \Lambda^y$. \\
	Now fix $x_1 \in \Lambda^y$ and $x_3 \in \mathbb{I}$. Then the mapping $x_2 \mapsto F_1(x_1,x_2,x_3,y)$ is absolutely continuous and non-decreasing, so there exists a set $\Lambda_{(x_1,x_3)}^y$ with $\lambda(\Lambda_{(x_1,x_3)}^y) = 1$ such that the partial derivative $\frac{\partial F_1}{\partial x_2}$ exists and fulfills
	\begin{align} \label{eq:F1}
		\frac{\partial}{\partial x_2}F_1(x_1,x_2,x_3,y)= \underbrace{\int_{[0,x_3]} K_C(x_1,x_2,s_3,[0,y]) d\lambda(s_3)}_{=: F_2(x_1,x_2,x_3,y)}.
	\end{align}
	Proceeding as before and considering $\Lambda_{x_1}^y:= (0,1) \cap \bigcap_{x_3 \in \mathbb{Q} \cap \mathbb{I}} \Lambda_{(x_1,x_3)}^y$ we get that for $x_1 \in \Lambda^y$, $x_2 \in \Lambda_{x_1}^y$ and arbitrary $x_3 \in \mathbb{I}$ the partial derivative of $F_1$ with respect to $x_2$ exists and satisfies Eq. \eqref{eq:F1}. Furthermore, we find a set $\Lambda_{(x_1,x_2)}^y \in \mathcal{B}(\mathbb{I})$ with $\lambda(\Lambda_{(x_1,x_2)}^y)=1$ such that   
	\begin{align}\label{eq:F2}
		\frac{\partial}{\partial x_3}F_2(x_1,x_2,x_3,y) = K_A(x_1,x_2,x_3,[0,y])
	\end{align}
	holds for $x_1 \in \Lambda^y$, $x_2 \in \Lambda_{x_1}^y$ and $x_3 \in \Lambda_{(x_1,x_2)}^y$. \\
	In the sequel we will use the following Dini derivatives of a function $f:(0,1)^4 \to \mathbb{I}$ with
	respect to the first coordinate (see, for instance, \cite{stromberg})
	\begin{align*}
		\frac{\partial^+}{\partial x_1} f(x_1,x_2,x_3,y):= \inf_{h > 0} \sup_{\Delta \in (-h,h)\setminus\{0\}} \frac{f(x_1 + \Delta, x_2,x_3,y) - f(x_1,x_2,x_3,y)}{\Delta} ,\\ 
		\frac{\partial^-}{\partial x_1} f(x_1,x_2,x_3,y):= \sup_{h > 0} \inf_{\Delta \in (-h,h)\setminus\{0\}} \frac{f(x_1 + \Delta, x_2,x_3,y) - f(x_1,x_2,x_3,y)}{\Delta}.
	\end{align*}  
	and consider the set $\Gamma^y$ defined by
	\begin{align*}
		\Gamma^y:=\bigg\{&(x_1,x_2,x_3) \in (0,1)^3: \\
		&\frac{\partial^+}{\partial x_1} C(x_1,x_2,x_3,y) =
		\frac{\partial^-}{\partial x_1} C(x_1,x_2,x_3,y) = F_1(x_1,x_2,x_3,y), \\
		&\frac{\partial^+}{\partial x_2} F_1(x_1,x_2,x_3,y) =
		\frac{\partial^-}{\partial x_2} F_1(x_1,x_2,x_3,y) = F_2(x_1,x_2,x_3,y), \\
		&\frac{\partial^+}{\partial x_3} F_2(x_1,x_2,x_3,y) =
		\frac{\partial^-}{\partial x_3} F_2(x_1,x_2,x_3,y) = K_C(x_1,x_2,x_3,[0,y]) 
		\bigg\}.
	\end{align*}
	Note that continuity of the maps $x_1 \mapsto C(x_1,x_2,x_3,y)$, $x_2 \mapsto F_1(x_1,x_2,x_3,y)$ and $x_3 \mapsto F_2(x_1,x_2,x_3,y)$ implies measurability of the considered Dini derivatives, so $\Gamma^y \in \mathcal{B}(\mathbb{I}^3)$ follows immediately. Finally using disintegration twice and considering  
	\begin{align*}
		\lambda^3(\Gamma^y) & = \int_{\mathbb{I}} \lambda^2(\Gamma^y_{x_1}) d\lambda(x_1) = \int_{\Lambda^y} \lambda^2(\Gamma^y_{x_1}) d\lambda(x_1) \\
		&=\int_{\Lambda^y} \int_{\mathbb{I}} \lambda\left(\left(\Gamma^y_{x_1}\right)_{x_2}\right) d\lambda(x_2)d\lambda(x_1)\\
		&=\int_{\Lambda^y} \int_{\Lambda_{x_1}^y}\underbrace{ \lambda\left(\left(\Gamma^y_{x_1}\right)_{x_2}\right)}_{ \geq \lambda(\Lambda_{(x_1,x_2)}^y) = 1} d\lambda(x_2)d\lambda(x_1)
	\end{align*}
	shows $\lambda^3(\Gamma^y)=1$.\\
	Using Fubini's theorem, repeating the above procedure with a different order of the partial derivatives yields another Borel set of full $\lambda^3$-measure and the proof is complete.   
\end{proof}

\begin{proof}[Proof of Lemma \ref{lem:metric}]
	First of all we show that the integrand of $D_1$ (or $D_p$, respectively) is measurable. Define $H$ on $[0,1]^\rho$ by $H(\boldsymbol{x},y)=K_{L(A)}(\boldsymbol{x},[0,y])$. Then $H$ is measurable in $\boldsymbol{x}$ and non-decreasing and right-continuous in $y$. Fix $z \in [0,1]$. For every $q \in \mathbb{Q} \cap [0,1]$ define 
	$$\Lambda_q:=\{\boldsymbol{x} \in \mathbb{I}^d: H(\boldsymbol{x},q) < z\} \in \mathcal{B}([0,1]^d),$$
	and set 
	$$\Lambda:= \bigcup_{q \in \mathbb{Q} \cap \mathbb{I}} \Lambda_q \times [0,q] \in \mathcal{B}([0,1]^\rho).$$  
	Using right-continuity it is straightforward to see that $\Lambda=H^{-1}([0,z))$, from which measurability of $H$ directly follows. \\
	Furthermore, if $D_1(A,B) = 0$ then there exists a set $\Lambda \subseteq [0,1]^\rho$ with $\lambda^\rho(\Lambda) = 1$ such that for every $(\boldsymbol{x},y) \in \Lambda$ we have $K_{L(A)}(\boldsymbol{x},[0,y]) = K_{L(B)}(\boldsymbol{x},[0,y])$. It follows that $$\lambda(\Lambda_{\boldsymbol{x}}) = \lambda(\{y \in \mathbb{I}: (\boldsymbol{x},y) \in \Lambda\}) = 1$$ holds for almost every $\boldsymbol{x} \in \mathbb{I}^d$. For every such $\boldsymbol{x}$ we have that the kernels coincide on a dense set, so the conditional distribution functions have to be identical. Using disintegration shows $L(A) = L(B)$. Note that on the space of copulas we might have $D_1(A,B)=0$ although $A \neq B$. 
	The remaining properties of a metric are obviously fulfilled. The fact that $D_\infty$ and $D_p$ are metrics can be shown analogously. 
\end{proof}

\begin{proof}[Proof of Proposition \ref{prop:ineq}]
	We start by showing that the function $\phi_{A,B}$ defined by
	\begin{align*}
		\phi_{A,B}(y):= \int_{\mathbb{I}^d} \lvert K_{L(A)}(\boldsymbol{x},[0,y]) - K_{L(B)}(\boldsymbol{x},[0,y])\rvert d\lambda^d(\boldsymbol{x}) 
	\end{align*}
	is Lipschitz continuous with Lipschitz constant $L=2$. In fact, if $s < t$, defining $G$ by 
	$$G:= \{\boldsymbol{x} \in [0,1]^d: K_{L(A)}(\boldsymbol{x},(s,t]) > K_{L(B)}(\boldsymbol{x},(s,t])\}$$ 
	and using Scheff\'e's theorem (see \cite{devroye1987course}), we have
	\begin{align*}
		\lvert \phi_{A,B}(s) - \phi_{A,B}(t) \rvert &\leq \int_{\mathbb{I}^d} \lvert K_{L(A)}(\boldsymbol{x},(s,t]) - K_{L(B)}(\boldsymbol{x},(s,t]) \rvert d\lambda^d(\boldsymbol{x}) \\
		&= 2 \int_G K_{L(A)}(\boldsymbol{x},(s,t]) - K_{L(B)}(\boldsymbol{x},(s,t]) d\lambda^d(\boldsymbol{x}) \\
		&\leq 2\int_{\mathbb{I}^d} K_{L(A)}(\boldsymbol{x}, (s,t]) \, d\lambda^d(\boldsymbol{x})
		\leq 2 \lambda((s,t]) = 2(t-s),
	\end{align*} 
	which shows Lipschitz continuity. \\	
	Statement (1) follows directly from disintegration and the triangle inequality since
	\begin{align*}
		d_\infty(L(A),L(B)) &= \sup_{(\boldsymbol{x},y) \in \mathbb{I}^\rho} \lvert L(A)(\boldsymbol{x},y) - L(B)(\boldsymbol{x},y) \rvert \\
		&=\sup_{(\boldsymbol{x},y) \in \mathbb{I}^\rho} \left| \int_{[0,\boldsymbol{x}]} K_{L(A)}(\boldsymbol{s},[0,y]) - K_{L(B)}(\boldsymbol{s},[0,y]) d\lambda^d(\boldsymbol{s}) \right| \\
		&\leq \sup_{y \in \mathbb{I}} \int_{\mathbb{I}^d}  \left| K_{L(A)}(\boldsymbol{s},[0,y]) - K_{L(B)}(\boldsymbol{s},[0,y]) \right| d\lambda^d(\boldsymbol{s}) \\
		&= D_\infty(A,B).
	\end{align*}
	The first inequality in the second assertion is obvious. Using Lipschitz continuity we can find some $y_0 \in [0,1]$ such that $\phi_{A,B}(y_0) = \sup_{y \in [0,1]} \phi_{A,B}(y)$. Furthermore the area between the graph of $\phi_{A,B}$ and the $x$-axis contains the triangle $\Delta_1$ with vertices $\{(y_0-\tfrac{\phi_{A,B}(y_0)}{2}, 0), (y_0,0), (y_0, \phi_{A,B}(y_0))\}$ or the triangle $\Delta_2$ with vertices $\{(y_0,0),(y_0+\tfrac{\phi_{A,B}(y_0)}{2}, 0), (y_0, \phi_{A,B}(y_0))\}$. Hence, we have
	\begin{align*}
		\int_{[0,1]} \phi_{A,B}(y) d\lambda(y) \geq \frac{\phi_{A,B}(y_0) \frac{\phi_{A,B}(y_0)}{2}}{2} = \frac{D_\infty(A,B)^2}{4},
	\end{align*}  
	which proves assertion (2). The first inequality in the third assertion is trivial since the integrand only attains values in $[0,1]$, the second one follows from Hölder's inequality.
\end{proof}

\begin{proof}[Proof of Theorem \ref{thm:CBdense}]
	Since $C \in \mathcal{C}_{\Pi_d}^\rho$ it is straightforward to verify that $\mathfrak{CB}_N(C)\in \mathcal{C}_{\Pi_d}^\rho$ holds 
	for eve\-ry $N \in \mathbb{N}$. Fix $y \in (0,1)$, let $J_N(y)$ denote the unique interval of the form $J_N(y):=\left(\frac{j(y)-1}{N}, \frac{j(y)}{N}\right]$ containing $y$, and let $Q$ be a countable dense subset of $(0,1)$. For every $q \in Q$ we find a set $\Lambda_q$ with $\lambda^d(\Lambda_q)=1$ such that for every $\boldsymbol{x} \in \Lambda_q$ the point $\boldsymbol{x}$ is a Lebesgue point of the mappings $(\boldsymbol{x} \mapsto K_C(\boldsymbol{x},[0,q]))$ and $(\boldsymbol{x} \mapsto K_C(\boldsymbol{x},(q-\delta, q + \delta)))$ for every $\delta \in \mathbb{Q} \cap (0,1)$. 
	Setting $\Lambda_0 := \bigcap_{q \in Q} \Lambda_q$ and defining $E_q$ by
	\begin{align*}
		E_q:=\{\boldsymbol{x} \in \Lambda_0: K_C(\boldsymbol{x},\{q\})=0\}
	\end{align*}
	for every $q \in Q$, both $\Lambda_0$ and $E_q$ are of $\lambda^d$-measure $1$. 
	Considering $\Lambda:= (\bigcap_{q \in Q} E_q)\setminus \mathbb{Q}^d$ therefore yields 
	$\lambda^d(\Lambda) = 1$. \\
	
	\noindent As first major step we now show that  
	\begin{align*}
		\lim_{N \to \infty} K_{\mathfrak{CB}_N(C)}\left(\boldsymbol{x}, J_N(y)\right) = 0
	\end{align*}
	holds for every $(\boldsymbol{x},y) \in \Lambda \times Q$. Fix $(\boldsymbol{x},y) \in \Lambda \times Q$ and $\varepsilon>0$. Then there exists some $\delta \in \mathbb{Q} \cap (0,1]$ such that $K_C(\boldsymbol{x}, (y-\delta, y+\delta)) < \varepsilon$. Furthermore, there exists an $N_0 \in \mathbb{N}$ such that for all $N \geq N_0$ we have $J_N(y) \subseteq (y-\delta, y + \delta)$. 
	Let $(R_N^{\boldsymbol{i}}(\boldsymbol{x}))_{N \in \mathbb{N}}$ denote the unique sequence of hypercubes (according to equation (\ref{hypercubes})) containing $\boldsymbol{x}$. It is straightforward to verify that $(R_N^{\boldsymbol{i}}(\boldsymbol{x}))_{N \in \mathbb{N}}$ shrinks nicely to $\boldsymbol{x}$ with respect to $\lambda^d$ (see \cite{rudin}). Using the fact that the probability measure $K_{\mathfrak{CB}_N(C)}(\boldsymbol{x}, \cdot)$ is constant on each hypercube and 
	$$\mu_{\mathfrak{CB}_N(C)}(R_N^{\boldsymbol{i}}(\boldsymbol{x}) \times J_N(y)) = \mu_{C}(R_N^{\boldsymbol{i}}(\boldsymbol{x})\times J_N(y))$$ 
	holds for every $N \in \mathbb{N}$ yields 
	\begin{align*}
		K_{\mathfrak{CB}_N(C)}\big(\boldsymbol{x}, J_N(y)\big) &= \frac{1}{\lambda^d(R_N^{\boldsymbol{i}}(\boldsymbol{x}))} \int_{R_N^{\boldsymbol{i}}(\boldsymbol{x})} K_{\mathfrak{CB}_N(C)}\big(\boldsymbol{s}, J_N(y)\big) d\lambda^d(\boldsymbol{s})\\
		&= \frac{1}{\lambda^d(R_N^{\boldsymbol{i}}(\boldsymbol{x}))} \mu_{\mathfrak{CB}_N(C)}\left(R_N^{\boldsymbol{i}}(\boldsymbol{x}) \times J_N(y)\right)\\
		&= \frac{1}{\lambda^d(R_N^{\boldsymbol{i}}(\boldsymbol{x}))} \int_{R_N^{\boldsymbol{i}}(\boldsymbol{x})}K_{C}\big(\boldsymbol{s}, J_N(y)\big) d\lambda^d(\boldsymbol{s})\\
		&\leq \frac{1}{\lambda^d(R_N^{\boldsymbol{i}}(\boldsymbol{x}))} \int_{R_N^{\boldsymbol{i}}(\boldsymbol{x})} K_{C}\big(\boldsymbol{s}, (y-\delta, y+ \delta)\big) d\lambda^d(\boldsymbol{s}).
	\end{align*} 
	Applying Lebesgue's differentiation theorem (see \cite[Theorem 7.10]{rudin}) we obtain
	\begin{align*}
		\limsup_{N \to \infty} K_{\mathfrak{CB}_N(C)}\big(&\boldsymbol{x}, J_N(y)\big) = 
		\limsup_{N \to \infty}\frac{1}{\lambda^d(R_N^{\boldsymbol{i}}(\boldsymbol{x}))} \int_{R_N^{\boldsymbol{i}}(\boldsymbol{x})} K_{C}\big(\boldsymbol{s}, J_N(y)\big) d\lambda^d(\boldsymbol{s})\\
		&\leq  \limsup_{N \to \infty}\frac{1}{\lambda^d(R_N^{\boldsymbol{i}}(\boldsymbol{x}))} \int_{R_N^{\boldsymbol{i}}(\boldsymbol{x})} K_{C}\big(\boldsymbol{s}, (y-\delta, y+ \delta)\big) d\lambda^d(\boldsymbol{s})\\
		&= K_{C}\big(\boldsymbol{x}, (y-\delta, y+ \delta)\big) < \varepsilon.
	\end{align*}
	Since $\varepsilon$ was arbitrary we have shown that $\lim_{N \to \infty} K_{\mathfrak{CB}_N(C)}\left(\boldsymbol{x}, J_N(y)\right) = 0$. \\
	
	\noindent Let $(\boldsymbol{x}, y)$ be as before. Then using the afore-mentioned property of 
	$\mathfrak{CB}_N(C)$ again yields  
	\begin{align*}
		K_{\mathfrak{CB}_N(C)} (\boldsymbol{x}, [0,y])  
		&= \frac{1}{\lambda^d(R_N^{\boldsymbol{i}}(\boldsymbol{x}))} \int_{R_N^{\boldsymbol{i}}(\boldsymbol{x})} K_{\mathfrak{CB}_N(C)}(\boldsymbol{s}, [0,y]) \, d\lambda^d(\boldsymbol{s})\\
		&=  \frac{1}{\lambda^d(R_N^{\boldsymbol{i}}(\boldsymbol{x}))} \int_{R_N^{\boldsymbol{i}}(\boldsymbol{x})} K_{C}(\boldsymbol{s},[0,y]) \, d\lambda^d(\boldsymbol{s})\\
		&\quad +
		\frac{1}{\lambda^d(R_N^{\boldsymbol{i}}(\boldsymbol{x}))} \int_{R_N^{\boldsymbol{i}}(\boldsymbol{x})} K_{\mathfrak{CB}_N(C)}(\boldsymbol{s},[0,y]) - K_{C}(\boldsymbol{s},[0,y]) \, d\lambda^d(\boldsymbol{s}) \\
		&\leq \frac{1}{\lambda^d(R_N^{\boldsymbol{i}}(\boldsymbol{x}))} \int_{R_N^{\boldsymbol{i}}(\boldsymbol{x})} K_{C}(\boldsymbol{s},[0,y]) \, d\lambda^d(\boldsymbol{s})\\
		&\quad +
		\frac{1}{\lambda^d(R_N^{\boldsymbol{i}}(\boldsymbol{x}))} \int_{R_N^{\boldsymbol{i}}(\boldsymbol{x})} K_{\mathfrak{CB}_N(C)}\left(\boldsymbol{s},\left[0,\tfrac{j(y)}{N}\right]\right) \, d\lambda^d(\boldsymbol{s})\\
		&\quad - \frac{1}{\lambda^d(R_N^{\boldsymbol{i}}(\boldsymbol{x}))}\int_{R_N^{\boldsymbol{i}}(\boldsymbol{x})} K_{C}\left(\boldsymbol{s},\left[0,\tfrac{j(y)-1}{N}\right]\right) \, d\lambda^d(\boldsymbol{s}) \\
		&=\frac{1}{\lambda^d(R_N^{\boldsymbol{i}}(\boldsymbol{x}))} \int_{R_N^{\boldsymbol{i}}(\boldsymbol{x})} K_{C}(\boldsymbol{s},[0,y]) \, d\lambda^d(\boldsymbol{s}) \\
		&\quad + \frac{1}{\lambda^d(R_N^{\boldsymbol{i}}(\boldsymbol{x}))} \int_{R_N^{\boldsymbol{i}}(\boldsymbol{x})} K_{\mathfrak{CB}_N(C)}(\boldsymbol{s}, J_N(y)) \, d\lambda^d(\boldsymbol{s}) \\
		&=\frac{1}{\lambda^d(R_N^{\boldsymbol{i}}(\boldsymbol{x}))} \int_{R_N^{\boldsymbol{i}}(\boldsymbol{x})} K_{C}(\boldsymbol{s},[0,y]) \, d\lambda^d(\boldsymbol{s}) \\
		&\quad + K_{\mathfrak{CB}_N(C)}(\boldsymbol{x}, J_N(y)).
	\end{align*}
	Applying Lebesgue's differentiation theorem and the results of the first part of the proof yields
	\begin{align*}
		\limsup_{N \to \infty} K_{\mathfrak{CB}_N(C)} (\boldsymbol{x}, [0,y])  \leq K_{C} (\boldsymbol{x}, [0,y]).
	\end{align*} 
	Proceeding in a similar manner we also get
	\begin{align*}
		K_{\mathfrak{CB}_N(C)} (\boldsymbol{x}, [0,y])  &\geq  \frac{1}{\lambda^d(R_N^{\boldsymbol{i}}(\boldsymbol{x}))} \int_{R_N^{\boldsymbol{i}}(\boldsymbol{x})} K_{C}(\boldsymbol{s},[0,y]) \, d\lambda^d(\boldsymbol{s})\\
		&\quad +
		\frac{1}{\lambda^d(R_N^{\boldsymbol{i}}(\boldsymbol{x}))} \int_{R_N^{\boldsymbol{i}}(\boldsymbol{x})} K_{\mathfrak{CB}_N(C)}\left(\boldsymbol{s},\left[0,\tfrac{j(y)-1}{N}\right]\right) \, d\lambda^d(\boldsymbol{s})\\
		&\quad -  \frac{1}{\lambda^d(R_N^{\boldsymbol{i}}(\boldsymbol{x}))}\int_{R_N^{\boldsymbol{i}}(\boldsymbol{x})} K_{C}\left(\boldsymbol{s},\left[0,\tfrac{j(y)}{N}\right]\right) \, d\lambda^d(\boldsymbol{s}) \\
		&=\frac{1}{\lambda^d(R_N^{\boldsymbol{i}}(\boldsymbol{x}))} \int_{R_N^{\boldsymbol{i}}(\boldsymbol{x})} K_{C}(\boldsymbol{s},[0,y]) \, d\lambda^d(\boldsymbol{s}) \\
		&\quad - K_{\mathfrak{CB}_N(C)}(\boldsymbol{x}, J_N(y)),
	\end{align*}
	implying
	\begin{align*}
		\liminf_{N \to \infty} K_{\mathfrak{CB}_N(C)} (\boldsymbol{x}, [0,y])  \geq K_{C} (\boldsymbol{x}, [0,y]).
	\end{align*} 
	Altogether we have shown  
	\begin{align*}
		\lim_{N \to \infty}  K_{\mathfrak{CB}_N(C)} (\boldsymbol{x}, [0,y]) = K_C(\boldsymbol{x}, [0,y]).
	\end{align*}
	Since weak convergence of univariate distribution functions $F_1,F_2,\ldots$ to $F$ is equivalent to pointwise convergence on a dense subset (see \cite{billingsley1968}), we have shown that $\lambda^d$-almost all conditional distribution functions $(y \mapsto K_{\mathfrak{CB}_N(C)}(\boldsymbol{x}, [0,y]))$ converge weakly to $(y \mapsto K_C(\boldsymbol{x}, [0,y]))$.
\end{proof}

\section{Simulation study} \label{app:simstudy}
In order to illustrate the small/moderate sample performance as well as the convergence speed of our estimator $\hat{\zeta}_n^1$ we consider several dependence structures ranging from independence to complete dependence. If not specified otherwise, we considered $s \in \left\{\tfrac{1}{\rho}, \tfrac{1}{2d}\right\}$. Furthermore, the minimum resolution of the checkerboard aggregation was set to $N=2$. For the extreme cases of independence and complete dependence we compared our estimator with the results obtained by the \textquoteleft simple measure of conditional dependence $T_n$' (using the function \textquoteleft codec' in the R-package \textquoteleft FOCI', see \cite{foci}).

\subsection{Independence}
To test the performance of $\hat{\zeta}_n^1$ in the setting of $Y$ and $\boldsymbol{X}$ being independent
for different dependence structures of $\boldsymbol{X}$ we considered $d \in \{2,3,4\}$, generated samples 
$(\boldsymbol{x}_1,y_1),\ldots,(\boldsymbol{x}_n,y_n)$
of sizes 
$n \in \{100, 500, 1.000, 5.000, 10.000 \}$ and calculated $\hat{\zeta}^1_n$ as well as $T_n$. These steps were repeated $R = 1.000$ times, the obtained results are depicted as boxplots in Figs. \ref{fig:indep1}, \ref{fig:indep2} and \ref{fig:indep3}. Obviously, $\hat{\zeta}_n^1$ only attains positive values and tends to $0$ for increasing sample size, whereas $T_n$ varies strongly around $0$, i.e., also attains large negative values. Hence, interpreting values of $\hat{\zeta}_n^1$ between $0$ and $0.3$ must be done with care and always under consideration of the sample size $n$. 
The simulations insinuate that $\hat{\zeta}_n^1$ exhibits slightly smaller variance. 

\begin{figure}[!ht]
	\centering
	\includegraphics[width=10.5cm, page = 1]{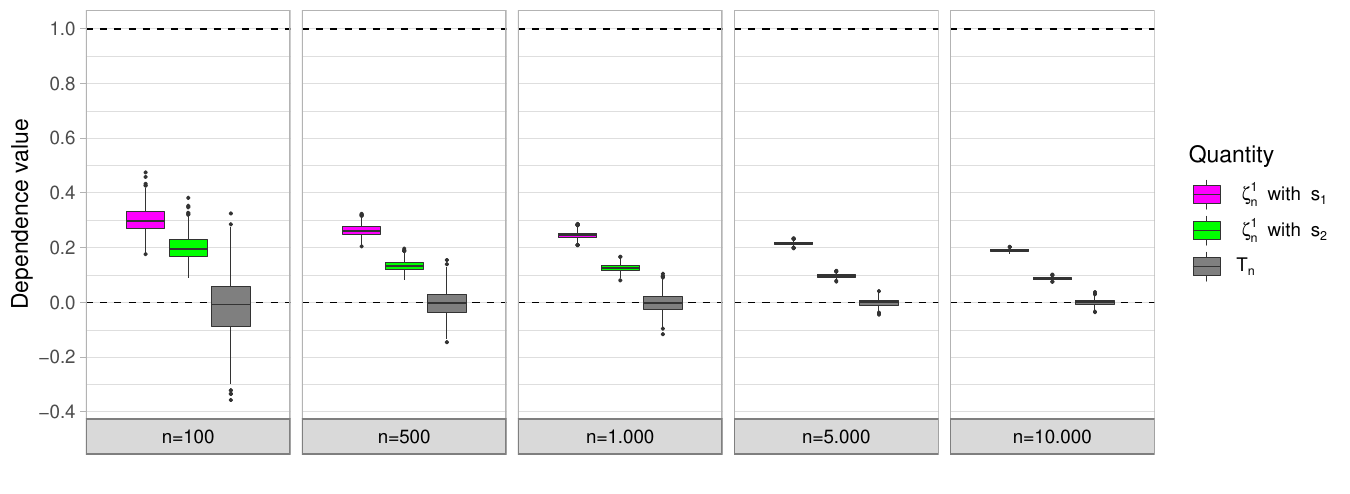}
	\caption{Boxplots summarizing the $1.000$ obtained estimates for $\hat{\zeta}_n^1((X_1,X_2),Y)$ for two different choices of $s$ ($s_1 = 1/3$ (magenta) and $s_2 = 1/4$ (green)) as well as for $T_n(Y,(X_1,X_2))$ (gray). Samples of size $n$ are drawn from normally distributed as well as exponentially distributed random variables: $X_1 \sim \mathcal{N}(0,1)$, $X_2 \sim \mathcal{N}(0,1)$, $Y \sim \mathcal{E}(1)$; $X_1,X_2,Y$ independent.}
	\label{fig:indep1}
\end{figure}
\begin{figure}[!ht]
	\centering
	\includegraphics[width=10.5cm, page = 2]{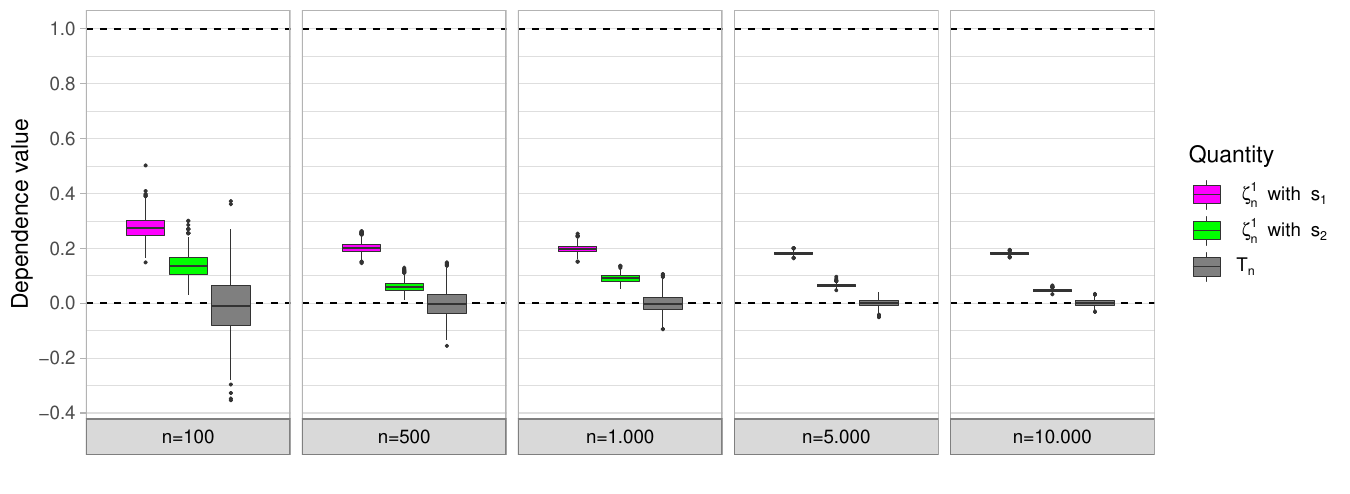}
	\caption{Boxplots summarizing the $1.000$ obtained estimates for $\hat{\zeta}_n^1((X_1,X_2,X_3),Y)$ for two different choices of $s$ ($s_1 = 1/4$ (magenta) and $s_2 = 1/6$ (green)) as well as for $T_n(Y,(X_1,X_2,X_3))$ (gray). Samples of size $n$ are drawn from: $X_1 \sim \mathcal{U}(0,1)$, $X_2:= X_1 + \varepsilon$ with $\varepsilon \sim \mathcal{N}(0,0.1^2)$, $X_3 \sim \mathcal{E}(1)$ and $Y \sim \mathcal{U}(0,1)$; $\boldsymbol{X}$ and $Y$ are independent.}
	\label{fig:indep2}
\end{figure}
\begin{figure}[!ht]
	\centering
	\includegraphics[width=10.5cm, page = 3]{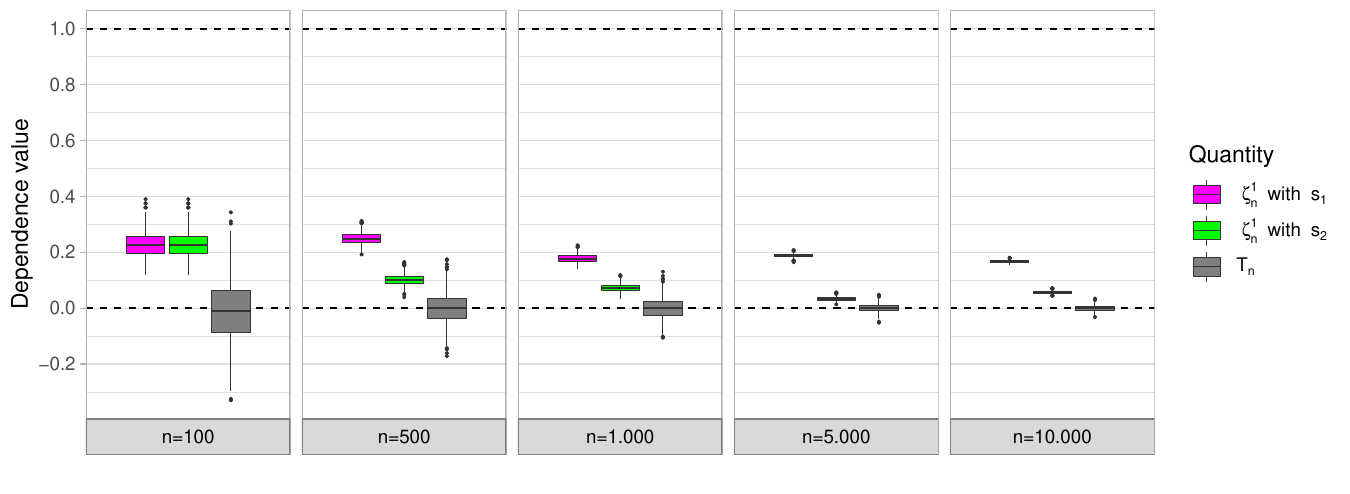}
	\caption{Boxplots summarizing the $1.000$ obtained estimates for $\hat{\zeta}_n^1((X_1,X_2,X_3,X_4),Y)$ for two different choices of $s$ ($s_1 = 1/5$ (magenta) and $s_2 = 1/8$ (green)) as well as for $T_n(Y,(X_1,X_2,X_3,X_4))$ (gray). Samples of size $n$ are drawn from: $X_1 \sim \mathcal{U}(0,1)$, $X_2:= 2X_1 (mod1) + \varepsilon_1$, $X_3 \sim \mathcal{U}(0,1)$, $X_4:= 2X_3 (mod1) + \varepsilon_2$, whereby $\varepsilon_1,\varepsilon_2 \sim \mathcal{N}(0,0.1^2)$ and $Y \sim \mathcal{U}(0,1)$; $\boldsymbol{X},Y$ independent.}
	\label{fig:indep3}
\end{figure}

\subsection{$C_{Cube}$}
To test the performance of $\hat{\zeta}_n^1$ for $C_{Cube} \in \mathcal{C}^3$ according to Example \ref{exa:cube}, we generated samples of size $n \in \{100, 500, 1.000, 5.000, 10.000 \}$ and calculated $\hat{\zeta}^1_n$. These steps were repeated $R = 1.000$ times, the obtained results are depicted as boxplots in Figs. \ref{fig:cube_sim1}. Obviously, $\hat{\zeta}_n^1$ converges to the true value from below. Not surprisingly the speed of convergence of 
$\hat{\zeta}_n^1$ strongly depends on the resolution $N$, defined by $N(n):= \lfloor n^s \rfloor$. 
\begin{figure}[!ht]
	\centering
	\includegraphics[width=10.5cm, page = 1]{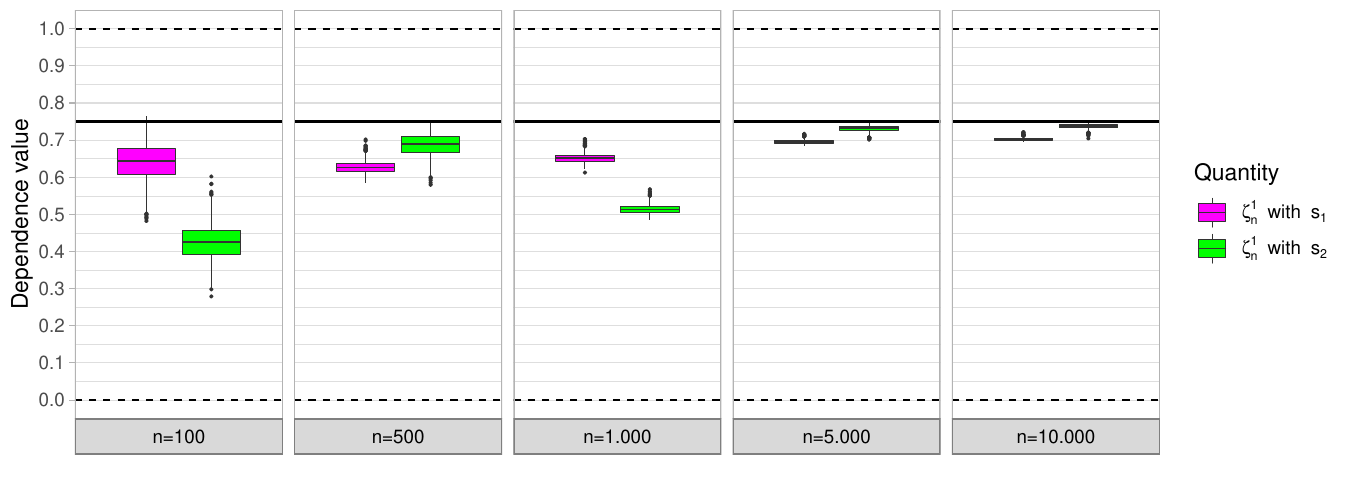}
	\caption{Boxplots summarizing the $1.000$ obtained estimates for $\hat{\zeta}_n^1(C_{Cube})$ for two different choices of $s$ ($s_1 = 1/3$ (magenta) and $s_2 = 1/4$ (green)). The true value of $\zeta^1(C_{Cube})=0.75$ is depicted as black horizontal line.}
	\label{fig:cube_sim1}
\end{figure}

\subsection{Complete (functional) dependence}
In order to test the performance of $\hat{\zeta}_n^1$ for the opposite extreme of complete dependence, we 
generated samples of size $n \in \{100, 500, 1.000, 5.000, 10.000 \}$ and calculated $\hat{\zeta}^1_n$ as well as $T_n$ for several different functional dependence structures. These steps were repeated $R = 1.000$ times. The results for some specific dependence structures in the three-, four- and five-dimensional setting are depicted in Figs. \ref{fig:dep1}, \ref{fig:dep2}, \ref{fig:dep3} and \ref{fig:dep4}. It can be seen that the convergence speed of $\hat{\zeta}_n^1$ is 
the better the larger the parameter $s$ (or the lower the dimension).        
\begin{figure}[!ht]
	\centering
	\includegraphics[width=10.5cm, page = 1]{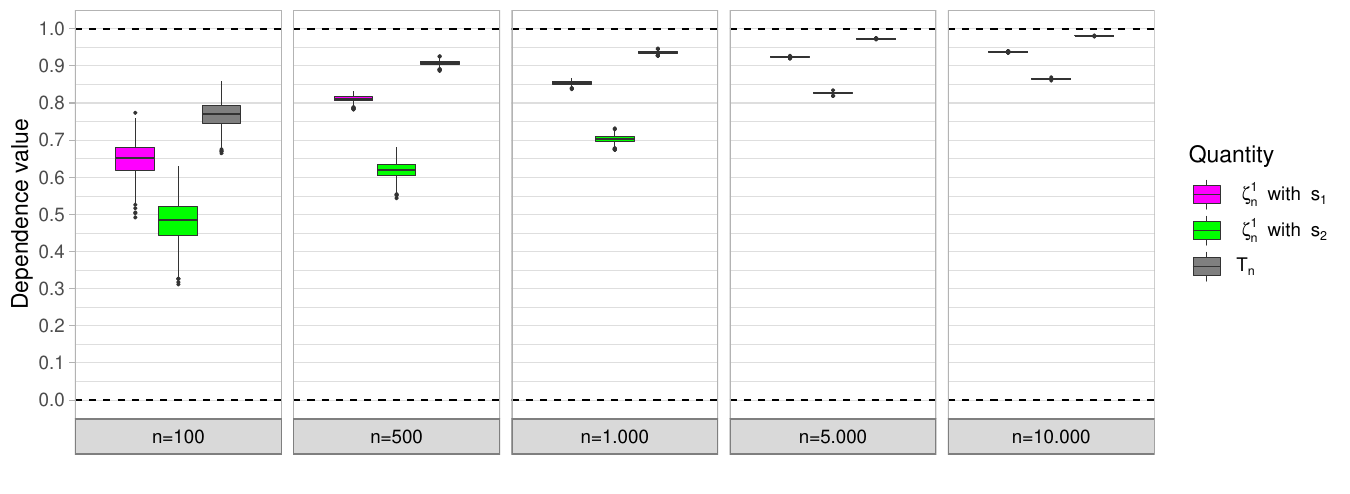}
	\caption{Boxplots summarizing the $1.000$ obtained estimates for $\hat{\zeta}_n^1((X_1,X_2),Y)$ for two different choices of $s$ ($s_1 = 1/3$ (magenta) and $s_2 = 1/4$ (green)) as well as for $T_n(Y,(X_1,X_2))$ (gray). Samples of size $n$ are drawn from: $X_1 \sim \mathcal{U}(-1,1)$, $X_2 \sim \mathcal{U}(-1,1)$ and $Y:=X_1^2 + X_2^2$.}
	\label{fig:dep1}
\end{figure}

\begin{figure}[!ht]
	\centering
	\includegraphics[width=10.5cm, page = 2]{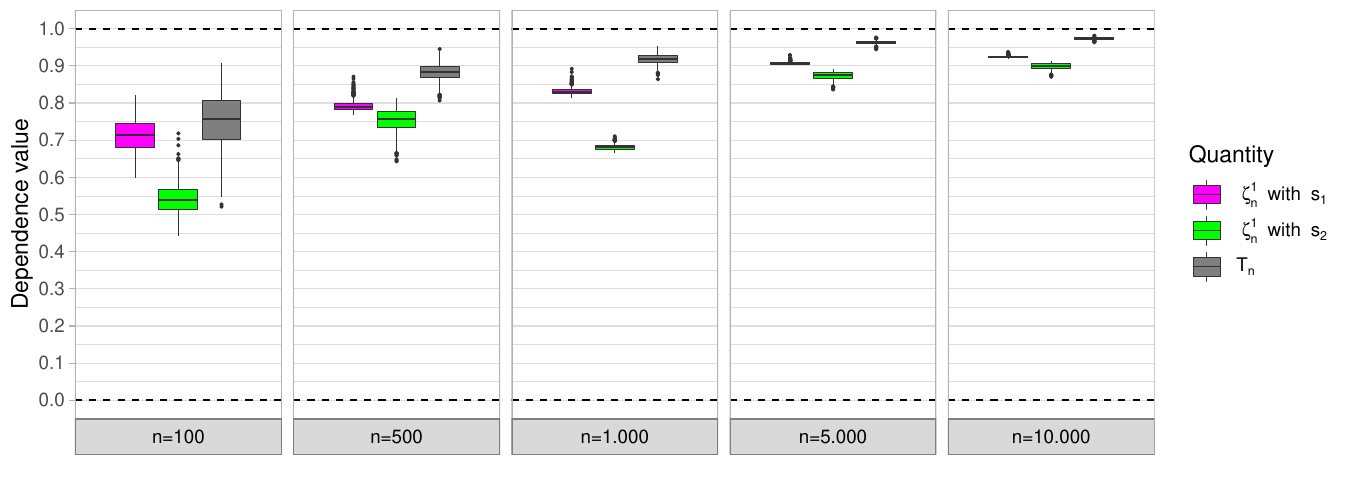}
	\caption{Boxplots summarizing the $1.000$ obtained estimates for $\hat{\zeta}_n^1((X_1,X_2),Y)$ for two different choices of $s$ ($s_1 = 1/3$ (magenta) and $s_2 = 1/4$ (green)) as well as for $T_n(Y,(X_1,X_2))$ (gray). Samples of size $n$ are drawn from: $X_1 \sim \mathcal{N}(0,1)$, $X_2 \sim \mathcal{N}(0,1)$ and $Y:=X_1/X_2$.}
	\label{fig:dep2}
\end{figure}

\begin{figure}[!ht]
	\centering
	\includegraphics[width=10.5cm, page = 3]{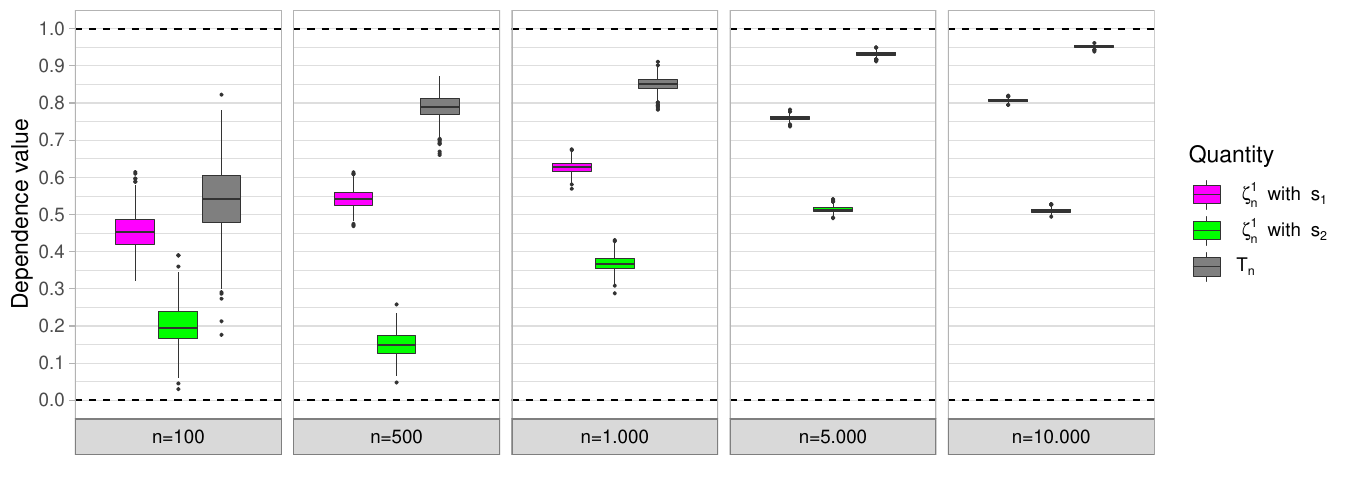}
	\caption{Boxplots summarizing the $1.000$ obtained estimates for $\hat{\zeta}_n^1((X_1,X_2,X_3),Y)$ for two different choices of $s$ ($s_1 = 1/4$ (magenta) and $s_2 = 1/6$ (green)) as well as for $T_n(Y,(X_1,X_2,X_3))$ (gray). Samples of size $n$ are drawn from: $X_1 \sim \mathcal{U}(0,1)$, $X_2:=2X_1 (mod1)$, $X_3 \sim \mathcal{U}(0,1)$ and $Y:=X_1 + X_2 + X_3 (mod1)$.}
	\label{fig:dep3}
\end{figure}

\begin{figure}[!ht]
	\centering
	\includegraphics[width=11cm, page = 4]{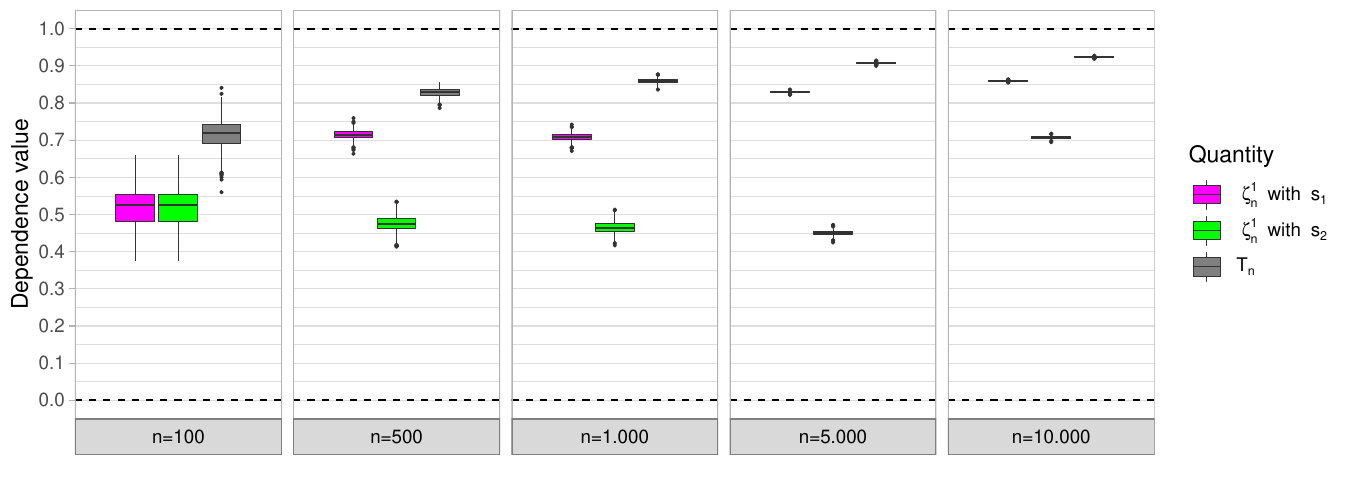}
	\caption{Boxplots summarizing the $1.000$ obtained estimates for $\hat{\zeta}_n^1((X_1,X_2,X_3,X_4),Y)$ for two different choices of $s$ ($s_1 = 1/5$ (magenta) and $s_2 = 1/8$ (green)) as well as for $T_n(Y,(X_1,X_2,X_3,X_4))$ (gray). Samples of size $n$ are drawn from: $X_1 \sim \mathcal{U}(0,1)$, $X_2 \sim \mathcal{U}(0,1)$, $X_3 \sim \mathcal{U}(0,1)$, $X_4 \sim \mathcal{U}(0,1)$ and $Y:=X_1 + X_2 + X_3 +X_4$.}
	\label{fig:dep4}
\end{figure}

\section*{Acknowledgement}
	\noindent The first and the second author gratefully acknowledge the support of the Austrian FWF START project Y1102 `Successional Generation of Functional Multidiversity'. Moreover, the third author gratefully acknowledges the support of the WISS 2025 project \textquoteleft IDA-lab Salzburg' (20204-WISS/225/197-2019 and 0102-F1901166-KZP).


\end{document}